\documentclass[reqno]{amsart}
\usepackage[T1]{fontenc}
\usepackage{amssymb}
\usepackage[left=1in,right=1in,top=1in,bottom=1in]{geometry}
\usepackage[ddmmyyyy]{datetime}
\usepackage{enumitem}
\usepackage{physics}
\usepackage[normalem]{ulem}
\useunder{\uline}{\ul}{}
\usepackage[foot]{amsaddr}
\usepackage{bm}
\usepackage{tikz}
\usepackage{mathrsfs}
\usepackage{mathtools}
\usepackage[hidelinks]{hyperref}
\usepackage{scalerel,stackengine}
\stackMath
\usetikzlibrary{arrows.meta}
\usepackage{subcaption}
\usepackage[style=numeric]{biblatex}
\newenvironment{mycases}
   {\begin{dcases*}}
   {\end{dcases*}}
\theoremstyle{plain}
\newtheorem{theorem}{Theorem}[section]
\newtheorem{lemma}[theorem]{Lemma}
\newtheorem{definition}[theorem]{Definition}
\newtheorem{proposition}[theorem]{Proposition}
\newtheorem{corollary}[theorem]{Corollary}
\newtheorem*{corollary*}{Corollary}
\theoremstyle{remark}
\newtheorem{remark}[theorem]{Remark}
\counterwithout{figure}{section}
\counterwithout{equation}{section}

\newcommand{\br}[1]{\left(#1\right)}
\newcommand{\tpd}{\frac{1}{(2\pi)^{d}}}
\newcommand\pig[1]{\scalerel*[2.35pt]{\big#1}{%
  \ensurestackMath{\addstackgap[1.2pt]{\big#1}}}}
\newcommand\pigl[1]{\mathopen{\pig{#1}}}
\newcommand\pigr[1]{\mathclose{\pig{#1}}}
\newcommand{\ak}{\abs{k}}
\newcommand{\ap}{\abs{p}}
\newcommand{\aal}{\abs{\alpha}}
\newcommand{\sbr}[1]{\left[#1\right]}
\newcommand{\half}{\frac{1}{2}}
\newcommand{\p}{\partial}
\newcommand{\jb}[1]{\langle#1\rangle}
\newcommand{\set}[1]{\left\{#1\right\}}
\newcommand{\mytextfrac}[2]{\textstyle\frac{#1}{#2}}
\newcommand{\qaq}{\quad\text{and}\quad}
\newcommand{\para}{\parallel}
\newcommand{\ep}{\varepsilon}
\newcommand{\R}{\mathbb{R}}
\newcommand{\C}{\mathbb{C}}

\newcommand{\intRd}{\int_{\R^d}}
\newcommand{\intRdk}{\int_{\R^d_k}}

\newcommand{\lam}{\lambda}
\newcommand{\ceils}[1]{\lceil#1\rceil}
\newcommand{\Qin}{Q_{\mathrm{in}}}
\newcommand{\hin}{h_{\mathrm{in}}}
\newcommand\reallywidehat[1]{%
\savestack{\tmpbox}{\stretchto{%
  \scaleto{%
    \scalerel*[\widthof{\ensuremath{#1}}]{\kern.pt\mathchar"0362\kern.1pt}%
    {\rule{0ex}{\textheight}}
  }{\textheight}%
}{2.4ex}}%
\stackon[-6.9pt]{#1}{\tmpbox}%
}

\numberwithin{equation}{section}

\title{Phase mixing for the Hartree equation and Landau damping in the semiclassical limit}
\author{Marnie Smith}
\address{Department of Pure Mathematics and Mathematical Statistics, University of Cambridge}
\email{ms2724@cam.ac.uk}
\begin{document}
\begin{abstract}The asymptotic behaviour of the Hartree equation is studied near translation-invariant steady states. For short-range interaction kernels satisfying a uniform Penrose stability condition, including the screened Coulomb interaction, phase-mixing estimates in finite regularity are established. These demonstrate density decay and scattering of solutions in weighted quantum Sobolev spaces, providing a quantum analogue of Landau damping in classical plasma physics. The results hold uniformly in the semiclassical limit, thereby bridging the quantum and classical regimes.
\end{abstract}
\date{\today}
\maketitle

\tableofcontents
\section{Introduction}

\subsection{The Hartree equation}

The \textit{Hartree equation}, a mean-field model describing the dynamics of infinitely many interacting fermions, is expressed as:
\begin{align}\label{H}
    \begin{mycases}
        i \hbar\p_t\gamma = \sbr{-\frac{\hbar^2}{2}\Delta + w*\rho_{\gamma}, \gamma}, \\
        \gamma(0) = \gamma_{\mathrm{in}},
    \end{mycases}
\end{align}
where $\gamma(t)$ is a non-negative, self-adjoint operator on $L^2(\R^d; \C)$. The function $\rho_{\gamma}(t,x) = \gamma(t,x,x)$ represents the density, with $\gamma(t,x,y)$ denoting the integral kernel of $\gamma(t)$. The nonlinear convolution term $w*\rho_\gamma$ describes the mean-field interaction between the particles, where $w\colon \R^d \to \R$ is a specified two-body interaction kernel. The parameter \(\hbar>0\) represents the Planck constant, and the semiclassical limit involves taking \(\hbar \to 0\).

For any nonnegative function $g\in L^\infty(\R^d)\cap L^1(\R^d)$, define the translation-invariant operator
\[
    \gamma_g := (2\pi\hbar)^d g(-i\hbar\nabla)
\]
on \( L^2(\R^d; \C) \). The integral kernel of \(\gamma_g\) is
\[
    \gamma_g(x,y) =  \widehat g\br{\frac{y-x}{\hbar}},
\]
and its constant density is given by
\[
    \rho_{\gamma_g} = \widehat g(0) = \int_{\R^d} g(k) \,dk.
\]
These operators represent a class of steady-state solutions to the Hartree equation \eqref{H}, as they satisfy the commutation relations \([-\Delta, \gamma_g] = 0\) and $[w*\rho_{\gamma_g}, \gamma_g]=0$, where $w*\rho_{\gamma_g}$ is a constant multiple of the identity operator. The trace of a solution to \eqref{H} determines the total number of particles. For the steady states \(\gamma_g\) with \(g\not\equiv0\), this number is infinite due to the constant positive density: \(\Tr\gamma_g = \int_{\R^d} \rho_{\gamma_g} \,dx = +\infty.\)

The interaction kernel and steady state are assumed to satisfy a linear stability criterion, known in kinetic theory as the \textit{Penrose condition}. In the present setting, the condition must hold uniformly for all $\hbar \in (0,1]$ in order to remain valid in the semiclassical limit $\hbar \to 0$. This motivates the following formulation:

\begin{definition}[Uniform Penrose condition]
An interaction kernel \(w\colon \R^d \to \R\) and steady state \(\gamma_g = (2\pi\hbar)^d g(-i\hbar\nabla)\) satisfy the uniform Penrose condition if there exists a constant \({c_{\mathrm{P}}} > 0\) such that the inequality
\begin{align}\label{Penrose}
    \inf_{\substack{\Re\lambda \geq 0 \\ k \in \R^d}} \abs{1 + \widetilde{\mathscr{L}}(\lambda, k)} \geq {c_{\mathrm{P}}}
\end{align}
holds uniformly for all \(\hbar \in (0,1]\), where
\[
    \widetilde{\mathscr{L}}(\lambda, k) = \frac{2}{\hbar} \widehat{w}(k) \int_0^\infty e^{-\lam t}\sin\br{\frac{1}{2}\hbar t \ak^2} \widehat{g}\br{kt}\,dt.
\]

\end{definition}
Sufficient conditions on $w$ and $g$ ensuring the uniform Penrose condition~\eqref{Penrose} are established in Section~\ref{Penroseconditionsection}. These include the case of a radial, positive steady state and a repulsive interaction, corresponding to the defocusing regime. Specifically, the condition~\eqref{Penrose} holds if \( w \in L^1(\R^d) \) satisfies \( \widehat{w}(k) \geq 0 \) and \( g \in H^{3/2+\delta_0}_{1+\lceil d/2 \rceil}(\R^d) \) is radial and positive-valued, for some \( 0 < \delta_0 \ll 1 \).

Only short-range interaction kernels \( w \in L^1(\R^d) \) are considered, excluding the Coulomb kernel. For the nonlinear statement, an additional decay condition is required:
\[
    \abs{\widehat{w}(k)} \lesssim \frac{1}{\jb{k}^{M-1/2}},
\]
where \( M := \lceil (d+1)/2 \rceil \) and \( \jb{k} := (1 + \abs{k}^2)^{1/2} \). A physically relevant example satisfying these assumptions in dimension \( d = 3 \) is the \textit{screened Coulomb kernel} (or Yukawa potential), given by
\[
    w(x) = \frac{e^{-\alpha\abs{x}}}{\abs{x}}, \qquad \widehat{w}(k) = \frac{4\pi}{\abs{k}^2 + \alpha^2},
\]
for \( \alpha > 0 \). The screening length \( \alpha^{-1} \) controls the range of the interaction: as \( \alpha \to 0 \), the kernel approaches the unscreened Coulomb kernel, which does not satisfy the required assumptions. A further physically relevant example in dimension $d=3$ is the regularised attractive van der Waals-type potential
\[w(x)=-\frac{a}{(1+\abs{x}^2)^3},\qquad \widehat w(k)=-\frac{a\pi^2}{4}\br{1+\ak}e^{-\ak},\]
for $a>0$.

Define the perturbation \(Q := \gamma - \gamma_g\), where \(\gamma\) is a solution of \eqref{H} and $\gamma_g=(2\pi\hbar)^dg(-i\hbar\nabla)$ is a translation-invariant steady state. The evolution of \(Q\) is governed by the nonlinear Hartree equation
\begin{align}\label{NH}
    \begin{mycases}
        i\hbar\p_t Q = \sbr{-\frac{\hbar^2}{2}\Delta + w*\rho_{Q}, Q + \gamma_g}, \\
        Q(0) = Q_{\mathrm{in}}.
    \end{mycases}
\end{align}
The main result establishes phase-mixing estimates for \(Q\), describing the long-time behaviour of solutions near equilibrium: Theorem~\ref{maintheorem} demonstrates mode-by-mode decay of the density \(\rho_Q\) and the scattering behaviour of \(Q\), with bounds that remain uniform in the semiclassical limit $\hbar\to0$.

\subsection{Existing results}
The well-posedness of the Hartree equation was established by Bove, Da Prato and Fano~\cite{BoveDaPratoFano1974,BoveDaPratoFano1976}, and independently by Chadam~\cite{Chadam1976}, under the assumption that the initial data is a trace-class operator. This framework does not include solutions near translation-invariant steady states, which correspond to systems with infinitely many particles and fall outside the trace-class setting. Lewin and Sabin~\cite{LewinSabinI2015} proved local well-posedness for sufficiently regular perturbations of translation-invariant states, together with global well-posedness in several physically relevant cases. Global well-posedness for the analogous nonlinear Schrödinger system was obtained by Chen, Hong and Pavlović~\cite{ChenHongPavlovic2017}.

Lewin and Sabin~\cite{LewinSabinII2014} also addressed large-time stability of stationary states, proving weak convergence to equilibrium in two dimensions. These results were extended to higher dimensions by Chen, Hong and Pavlović~\cite{ChenHongPavlovic2018}, and to the nonlinear Schrödinger system by Hadama~\cite{hadama2024asymptoticstabilitywideclass}. Related results were obtained by Collot and de Suzzoni~\cite{CollotdeSuzzoni2020,CollotdeSuzzoni2022}, who studied an analogous formulation of the Hartree equation for random fields and proved convergence under weaker assumptions on the interaction kernel and at the critical regularity for initial data. A recent preprint by Borie, Hadama and Sabin~\cite{BorieHadamaSabin2025} further establishes asymptotic stability for all short-range kernels within this critical regularity framework.

Under stronger regularity assumptions on the initial data, sharper results have been obtained, with decay and scattering statements that are pointwise in time. Nguyen and You addressed the case of the Coulomb kernel, which differs significantly as the Penrose condition fails for all steady states. Their first results concerned the linearised dynamics~\cite{NguyenYou2023}, followed by decay estimates and modified scattering for the nonlinear equation near vacuum~\cite{NguyenYou2024}.

The classical analogue is \textit{Landau damping}, a well-studied phenomenon in plasma physics where small perturbations in plasma density decay over time due to collective particle motion. This behaviour is described by the Vlasov equation
\begin{align}\label{VP}
\p_tf+v\cdot\nabla_xf-\nabla_x(w*\rho_f)\cdot\nabla_vf=0,\vphantom{2_B}
\end{align}where $f=f(t,x,v)$ is a distribution function, $w$ is an interaction kernel and $\rho_f=\int_{\R^d} f\,dv$ is the spatial density. In his foundational work, Landau demonstrated that solutions of the linearised Vlasov equation around homogeneous steady states, under specific stability conditions, exhibit this time decay~\cite{Landaupaper1946}. The first proof of Landau damping for the nonlinear Vlasov equation was provided by Mouhot and Villani in analytic regularity~\cite{mou-vil-2011}, and nonlinear Landau damping has remained an active area of research since, e.g.~\cite{bed-mas-mou-2016,BedrossianMasmoudiMouhot2018,GrenierNguyenRodnianski2021,BedrossianMasmoudiMouhot2022,IonescuPausaderWangWidmayer2024Poisson}. This paper establishes the quantum analogue of a Landau damping result by Bedrossian, Masmoudi and Mouhot~\cite{BedrossianMasmoudiMouhot2018}, proved in finite regularity on the whole space for the Vlasov equation with the screened Coulomb kernel.

The reduced Planck constant $\hbar\approx 1.054\times 10^{-34}\text{ J\,s}$ is a fundamental physical constant that governs the scale at which quantum effects become significant. Its small magnitude means that, in macroscopic systems, classical mechanics provides an accurate description. However, in regimes with sufficiently small spatial or energetic scales, quantum effects become dominant. Mathematically, the limit $\hbar\to0$ formalises the transition from quantum to classical dynamics: quantum behaviour appears as a perturbation of classical behaviour, and in the limit, quantum corrections vanish and classical dynamics emerge. For mean-field models, this transition connects the Hartree equation, describing the evolution of a quantum many-body system, to the Vlasov equation, the classical analogue that governs the evolution of a macroscopic particle density under mean-field interactions:\begin{align*}
\vphantom{\begin{pmatrix}
    1\\2\\3
\end{pmatrix}}\fbox{\text{\textbf{Hartree}}}\quad\xrightarrow{\hspace{0.6cm}\text{\footnotesize\text{$\hbar\to0$}}\hspace{0.6cm}}\quad\fbox{\text{\textbf{Vlasov}}}
\end{align*}
This correspondence is supported by several results. Among these, Lewin and Sabin~\cite{LewinSabin2020} showed that, for the nonlinear Hartree equation near a translation-invariant steady state, solutions converge weakly to solutions of the Vlasov equation in the semiclassical limit. Chong, Lafleche and Saffirio~\cite{ChongLaflecheSaffirio2023} further established strong $L^2$ convergence for trace-class solutions.

In a recent preprint, You~\cite{You2024} independently established phase-mixing estimates for the nonlinear Hartree equation with short-range kernels. Although the conclusions, such as density decay and scattering, resemble those of the present work, the results are of a different nature: You works at fixed $\hbar$, whereas the present analysis establishes uniform control in the semiclassical limit. His approach leverages the finite-difference structure of the Hartree equation, which features a commutator in place of a derivative. After suitable transforms, the nonlinearity takes the approximate form
\begin{align}\label{finitedifference}
\frac{f^\hbar(v+\frac{\hbar}{2}\nabla_x)-f^\hbar(v-\frac{\hbar}{2}\nabla_x)}{\hbar}(w*\rho_{f^\hbar}),\vphantom{\int^{A^A}_{B_B}}
\end{align}
which, as $\hbar\to0$, formally converges to the nonlinear term in the Vlasov equation:
\begin{align*}
\nabla_vf\cdot\nabla_x(w*\rho_f).\vphantom{\int^A_B}
\end{align*}
At fixed $\hbar$, this structure allows the two terms in \eqref{finitedifference} to be controlled separately, effectively gaining one derivative compared to the Vlasov equation. The trade-off is that the estimates do not remain valid as $\hbar\to0$. In contrast, the present work establishes decay and scattering estimates that are uniform in the semiclassical parameter, by extending the methods of Bedrossian, Masmoudi and Mouhot~\cite{BedrossianMasmoudiMouhot2018}. The price is a stronger regularity requirement, but uniform control is achieved in a more robust norm, namely a weighted quantum Sobolev space. In short, the gain in regularity at fixed $\hbar$ in~\cite{You2024} is exchanged here for uniformity in the semiclassical limit.

\begin{figure}[ht]
\vspace{0cm}
\centering
\begin{tikzpicture}[domain=0:5,xscale=0.25,yscale=0.25]
\draw [->] (1.75,0)--(18.25,0);
\draw [->] (1.75,14)--(18.25,14);
\draw [->] (0,12.5)--(0,1.5);
\draw [->] (20,12.5)--(20,1.5);
\node at (0,0) {$f_{\mathrm{in}}$};
\node at (20,0) {$f_{\infty}$};
\node at (20,14) {$\gamma_{\infty}$};
\node at (0,14) {$\gamma_{\mathrm{in}}$};
\node [below] at (10,0) {\footnotesize{$t\to\infty$}};
\node [above] at (10,14) {\footnotesize{$t\to\infty$}};
\node [left] at (0,7) {\footnotesize{$\hbar\to0$}};
\node [right] at (20,7) {\footnotesize{$\hbar\to0$}};
\end{tikzpicture}
\captionsetup{width=.8\linewidth}
\caption{Semiclassical limit and scattering map (schematic; see Section~\ref{semiclassicalscatteringsection}).}
\label{map}
\end{figure} 
Figure~\ref{map} illustrates how uniform-in-$\hbar$ scattering can be combined with a fixed-time semiclassical limit to identify the semiclassical limit of the Hartree scattering profile with the Vlasov scattering profile obtained by Bedrossian, Masmoudi and Mouhot~\cite{BedrossianMasmoudiMouhot2018}. The interpretation of the diagram, the relevant topology and the conditional limit-exchange argument are made precise in Section~\ref{semiclassicalscatteringsection}, once the required tools are available.

\subsection{Quantum operators and norms} To formulate phase-mixing and scattering estimates in the quantum setting, quantum analogues of classical derivatives and weighted Sobolev norms are introduced. The standard and double Japanese brackets are defined by $\jb{x}:=(1+\abs{x}^2)^{1/2}$ and $\jb{x,y}:=(1+\abs{x}^2+\abs{y}^2)^{1/2}$, respectively. For a multi-index $\alpha=(\alpha_1,\dots,\alpha_d)\in\mathbb N^d$, write
$\abs{\alpha}=\alpha_1+\dots+\alpha_d$ and $D^{\alpha}=\p_1^{\alpha_1}\cdots\p_d^{\alpha_d}$. 

\begin{definition}[Quantum derivatives and weights]\label{toolboxdefinition}
    The quantum analogues of spatial and velocity derivatives are
    \begin{align*}
        \bm{\nabla}_xQ&:=[\nabla,Q],\quad\hspace{0.045cm}\, (\bm{\nabla}_xQ)(x,y)=(\nabla_x+\nabla_y)Q(x,y),\\
        \bm{\nabla}_\xi Q&:=[\mytextfrac{x}{i\hbar},Q],\quad \hspace{0.015cm}(\bm{\nabla}_\xi Q)(x,y)=(\mytextfrac{x-y}{i\hbar})Q(x,y).
    \end{align*}
    The quantum analogues of spatial and velocity weights are 
    \begin{alignat*}{2}
        \bm{x}Q
        &:=\mytextfrac{1}{2}(xQ+Qx),\quad\quad\,\,
        (\bm{x}Q)(x,y)&&=(\mytextfrac{x+y}{2})Q(x,y),\\
        \bm{\xi}Q&:=-\mytextfrac{i\hbar}{2}(\nabla Q+Q\nabla),\quad 
        (\bm{\xi}Q)(x,y)
        &&=-\mytextfrac{i\hbar}{2}(\nabla_x-\nabla_y)Q(x,y).
    \end{alignat*} 
\end{definition}
\begin{remark}[Commutation of the quantum operators]\label{commutationremark}
The operators of Definition~\ref{toolboxdefinition} satisfy the commutation relations
\begin{align*}
    [\bm{\nabla}_x,\bm{\nabla}_\xi]=[\bm{x},\bm{\xi}]=[\bm{\nabla}_x,\bm{\xi}]=[\bm{x},\bm{\nabla}_\xi]=0,
\end{align*}
in exact analogy with the classical operators $\nabla_x$, $\nabla_\xi$, $x$, $\xi$ acting on functions $f(x,\xi)$. Indeed, on the level of integral kernels, $\bm{x}$ and $\bm{\nabla}_x$ involve only the sum $x+y$, while $\bm{\xi}$ and $\bm{\nabla}_\xi$ involve only the difference $x-y$: for instance,
\begin{align*}
    (\nabla_x-\nabla_y)\sbr{\br{\mytextfrac{x+y}{2}}Q(x,y)}=\br{\mytextfrac{x+y}{2}}(\nabla_x-\nabla_y)Q(x,y),
\end{align*}
which gives $[\bm{x},\bm{\xi}]=0$, and similarly $(\nabla_x+\nabla_y)(x-y)=0$ gives $[\bm{\nabla}_x,\bm{\nabla}_\xi]=0$. The remaining relations hold since $\bm{\nabla}_x$ and $\bm{\xi}$ are both constant-coefficient differential operators, and $\bm{x}$ and $\bm{\nabla}_\xi$ are both multiplication operators. In particular, products such as $\bm{\xi}^{\alpha}\bm{x}^{\beta}$ are independent of the order of the factors. By contrast, $[\bm{\nabla}_x,\bm{x}]\neq0$ and $[\bm{\nabla}_\xi,\bm{\xi}]\neq0$, again as in the classical case.
\end{remark}

The classical weighted Sobolev norm $\norm{\cdot}_{H^\sigma_M}$ is now recalled.

\begin{definition}[Weighted Sobolev norm]
Let $\sigma\geq0$ and let $M\geq0$ be an integer. For a function $g=g(v)$, the weighted Sobolev norm $H^{\sigma}_M$ is defined as
\begin{align*}
    \norm{g}_{H^{\sigma}_M}:=\sum_{\aal\leq M}\norm{v^\alpha\jb{\nabla}^\sigma g}_{L^2}.
\end{align*}
For a function $f=f(x,v)$, the weighted Sobolev norm $H^{\sigma}_{M}$ is defined as
\begin{align*}
    \norm{f}_{H^{\sigma}_{M}}:=\sum_{\aal\leq M}\norm{v^\alpha \jb{\nabla_x,\nabla_v}^{\sigma}f}_{L^2},
\end{align*}where the weight applies only to the velocity variable.\end{definition}

The quantum Sobolev norms $\mathcal H^{\sigma}_{M}$ and $\mathcal H^{\sigma}_{N,M}$ are now defined. The former includes weights in $\bm{\xi}$, while the latter includes weights in both $\bm{x}$ and $\bm{\xi}$. The norms depend on $\hbar$ both through a prefactor and through the definitions of the quantum derivative and weight operators $\bm{\nabla}_\xi$ and $\bm{\xi}$.
\begin{definition}[Quantum norms]
For an operator $Q$ on $L^2(\R^d;\C)$ with integral kernel $Q = Q(x,y)$, the following norms are defined for $N, M \in \mathbb{N}$:
\begin{enumerate}[label=(\alph*)]
    \item Quantum Lebesgue norm $\mathcal{L}^2$:
    \begin{align*}
        \norm{Q}_{\mathcal{L}^2} := (2\pi\hbar)^{-d/2} \norm{Q(x,y)}_{L^2_{x,y}}.
    \end{align*}
    
    \item $\bm{\xi}$-weighted quantum Sobolev norm $\mathcal{H}^{\sigma}_M$:
    \begin{align*}
        \norm{Q}_{\mathcal{H}^{\sigma}_{M}} := \sum_{\aal \leq M} \norm{\bm{\xi}^\alpha \jb{\bm{\nabla}_x,\bm{\nabla}_\xi}^\sigma Q}_{\mathcal{L}^2}.
    \end{align*}
    
    \item $\bm{x}$–$\bm{\xi}$-weighted quantum Sobolev norm $\mathcal{H}^{\sigma}_{N,M}$:
    \begin{align*}
        \norm{Q}_{\mathcal{H}^{\sigma}_{N,M}} := \sum_{\aal \leq M,\, \abs{\beta} \leq N} \norm{\bm{x}^\beta \bm{\xi}^\alpha \jb{\bm{\nabla}_x,\bm{\nabla}_\xi}^\sigma Q}_{\mathcal{L}^2}.
    \end{align*}
\end{enumerate}
\end{definition}

Note that the norm $\norm{Q(x,y)}_{L^2_{x,y}}$ is equal to the $\mathfrak{S}^2$~Schatten norm, that is, the Hilbert--Schmidt norm.

\subsection{Main result}The main result establishes quantitative phase-mixing and scattering estimates for the nonlinear Hartree equation near translation-invariant steady states, with bounds that remain uniform in the semiclassical limit. This provides a robust quantum analogue of nonlinear Landau damping in the Vlasov equation.
\begin{theorem}[Nonlinear phase mixing]\label{maintheorem}
Fix dimension $d\geq3$ and $M:=\ceils{(d+1)/2}$. Choose constants $N_0>\sigma_1>\sigma_0$ satisfying \[\sigma_1\geq d+8,\quad N_0-\sigma_1>3d/2+7/2 \qaq \sigma_1-\sigma_0>d+1.\] Let $g\in H^{N_0+7/2+\delta_0}_{2+\ceils{d/2}}(\R^d;\R_+)$ for some constant $\delta_0\in(0,1]$, and let $w\in L^1(\R^d,\R)$ satisfy \begin{align}
    \abs{\widehat w(k)}\lesssim\frac{1}{\jb{k}^{M-1/2}}\label{wassumption}
\end{align}for all $k\in\R^d$. Assume $g$ and $w$ satisfy the uniform Penrose condition \eqref{Penrose} with constant ${c_{\mathrm{P}}}$. Then there exist constants $C>0$ and $\ep_0>0$, depending only on $N_0,\sigma_1,\sigma_0,w,g,{c_{\mathrm{P}}},d,\delta_0$, and independent of $\hbar\in(0,1]$, such that the following holds.

For any $\ep\in(0,\ep_0]$ and any self-adjoint initial data $\Qin\in\mathcal{H}^{N_0}_{M,M}$ with \[\norm{\Qin}_{\mathcal{H}^{N_0}_{M,M}}\leq\ep,\]
the nonlinear Hartree equation \eqref{NH} admits a unique global self-adjoint solution $Q(t)$, which satisfies:
\begin{itemize}
    \item Density decay: for all $k\in\R^d$ and $t\geq 0$,\begin{align}
        \label{T1}\abs{\widehat{\rho_Q}(t,k)}&\leq \frac{C\ep}{\jb{k,  kt}^{\sigma_1}}.
    \end{align}
    \item Scattering: there exists an operator $Q_\infty\in\mathcal H^{\sigma_0}_M(\R^d;\C)$ such that, for all $t\geq 0$, \begin{align}
        \label{T2}
        \norm{e^{-i\frac{\hbar}{2}t\Delta}Q(t)e^{i\frac{\hbar}{2}t\Delta}-Q_\infty}_{\mathcal H^{\sigma_0}_{M}}&\leq \frac{C\ep}{\jb{t}^{d/2}}.
    \end{align}
\end{itemize}
\end{theorem}

The decay estimate for the density in Fourier space established in Theorem~\ref{maintheorem} implies physical-space decay in $L^p$ norms, as stated in the following corollary.
\begin{corollary}[Physical-space decay]
    Let $d\geq 3$, $p\in[2,\infty]$ and let $N_0>\sigma_1>0$ with
    \begin{align*}
        \sigma_1\geq d+8\qaq N_0-\sigma_1>3d/2+7/2.
    \end{align*}
    Then for any $n<\sigma_1-d(1-1/p)$, there exist constants $C>0$ and $\ep_0>0$, depending only on $d,p,n,N_0,\sigma_1$, and the parameters in Theorem~\ref{maintheorem}, and independent of $\hbar\in(0,1]$, such that the following holds.

    For any $\ep\in(0,\ep_0]$ and any self-adjoint initial data $\Qin\in\mathcal{H}^{N_0}_{M,M}$ with\begin{align*}
        \norm{Q_{\mathrm{in}}}_{\mathcal H^{N_0}_{M,M}}\leq\ep,
    \end{align*}
     the solution $Q(t)$ to the nonlinear Hartree equation satisfies
    \begin{align}\label{physicalspacedecay}
        \norm{\jb{\nabla_x,t\nabla_x}^n\rho_Q(t,\cdot)}_{L^p(\R^d)}\leq \frac{C\ep}{\jb{  t}^{d(1-\frac{1}{p})}}.
    \end{align}
\end{corollary}
\begin{proof}
By the Hausdorff--Young inequality,
\begin{align*}
    \norm{\jb{\nabla_x,t\nabla_x}^n\rho_Q(t,\cdot)}_{L^p(\R^d)} 
    \lesssim_p \br{\int_{\R^d} \jb{k,kt}^{p'n} \abs{\widehat{\rho_Q}(t,k)}^{p'} \,dk}^{1/p'}.
\end{align*}
Using the decay estimate \eqref{T1} and requiring convergence of the integral yields the result.
\end{proof}

\begin{remark}\label{optimalityremark}
The physical-space decay rate in \eqref{physicalspacedecay} is optimal within the stated class. Indeed, the hypotheses admit $w=0$, for which the Penrose condition holds trivially, and for rank-one initial data the density is then the squared modulus of a solution of the free Schrödinger equation, see Lewin and Sabin~\cite{LewinSabinI2015}. The sharp large-time asymptotics of the free flow, see e.g.~\cite[Chapter~2]{Cazenave2003}, therefore preclude any decay uniformly faster than $\jb{t}^{-d(1-1/p)}$. The underlying free phase-mixing mechanism is described in Section~\ref{freephasemixingsection}. The scattering rate $\jb{t}^{-d/2}$ in \eqref{T2} is dictated by the low-frequency linear response: the leading linear contribution in the Duhamel representation of the conjugated perturbation decays at rate $\jb{t}^{-d/2-1}$, and its time integral at rate $\jb{t}^{-d/2}$. The scattering rate $\jb{t}^{-d/2}$ in~\eqref{T2} coincides with that obtained for the corresponding Vlasov problem in~\cite{BedrossianMasmoudiMouhot2018}.
\end{remark}

\begin{remark}
In dimension $d=3$, the assumptions on $w$ include the screened Coulomb kernel, extending the nonlinear Landau damping theorem of Bedrossian, Masmoudi and Mouhot~\cite{BedrossianMasmoudiMouhot2018} for the Vlasov equation to the quantum setting with uniform-in-$\hbar$ control. Compared with the weak-convergence and stability results discussed above, the present theorem assumes more regularity of the initial data and yields convergence in a stronger topology, while retaining uniformity in $\hbar$. The interaction kernel, by contrast, is not required to be bounded: only $w\in L^1(\R^d)$ together with the Fourier decay~\eqref{wassumption} is assumed, so that kernels which are singular in physical space, such as the screened Coulomb potential, are admitted. This contrasts with the bounded short-range potentials assumed, for instance, in the work of Lewin and Sabin~\cite{LewinSabinI2015}.
\end{remark}

\begin{remark}
The distinguishing feature of Theorem~\ref{maintheorem} is that the scattering estimate is uniform in $\hbar\in(0,1]$ and is formulated in a quantum Sobolev topology compatible with the semiclassical limit. This uniformity provides a framework for comparing the quantum and classical scattering profiles. A precise conditional statement is given in Section~\ref{semiclassicalscatteringsection}.
\end{remark}

\subsection{Fourier and Wigner preliminaries}\label{notationssection}

This subsection collects the Fourier conventions and basic properties of the Wigner transform used throughout the paper.

Throughout, integrals such as $\int_x\,dx$ denote integration over $\R^d$ in the $x$ variable. The following conventions are adopted for the Fourier transform of functions $g=g(v)$ and $f=f(x,v)$:
\begin{align*}
    \widehat g(p):=\int_v e^{-ip\cdot v}g(v)\,dv,\qquad \widehat f(k,p):=\int_x\int_ve^{-ik\cdot x}e^{-ip\cdot v}f(x,v)\,dvdx.
\end{align*}
On the Fourier side, the operators defined in Definition \ref{toolboxdefinition} obey the relations: \begin{alignat*}{3}
    \widehat{\bm{\nabla}_xQ}(k,p) &= i(k+p)\widehat Q(k,p), 
    &\qquad\quad \widehat{\bm{\nabla}_xQ}(\mytextfrac{k}{2}+p,\mytextfrac{k}{2}-p) &= ik\widehat Q(\mytextfrac{k}{2}+p,\mytextfrac{k}{2}-p),\\
    \widehat{\bm{\nabla}_\xi Q}(k,p) &= \mytextfrac{1}{\hbar}(\nabla_k-\nabla_p)\widehat Q(k,p), 
    &\qquad\quad \widehat{\bm{\nabla}_\xi Q}(\mytextfrac{k}{2}+p,\mytextfrac{k}{2}-p) &= \mytextfrac{1}{\hbar}\nabla_p \sbr{\widehat Q(\mytextfrac{k}{2}+p,\mytextfrac{k}{2}-p)},\\
    \widehat{\bm{x} Q}(k,p) &= \mytextfrac{i}{2}(\nabla_k+\nabla_p)\widehat Q(k,p), 
    &\qquad\quad \widehat{\bm{x} Q}(\mytextfrac{k}{2}+p,\mytextfrac{k}{2}-p) &= i\nabla_k \sbr{\widehat Q(\mytextfrac{k}{2}+p,\mytextfrac{k}{2}-p)},\\
    \widehat{\bm{\xi}Q}(k,p) &= \mytextfrac{\hbar}{2}(k-p)\widehat Q(k,p), 
    &\qquad\quad \widehat{\bm{\xi}Q}(\mytextfrac{k}{2}+p,\mytextfrac{k}{2}-p) &= \hbar p\widehat Q(\mytextfrac{k}{2}+p,\mytextfrac{k}{2}-p).
\end{alignat*}

The Wigner transform facilitates the representation of the density matrix in phase space, revealing structural parallels with classical kinetic equations.
\begin{definition}[Wigner transform]\label{Wignerdefinition}
    For an operator $Q$ on $L^2(\R^d;\C)$ with integral kernel $Q(x,y)$, define the Wigner transform of $Q$, $W[Q]\colon\R^d\times\R^d\to\C$, by \begin{align*}
        W[Q](x,\xi)= \tpd\frac{1}{\hbar^d}\intRd e^{-i\xi\cdot y/\hbar}Q\br{x+\mytextfrac{y}{2},x-\mytextfrac{y}{2}}\,dy.
    \end{align*}
\end{definition}
The Fourier transform of the Wigner transform of $Q$ satisfies
\begin{align*}
        \widehat{W[Q]}(k,\eta)=\int_xe^{-ik\cdot x}Q\br{x-\mytextfrac{\hbar \eta}{2},x+\mytextfrac{\hbar \eta}{2}}\,dx=\tpd\int_p e^{-i\hbar\eta\cdot p }\widehat{Q}\br{\mytextfrac{k}{2}+p,\mytextfrac{k}{2}-p}\,dp,
    \end{align*}and can be controlled by the $\mathcal H_{M,M}^{0}$ norm of $Q$ via the Sobolev embedding.
\begin{lemma}\label{FTofWignertransform}
Suppose $Q\in\mathcal H^0_{M,M}(\R^d;\C)$. Then the Wigner transform of $Q$ satisfies, for all $k,\eta\in\R^d$,
    \begin{align*}
        \abs{\widehat {W[Q]}(k,\eta)}\lesssim_d\norm{Q}_{\mathcal H^{0}_{M,M}}.
    \end{align*}
\end{lemma}
\begin{proof}
    By Definition \ref{toolboxdefinition}, Cauchy--Schwarz and the Sobolev embedding, for $0<\widetilde\delta\leq 1/2$,
\begin{align*}
    \abs{\widehat{W[Q]}(k,\eta)}&=\abs{\int_x e^{-ik\cdot x}Q\br{x-\mytextfrac{\hbar}{2}\eta,x+\mytextfrac{\hbar}{2}\eta}\,dx}\\
    &\lesssim\br{\int_x\abs{\jb{x}^{{{d/2}+\widetilde\delta}}Q\br{x-\mytextfrac{\hbar}{2}\eta,x+\mytextfrac{\hbar}{2}\eta}}^2\,dx}^{1/2}\\
    &\lesssim\br{\int_{x,y}\abs{\jb{\nabla_y}^{{d/2}+\widetilde\delta}\sbr{\jb{x}^{{{d/2}+\widetilde\delta}}Q\br{x-\mytextfrac{\hbar}{2}y,x+\mytextfrac{\hbar}{2}y}}}^2\,dxdy}^{1/2}\\
    &=\br{\int_{x,y}\abs{\jb{\mytextfrac{\hbar}{2}(\nabla_x-\nabla_y)}^{{d/2}+\widetilde\delta}\jb{x}^{{{d/2}+\widetilde\delta}}Q\br{x-\mytextfrac{\hbar}{2}y,x+\mytextfrac{\hbar}{2}y}}^2\,dxdy}^{1/2}\\
    &=\hbar^{-d/2}\norm{\jb{\bm{\xi}}^{{d/2}+\widetilde\delta}\jb{\bm{x}}^{{{d/2}+\widetilde\delta}}Q(x,y)}_{L^2_{x,y}}\\
    &\lesssim\hbar^{-d/2}\sum_{\aal,\abs{\beta} \leq M}\norm{\bm{\xi}^\alpha\bm{x}^\beta Q(x,y)}_{L^2_{x,y}},
\end{align*}
where the final inequality uses that $M:=\ceils{(d+1)/2}$ is an integer with $M\geq d/2+\widetilde\delta$, together with the commutation of $\bm{\xi}$ and $\bm{x}$ (Remark~\ref{commutationremark}). The lemma then follows from the definition of the quantum Sobolev norm.
\end{proof}

The operators defined in Definition \ref{toolboxdefinition} interact nicely with the Wigner transform:
\begin{lemma}\label{operatorsWigner}Suppose $Q\in\mathcal L^2(\R^d;\C)$. Then the Wigner transform of $Q$ satisfies the following:\begin{alignat*}{2}
    \nabla_xW[Q](x,\xi)&=W[\bm{\nabla}_x Q](x,\xi),\qquad ik\widehat{W[Q]}(k,\eta)&&=\reallywidehat{W[\bm{\nabla}_x Q]}(k,\eta), \\ 
     \nabla_\xi W[Q](x,\xi)&=W[\bm{\nabla}_\xi Q](x,\xi),\qquad i\eta \widehat{W[Q]}(k,\eta)&&=\reallywidehat{W[\bm{\nabla}_\xi Q]}(k,\eta),  \\
     xW[Q](x,\xi)&=W[\bm{x}Q](x,\xi),\qquad i\nabla_k\widehat{W[Q]}(k,\eta)&&=\reallywidehat{W[\bm{x}Q]}(k,\eta),  \\
         \xi W[Q](x,\xi)&=W[\bm{\xi}Q](x,\xi),\qquad i\nabla_\eta\widehat{W[Q]}(k,\eta)&&=\reallywidehat{W[\bm{\xi}Q]}(k,\eta). 
\end{alignat*}
\end{lemma}
\begin{proof} These statements follow directly from the definitions of the Wigner transform (Definition \ref{Wignerdefinition}) and the quantum operators (Definition \ref{toolboxdefinition}). \end{proof}

The $L^2$ norm of $W[Q]$ is related to the $\mathcal L^2$ norm of $Q$ via the following relation:
\begin{lemma}\label{L2normofWigner}
    Suppose $Q\in\mathcal L^2(\R^d;\C)$. Then the Wigner transform of $Q$ satisfies
    \begin{align*}
        \norm{W[Q](x,\xi)}_{L^2_{x,\xi}}=(2\pi\hbar)^{-d/2}\norm{Q(x,y)}_{L^2_{x,y}}=\norm{Q}_{\mathcal L^2}.
    \end{align*}
\end{lemma}
\begin{remark}\label{wignerisometry}
Combining Lemma~\ref{operatorsWigner} with Lemma~\ref{L2normofWigner} shows that the Wigner transform is an isometry between the quantum and classical weighted Sobolev scales: for every $\sigma\geq0$, every integer $M\geq0$ and every $Q\in\mathcal H^{\sigma}_M(\R^d;\C)$,
\begin{align*}
    \norm{W[Q]}_{H^{\sigma}_M}=\norm{Q}_{\mathcal H^{\sigma}_M}.
\end{align*}
Indeed, the Wigner transform maps the quantum operators $\bm{\nabla}_x$, $\bm{\nabla}_\xi$ and $\bm{\xi}$ to $\nabla_x$, $\nabla_\xi$ and multiplication by $\xi$, respectively, so that $\xi^\alpha\jb{\nabla_x,\nabla_\xi}^{\sigma}W[Q]=W\sbr{\bm{\xi}^\alpha\jb{\bm{\nabla}_x,\bm{\nabla}_\xi}^{\sigma}Q}$; the identity then follows term by term from Lemma~\ref{L2normofWigner}.
\end{remark}
\begin{proof}
    The result follows from the observation that \begin{align*}
        W[Q](x,\xi)=(2\pi\hbar)^{-d}\mathcal F\sbr{Q\br{x+\mytextfrac{\cdot}{2},x-\mytextfrac{\cdot}{2}}}\br{\xi/\hbar},
    \end{align*}combined with Plancherel's theorem.
\end{proof}
For an operator $Q$, the density $\rho_Q(x):=Q(x,x)$ is, formally, the velocity marginal of the Wigner transform, $\rho_Q(x)=\int_\xi W[Q](x,\xi)\,d\xi$: on the Fourier side, setting $\eta=0$ in the Fourier representation of the Wigner transform above,
\begin{align*}
    \widehat{\rho_Q}(k)=\widehat{W[Q]}(k,0)=\tpd\int_p\widehat Q\br{\mytextfrac{k}{2}+p,\mytextfrac{k}{2}-p}\,dp.
\end{align*}
The following lemma makes this representation rigorous.
\begin{lemma}\label{densityofoperator}
Suppose $Q\in\mathcal H^0_{M,M}(\R^d;\C)$, where $M=\ceils{(d+1)/2}$, identified with its integral kernel $Q=Q(x,y)$. Then the density $\rho_Q(x):=Q(x,x)$ satisfies, for all $k\in\R^d$,
\begin{align*}
    \widehat{\rho_Q}(k)=\tpd\int_p\widehat Q(k-p,p)\,dp=\tpd\int_p\widehat Q\br{\mytextfrac{k}{2}+p,\mytextfrac{k}{2}-p}\,dp,
\end{align*}
both integrals converging absolutely.
\end{lemma}
\begin{proof}
First suppose that the kernel $Q(x,y)$ is a Schwartz function on $\R^d\times\R^d$. By Fourier inversion and Fubini's theorem,
\begin{align*}
    \rho_Q(x)=Q(x,x)&=\frac{1}{(2\pi)^{2d}}\int_{k'}\int_p e^{i(k'+p)\cdot x}\widehat Q(k',p)\,dp\,dk'\\
    &=\tpd\int_k e^{ik\cdot x}\sbr{\tpd\int_p\widehat Q(k-p,p)\,dp}\,dk,
\end{align*}
where the final equality follows from the substitution $k'=k-p$. Hence $\rho_Q$ is the inverse Fourier transform of the bracketed function, which gives the first representation; the second follows from the change of variables $p\mapsto\mytextfrac{k}{2}-p$.

It remains to extend the identity to $Q\in\mathcal H^0_{M,M}$. Set $F_Q(k,p):=\widehat Q\br{\mytextfrac{k}{2}+p,\mytextfrac{k}{2}-p}$. Since $M>d/2$, the Cauchy--Schwarz inequality in $p$ and the Sobolev embedding in $k$ give
\begin{align*}
    \sup_k\int_p\abs{F_Q(k,p)}\,dp\lesssim_d\hbar^{-d/2}\sup_k\norm{\jb{\hbar p}^{M}F_Q(k,\cdot)}_{L^2_p}&\lesssim_d\hbar^{-d/2}\sum_{\abs{\beta}\leq M}\norm{\p_k^\beta\br{\jb{\hbar p}^{M}F_Q}}_{L^2_{k,p}}\\&\lesssim_d\norm{Q}_{\mathcal H^{0}_{M,M}},
\end{align*}
where the final inequality expands $\jb{\hbar p}^{M}$ into monomials ($M$ being an integer) and uses the relations
\begin{align*}
    \widehat{\bm{x}Q}\br{\mytextfrac{k}{2}+p,\mytextfrac{k}{2}-p}=i\nabla_kF_Q(k,p),\qquad
    \widehat{\bm{\xi}Q}\br{\mytextfrac{k}{2}+p,\mytextfrac{k}{2}-p}=\hbar p\,F_Q(k,p),
\end{align*}
iterated, together with Plancherel's theorem and the normalisation of the quantum Lebesgue norm. Thus the integrals in the statement converge absolutely, uniformly in $k$.

For each fixed $k$, the estimate above shows that the right-hand side of the stated identity depends continuously on $Q\in\mathcal H^{0}_{M,M}$, while, since $\widehat{\rho_Q}(k)=\widehat{W[Q]}(k,0)$ by the Fourier representation of the Wigner transform, Lemma~\ref{FTofWignertransform} shows the same for the left-hand side. The identity therefore extends to all of $\mathcal H^{0}_{M,M}$ by density of Schwartz kernels.
\end{proof}

\subsection{Free phase mixing}\label{freephasemixingsection} Consider the linear equation obtained by removing the interaction term from the Hartree equation, referred to here as the \textit{free Schrödinger equation}: \begin{align}\label{FS}
    \begin{mycases}
    i\hbar\p_tQ_{\mathrm{FS}}=\sbr{-\frac{\hbar^2}{2}\Delta,Q_{\mathrm{FS}}},\\
    Q_{\mathrm{FS}}(0)=\Qin,
\end{mycases}
\end{align} where ${Q_{\mathrm{FS}}}(t)$ is an operator on $L^2(\R^d;\C)$. The solution is given explicitly by\begin{align}
    {Q_{\mathrm{FS}}}(t)=e^{i\frac{\hbar}{2} t\Delta}\Qin e^{-i  \frac{\hbar}{2} t\Delta},\label{FSsol}
\end{align}with integral kernel and Fourier transform\begin{align*}
    {Q_{\mathrm{FS}}}(t,x,y)=e^{i {\frac{\hbar}{2}} t(\Delta_x-\Delta_y)}\Qin(x,y),\quad\widehat{Q_{\mathrm{FS}}}(t,k,p)=e^{-i {\frac{\hbar}{2}} t(\ak^2-\ap^2)}\widehat\Qin(k,p).
\end{align*}
By Lemma~\ref{densityofoperator}, the density of $Q_{\mathrm{FS}}$ has Fourier transform \begin{align*}
\widehat{\rho_{{Q_{\mathrm{FS}}}}}(t,k)=\tpd\int_pe^{-i \hbar p\cdot kt}\widehat \Qin\br{\mytextfrac{k}{2}+p,\mytextfrac{k}{2}-p}\,dp =\widehat{W[\Qin]}(k,kt).
\end{align*}

The next result gives decay estimates for each Fourier mode of the density under free evolution.
\begin{proposition}[Phase mixing for the free Schrödinger equation] \label{propfreephasemixing}
Fix $N_0\geq0$. Then for any $t\geq0$ and $k\in\R^d$, the density of the solution $Q_{\mathrm{FS}}$ to \eqref{FS} satisfies\begin{align*}
\abs{\widehat{\rho_{{Q_{\mathrm{FS}}}}}(t,k)}\lesssim_d\frac{ \norm{Q_{\mathrm{in}}}_{\mathcal H^{N_0}_{M,M}}}{\jb{k,kt}^{N_0}}.
\end{align*}
\end{proposition}
\begin{proof}
First recall that $\widehat{\rho_{Q_{\mathrm{FS}}}}(t,k)=\widehat{W[\Qin]}(k,kt)$. Since $\bm{\nabla}_x$ and $\bm{\nabla}_\xi$ commute (Remark~\ref{commutationremark}), the relations of Lemma~\ref{operatorsWigner} extend from the individual operators to the joint weight $\jb{\bm{\nabla}_x,\bm{\nabla}_\xi}^{N_0}$, whence
\begin{align*}
    \jb{k,kt}^{N_0}\widehat{\rho_{Q_{\mathrm{FS}}}}(t,k)=\jb{k,\eta}^{N_0}\widehat{W[\Qin]}(k,\eta)\Big|_{\eta=kt}=\reallywidehat{W\sbr{\jb{\bm{\nabla}_x,\bm{\nabla}_{\xi}}^{N_0}Q_{\mathrm{in}}}}(k,kt).
\end{align*}
The result then follows from Lemma~\ref{FTofWignertransform}.
\end{proof}

The same type of estimate for the nonlinear Hartree equation is established in Theorem~\ref{maintheorem} under stronger regularity assumptions on the initial data. Moreover, the scattering statement \eqref{T2} indicates that the solution $Q(t)$ of \eqref{NH} behaves like a solution of the free Schrödinger equation as $t\rightarrow\infty$.

\subsection{Semiclassical limit of the scattering profiles}\label{semiclassicalscatteringsection}
The limit exchange suggested by Figure~\ref{map} is now made precise as a conditional consequence of Theorem~\ref{maintheorem}. The upper horizontal arrow is provided by the uniform quantum scattering estimate, while the lower horizontal arrow follows from the classical scattering theorem of Bedrossian, Masmoudi and Mouhot~\cite{BedrossianMasmoudiMouhot2018}. Assuming that the left vertical arrow propagates from the initial data to each fixed time, these two scattering statements yield the right vertical arrow. The argument uses the three-dimensional formulation of the classical theorem, so $d=3$ is fixed throughout this subsection; this restriction comes from the classical input rather than from Theorem~\ref{maintheorem}.

Fix $0\leq\sigma\leq\sigma_0$, where $\sigma_0$ is the regularity index of the scattering estimate \eqref{T2}, and assume that $N_0$ is chosen sufficiently large that the regularity hypothesis on $g$ in Theorem~\ref{maintheorem} implies the regularity required by the classical theorem for scattering in $H^{\sigma}_M$. Since $w\in L^1$ is already assumed in Theorem~\ref{maintheorem}, impose in addition
\begin{align*}
    \abs{\widehat w(k)}\lesssim\jb{k}^{-2},
\end{align*}
as required by the three-dimensional classical result; this strengthens \eqref{wassumption} and is satisfied, in particular, by the screened Coulomb interaction. Under the stated regularity assumption on $g$, the classical Penrose condition follows from the uniform condition~\eqref{Penrose} by dominated convergence as $\hbar\to0$ in the dispersion function. Define the free transport operator
\begin{align*}
    (\mathcal T_tF)(x,\xi):=F(x+t\xi,\xi).
\end{align*}

\begin{itemize}
    \item For the upper horizontal arrow, fix a family of self-adjoint initial perturbations $\Qin^{\hbar}$, $\hbar\in(0,1]$, satisfying
\begin{align*}
    \sup_{\hbar\in(0,1]}\norm{\Qin^{\hbar}}_{\mathcal H^{N_0}_{M,M}}\leq\ep
\end{align*}
for some $\ep\in(0,\ep_0]$. For each $\hbar\in(0,1]$, write
\begin{align*}
    \gamma_g^{\hbar}:=(2\pi\hbar)^dg(-i\hbar\nabla),\qquad Q^{\hbar}(t):=\gamma^{\hbar}(t)-\gamma_g^{\hbar},\qquad P^{\hbar}(t):=e^{-i\frac{\hbar}{2}t\Delta}Q^{\hbar}(t)e^{i\frac{\hbar}{2}t\Delta},
\end{align*}
and let $Q^{\hbar}_{\infty}:=\lim_{t\to\infty}P^{\hbar}(t)$ be the scattering profile of Theorem~\ref{maintheorem}, the limit taken in $\mathcal H^{\sigma_0}_M$. Since $\sigma\leq\sigma_0$, the scattering estimate \eqref{T2} and Remark~\ref{wignerisometry} give
\begin{align}\label{uniformquantumscattering}
    \sup_{\hbar\in(0,1]}\norm{W[P^{\hbar}(t)]-W[Q^{\hbar}_{\infty}]}_{H^{\sigma}_M}\to0\qquad\text{as $t\to\infty$.}
\end{align}
    \item For the lower horizontal arrow, let $\hin$ be mean-zero and satisfy the weighted Sobolev regularity and smallness assumptions of the classical theorem, and let $f(t,x,v)=g(v)+h(t,x,v)$ be the solution of the Vlasov equation~\eqref{VP} with initial datum $g+\hin$. The theorem of Bedrossian, Masmoudi and Mouhot provides a scattering profile $h_\infty\in H^{\sigma}_M$ such that
\begin{align}\label{classicalscatteringprofile}
    \mathcal T_th(t)\to h_\infty\qquad\text{in $H^{\sigma}_M$ as $t\to\infty$.}
\end{align}

\item The left vertical arrow is assumed: suppose that
\begin{align*}
    W[\Qin^{\hbar}]\to\hin\qquad\text{in $H^{\sigma}_M$ as $\hbar\to0$,}
\end{align*}
and that this convergence propagates to each fixed time:
\begin{align}\label{fixedtimesemiclassical}
    W[Q^{\hbar}(t)]\to h(t)\qquad\text{in $H^{\sigma}_M$ as $\hbar\to0$}
\end{align}
for every fixed $t\geq0$.
\end{itemize}

Then the right vertical arrow holds:
\begin{align*}
    W[Q^{\hbar}_{\infty}]\to h_\infty\qquad\text{in $H^{\sigma}_M$ as $\hbar\to0$.}
\end{align*}
To prove this, note that direct computations from Definition~\ref{Wignerdefinition} give
\begin{align*}
    W[\gamma_g^{\hbar}](x,\xi)=g(\xi)\qaq W\sbr{e^{-i\frac{\hbar}{2}t\Delta}Qe^{i\frac{\hbar}{2}t\Delta}}=\mathcal T_tW[Q].
\end{align*}
Moreover, $\mathcal T_t$ is bounded on $H^{\sigma}_M$ for each fixed $t$: 
\begin{align*}
    \norm{\mathcal T_tF}_{H^{\sigma}_M}\lesssim_{t,\sigma,M}\norm{F}_{H^{\sigma}_M}.
\end{align*}
For every fixed $t\geq0$, the free-flow identity and the triangle inequality give
\begin{align*}
    \norm{W[Q^{\hbar}_{\infty}]-h_\infty}_{H^{\sigma}_M}
    &\leq\norm{W[Q^{\hbar}_{\infty}]-W[P^{\hbar}(t)]}_{H^{\sigma}_M}\\
    &\quad+\norm{\mathcal T_t\br{W[Q^{\hbar}(t)]-h(t)}}_{H^{\sigma}_M}+\norm{\mathcal T_th(t)-h_\infty}_{H^{\sigma}_M}.
\end{align*}
Taking the upper limit as $\hbar\to0$ and using \eqref{fixedtimesemiclassical} and the boundedness of $\mathcal T_t$ yields
\begin{align*}
    \limsup_{\hbar\to0}\norm{W[Q^{\hbar}_{\infty}]-h_\infty}_{H^{\sigma}_M}\leq\sup_{\hbar\in(0,1]}\norm{W[P^{\hbar}(t)]-W[Q^{\hbar}_{\infty}]}_{H^{\sigma}_M}+\norm{\mathcal T_th(t)-h_\infty}_{H^{\sigma}_M}.
\end{align*}
The left-hand side is independent of $t$. Letting $t\to\infty$ and applying \eqref{uniformquantumscattering} and \eqref{classicalscatteringprofile} proves the right vertical arrow. Equivalently, setting $\gamma^{\hbar}_{\infty}:=\gamma_g^{\hbar}+Q^{\hbar}_{\infty}$ and $f_\infty:=g+h_\infty$, the identity $W[\gamma_g^{\hbar}]=g$ gives
\begin{align*}
    W[\gamma^{\hbar}_{\infty}]-f_\infty=W[Q^{\hbar}_{\infty}]-h_\infty\to0\qquad\text{in $H^{\sigma}_M$ as $\hbar\to0$.}
\end{align*}

Establishing the fixed-time convergence \eqref{fixedtimesemiclassical} in this strong topology is left to future work; the uniform estimates of Theorem~\ref{maintheorem} provide a natural framework for it.

\subsection{Fourier and Wigner formulations of the Hartree equation}This subsection rewrites the nonlinear Hartree equation in both Fourier and Wigner form, making the structure of the linear and nonlinear terms explicit.
\subsubsection{Fourier formulation}
The nonlinear Hartree equation \eqref{H} takes the form
\begin{align*}
    i\hbar\p_tQ(t,x,y)=-{\frac{\hbar^2}{2}}\br{\Delta_x-\Delta_y}Q(t,x,y)+\br{V(t,x)-V(t,y)}\br{\gamma_g(x,y)+Q(t,x,y)},
\end{align*}where $V(s,x)=w*\rho_Q(s,x)$.
 Applying Duhamel's formula yields an implicit expression for $Q$:\begin{align*}\begin{split}\
    Q(t,x,y)&=e^{i{\frac{\hbar}{2}} (\Delta_x-\Delta_y)t}Q_{\mathrm{in}}(x,y)\\
    &\hspace{0.5cm}-
    \frac{i}{\hbar}\int_0^te^{i{\frac{\hbar}{2}}(\Delta_x-\Delta_y)(t-s)}\sbr{V(s,x)-V(s,y)}\br{\gamma_g(x,y)+Q(s,x,y)}\,ds.\end{split}
\end{align*}
Taking the Fourier transform in both spatial variables gives\begin{align*}
    \widehat Q(t,k,p)&=e^{-i{\frac{\hbar}{2}}(\ak^2-\ap^2)t}\widehat{Q_{\mathrm{in}}}(k,p)\\&\quad-
    \frac{i}{\hbar}\tpd\int_0^te^{-i{\frac{\hbar}{2}}(\ak^2-\ap^2)(t-s)}\sbr{\widehat V(s,k)*_k-\widehat V(s,p)*_p}\br{\widehat{\gamma_g}(k,p)+\widehat{Q}(s,k,p)}\,ds.
\end{align*}Recalling that $\widehat{\gamma_g}(k,p)=(2\pi)^{2d}\hbar^d g(\hbar k)\,\delta(p+k)$, where $\delta$ denotes the Dirac delta, and $\widehat V(t,k)=\widehat w(k)\widehat{\rho_Q}(t,k)$ gives\begin{align*}\begin{split}
    \widehat Q(t,k,p)&=e^{-i{\frac{\hbar}{2}}(\ak^2-\ap^2)t}\widehat{Q_{\mathrm{in}}}(k,p)\\
    &-i\hbar^{d-1}(2\pi)^d\widehat w(k+p)\sbr{g(-\hbar p)-g(\hbar k)}\int_0^te^{-i{\frac{\hbar}{2}}(\ak^2-\ap^2)(t-s)}\widehat{\rho_Q}(s,k+p)\,ds\\
    &-
    \frac{i}{\hbar}\tpd\int_0^t\int_\ell e^{-i{\frac{\hbar}{2}}(\ak^2-\ap^2)(t-s)}\widehat w(\ell )\widehat{\rho_Q}(s,\ell )\br{\widehat{Q}(s,k-\ell ,p)-\widehat{Q}(s,k,p-\ell )}\,d\ell ds.\end{split}
\end{align*}

\subsubsection{Conjugation by the free Schrödinger flow}\label{conjugationSection}
Recall the operator $P$ obtained by conjugating $Q$ along the free evolution:
\begin{align}\label{conjugationrelation}
    Q(t,x,y)=e^{i\frac{\hbar}{2}(\Delta_x-\Delta_y)t}P(t,x,y),\quad \widehat Q(t,k,p)=e^{-i\frac{\hbar}{2}(\ak^2-\ap^2)t}\widehat P(t,k,p).
\end{align}
This conjugation reverses $Q$ along the free flow, removing the dominant free dynamics and allowing $P$ to encode the interaction-driven dynamics alone. The operator $P$ thus provides a natural change of variables in which the regularity of the initial data is expected to be largely preserved, in contrast to $Q$, which is subject to the mixing effects of the free evolution. At the level of the Wigner transform, this conjugation is a free-transport change of variables. The conjugation relation~\eqref{conjugationrelation}, together with $\abs{\tfrac{k}{2}+p}^2-\abs{\tfrac{k}{2}-p}^2=2k\cdot p$, gives
\begin{align*}
    \widehat{W[Q]}(t,k,\eta)&=\tpd\int_p e^{-i\hbar\eta\cdot p}\widehat Q\br{t,\tfrac{k}{2}+p,\tfrac{k}{2}-p}\,dp\\
    &=\tpd\int_p e^{-i\hbar(\eta+kt)\cdot p}\widehat P\br{t,\tfrac{k}{2}+p,\tfrac{k}{2}-p}\,dp=\widehat{W[P]}(t,k,\eta+kt).
\end{align*}
Writing $z$ for the spatial variable carried along the free flow and $v$ for the velocity, this is equivalent to
\begin{align*}
    W[P](t,z,v)=W[Q](t,z+tv,v).
\end{align*}
Thus $W[P]$ is $W[Q]$ read along the free characteristics $x=z+tv$, the exact quantum analogue of the classical relation $g(t,z,v)=f(t,z+tv,v)$ between a profile $g$ and a solution $f$ of the Vlasov equation. 

The conjugated perturbation $P$ satisfies the integral equation
\begin{align}\label{Na}
    \widehat P(t,k,p)=\widehat{\Qin}(k,p)+\int_0^t \widehat {\mathrm{L}_P}(s,k,p)\,ds+\int_0^t\widehat{\mathrm{N}_P}(s,k,p)\,ds.
\end{align}
Here $\mathrm{L}_P$ denotes the linearised contribution, arising from the interaction of the perturbation with the homogeneous equilibrium, and is defined by\begin{align*}
    \widehat {\mathrm{L}_P}(s,k,p)&:=-i\hbar^{d-1}(2\pi)^d\widehat w(k+p)\sbr{g(-\hbar p)-g(\hbar k)}e^{i{\frac{\hbar}{2}}(\ak^2-\ap^2)s}\widehat{\rho_Q}(s,k+p).
\end{align*}The term $\mathrm{N}_P$ denotes the nonlinear self-interaction of the perturbation and is defined by
\begin{align*}
    \widehat{\mathrm{N}_P}(s,k,p):=-\frac{i}{\hbar}&\tpd\int_\ell \widehat w(\ell )\widehat{\rho_Q}(s,\ell )\\
    &\times\Big(e^{-i\frac{\hbar}{2}(\abs{k-\ell }^2-\ak^2)s}\widehat{P}(s,k-\ell ,p)-e^{-i\frac{\hbar}{2}(\ap^2-\abs{p-\ell }^2)s}\widehat P(s,k,p-\ell)\Big)\,d\ell.
\end{align*}
By Lemma~\ref{densityofoperator} and the Fourier conjugation relation~\eqref{conjugationrelation}, the density
$\widehat{\rho_Q}$ can be expressed in terms of $P$ as \begin{align}\label{rhoQintermsofP}
    \widehat{\rho_Q}(t,k)&=\tpd\int_p e^{-i\hbar p\cdot kt}\widehat P(t,\mytextfrac{k}{2}+p,\mytextfrac{k}{2}-p)\,dp.
\end{align} Substituting the representation \eqref{Na} into this formula and applying the changes of variables \( p \mapsto p - \ell/2 \) and \( p \mapsto p + \ell/2 \) in the two nonlinear terms leads to
\begin{align}\label{Nb}\begin{split}
    \widehat{\rho_Q}(t,k)&=\widehat{\rho_{{Q_{\mathrm{FS}}}}}(t,k)\\
    &\hspace{0.25cm}-i\hbar^{d-1}\widehat w(k)\int_p\sbr{g\br{\hbar(p-\mytextfrac{k}{2})}-g\br{\hbar(p+\mytextfrac{k}{2})}}\int_0^te^{-i\hbar p\cdot k(t-s)}\widehat{\rho_Q}(s,k)\,dsdp\\
    &\hspace{0.25cm}-\frac{2\hbar^{-1}}{(2\pi)^{2d}}\int_0^t\int_\ell\widehat w(\ell)\widehat{\rho_Q}(s,\ell) \sin\br{\mytextfrac{\hbar}{2}\ell\cdot k(t-s)}\\
    &\hspace{4cm}\times\int_p e^{-i\hbar p\cdot(kt-\ell s)}\widehat P\br{s,\mytextfrac{k-\ell}{2}+p,\mytextfrac{k-\ell}{2}-p}\,dpd\ell ds.
\end{split}
\end{align}The sine factor in the nonlinear term follows from the identity
\[
e^{-i\frac{\hbar}{2}\ell\cdot k(t-s)} - e^{i\frac{\hbar}{2}\ell\cdot k(t-s)} \equiv -2i\sin\left(\mytextfrac{\hbar}{2}\ell\cdot k(t-s)\right).
\]

\subsubsection{Wigner formulation and derivative bound} Taking the Wigner transform of equation~\eqref{Na}, as defined in Definition~\ref{Wignerdefinition}, yields the following evolution equation for $\widehat{W[P]}$: \begin{align}\label{WL}
    \begin{mycases}
     \p_t\widehat{W[P]}(t,k,\eta)=\widehat{W[{\mathrm{L}_P}]}(t,k,\eta)+\widehat{W[{\mathrm{N}_P}]}(t,k,\eta),\\
    \widehat{W[{\mathrm{L}_P}]}(t,k,\eta)=-2\hbar^{-1}\widehat w(k)\widehat{\rho_Q}(t,k)\sin\br{\mytextfrac{\hbar}{2}k\cdot(\eta-kt)}\widehat g(\eta-kt),\\
    \widehat{W[{\mathrm{N}_P}]}(t,k,\eta)=-\frac{2\hbar^{-1}}{(2\pi)^d}\int_\ell \widehat{w}(\ell )\widehat{\rho_Q}(t,\ell )\sin\br{\mytextfrac{\hbar}{2}\ell \cdot(\eta-kt)}\widehat{W[P]}(t,k-\ell ,\eta-\ell t)\,d\ell,
\end{mycases}
\end{align}where $\mathrm{L}_{P}$ and $\mathrm{N}_{P}$ are defined in \eqref{Na}.

Essential to the analysis are bounds on derivatives of the Wigner transform in the \( \eta \) variable, \( D^\alpha_\eta\widehat{W[P]} \), corresponding to velocity weights \( \bm{\xi} \) on \( P \). For multi-indices $\alpha$ with $\abs{\alpha}\geq2$, additional terms of order \( O(\hbar) \) arise from differentiation of the oscillatory phase, in stark contrast to the classical case.
\begin{lemma}\label{etalemma}
Let $\alpha$ be a multi-index. The Wigner transform of ${\mathrm{L}_P}$ satisfies \begin{align*}
    \abs{D^\alpha_\eta\widehat{W[{\mathrm{L}_P}]}(t,k,\eta)}&\lesssim \ak\abs{\widehat w(k)}\abs{\widehat{\rho_Q}(t,k)}\abs{\eta-kt}\abs{D^\alpha_\eta\widehat g(\eta-kt)}\\
    &\hspace{1cm}+\ak\jb{\hbar k}^{\abs{\alpha}-1}\abs{\widehat w(k)}\abs{\widehat{\rho_Q}(t,k)}\sum_{\abs{\beta}\leq \aal-1}\abs{D^\beta_\eta\widehat g(\eta-kt)}.
\end{align*}
The Wigner transform of ${\mathrm{N}_P}$ satisfies
\begin{align*}
    \abs{D^\alpha_\eta\widehat{W[{\mathrm{N}_P}]}(t,k,\eta)}&\lesssim\int_\ell \abs{\ell }\abs{\widehat w(\ell )}\abs{\widehat{\rho_Q}(t,\ell )}\abs{\eta-kt}\abs{D_\eta^\alpha \widehat{W[P]}(t,k-\ell ,\eta-\ell t)}\,d\ell\\
    &\hspace{0.5cm}+\sum_{\abs{\beta}\leq \abs{\alpha}-1}\int_\ell\abs{\ell}\jb{\hbar \ell}^{\abs{\alpha}-1}\abs{\widehat w(\ell )}\abs{\widehat{\rho_Q}(t,\ell )}\abs{D_\eta^\beta \widehat{W[P]}(t,k-\ell ,\eta-\ell t)}\,d\ell.
\end{align*}
When $\abs{\alpha}=0$, sums over $\abs{\beta}\leq\abs{\alpha}-1$ are understood to be empty.
\end{lemma}
\begin{proof}
Applying the Leibniz rule to the linear term gives
    \begin{align*}
        D_\eta^\alpha&\sbr{\sin\br{\mytextfrac{\hbar}{2}k\cdot(\eta-kt)}\widehat g(\eta-kt)}\\
        &\hspace{2cm}=\sum_{\beta\leq \alpha}\begin{pmatrix}
            \alpha\\
            \beta
        \end{pmatrix}D_\eta^{\alpha-\beta}\sbr{\sin\br{\mytextfrac{\hbar}{2}k\cdot(\eta-kt)}}D^\beta_\eta\widehat g(\eta-kt).
    \end{align*}
    For the term with \( \beta = \alpha \), use \( \abs{\sin x} \leq \abs{x} \) to get
    \begin{align*}
        \abs{\sin\br{\mytextfrac{\hbar}{2}k\cdot(\eta-kt)}}\leq \hbar\abs{k}\abs{\eta-kt}.
    \end{align*}
    For terms with \( \beta < \alpha \), differentiate $\sin$ directly and use \( \abs{\sin x},\abs{\cos x} \leq 1 \) to obtain\begin{align*}
         \abs{D^{\alpha-\beta}_{\eta}\sbr{\sin\br{\mytextfrac{\hbar}{2}k\cdot(\eta-kt)}}}\leq (\hbar \abs{k})^{\abs{\alpha-\beta}}.
    \end{align*} This leads to the estimate
    \begin{align*}
          &\abs{D_\eta^\alpha\sbr{\br{e^{-i\frac{\hbar}{2}k\cdot(\eta-kt)}-e^{i\frac{\hbar}{2}k\cdot(\eta-kt)}}\widehat g(\eta-kt)}}\\
          &\hspace{1.5cm}\lesssim \hbar\ak\abs{\eta-kt}\abs{D_\eta^\alpha\widehat g(\eta-kt)}+\sum_{\beta<\alpha}(\hbar\ak)^{\abs{\alpha-\beta}}\abs{D^{\beta}\widehat g(\eta-kt)}.
    \end{align*}Since for any \( \beta < \alpha \), $(\hbar |k|)^{|\alpha - \beta|} \lesssim \hbar |k| \jb{\hbar k}^{|\alpha| - 1}$, this gives the bound for $\widehat{W[{\mathrm{L}_P}]}$. The estimate for \( \widehat{W[{\mathrm{N}_P}]} \) follows similarly.\end{proof}

\subsection{Bootstrap scheme} This subsection introduces the bootstrap scheme used to prove the nonlinear phase-mixing estimates, based on propagating a set of weighted quantum Sobolev norms.
\subsubsection{Bootstrap assumptions}\label{bootstrapassumptionssection}
Fix positive constants $\sigma_4>\sigma_3>\sigma_2>\sigma_1$ such that \begin{align}\label{sigmai}
    \sigma_1\geq d+8,\quad \sigma_2-\sigma_1>3/2,\quad\sigma_3-\sigma_2>d/2,\quad\sigma_4-\sigma_3>(d+1)/2,\quad \sigma_4<N_0-(d+3)/2
\end{align}and constants $K_i\geq1$ determined by the proof. Set $0<\delta\ll1/2$ such that $\sigma_4-\sigma_3>(d+1)/2+\delta$. For each fixed $\hbar\in(0,1]$, standard local well-posedness and persistence of regularity provide a unique maximal self-adjoint solution $Q$ to~\eqref{NH} on an interval $[0,T_{\max})$. Fix $T^*\in(0,T_{\max})$ and suppose that the following bootstrap controls hold on $[0,T^*]$:
\begin{align}
\sum_{\aal\leq M}\norm{\jb{\bm{\nabla}_x,\bm{\nabla}_\xi}^{\sigma_4}\jb{t\bm{\nabla}_x,\bm{\nabla}_\xi}\br{\bm{\xi}^\alpha P(t)}}_{\mathcal L^2}^2&\leq 4K_1\jb{t}^{2d-1}\ep^2,\vphantom{\int_a^b}\label{B1}\tag{B1}\\
     \norm{\jb{\hbar k}\ak^{1/2}\jb{k,kt}^{\sigma_4}\widehat{\rho_Q}(t,k)}_{L^2_{t}([0,T^*],L^2_k)}^2&\leq 4K_2\ep^2,\vphantom{\int_a}\label{B2}\tag{B2}\\
     \sum_{\aal\leq M}\norm{\jb{\bm{\nabla}_x,\bm{\nabla}_\xi}^{\sigma_3}\abs{\bm{\nabla}_x}^\delta \br{\bm{\xi}^\alpha P(t)}}_{\mathcal L^2}^2&\leq 4K_3\ep^2,\vphantom{\int_a^b}\tag{B3}\label{B3}\\
    \norm{\ak^{1/2}\jb{k,kt}^{\sigma_2}\widehat{\rho_{Q}}(t,k)}_{L^\infty_kL^2_{t}([0,T^*])}^2&\leq 4K_4\ep^2,\vphantom{\int_a}\tag{B4}\label{B4}\\
    \jb{k,\eta}^{2\sigma_1}\abs{\widehat{W[P]}(t,k,\eta)}^2&\leq 4K_5\ep^2.\vphantom{\int_a^b}\label{B5}\tag{B5}
\end{align}
\begin{remark}The bootstrap controls \eqref{B1}, \eqref{B3}, and \eqref{B5} are analogous to those used by Bedrossian, Masmoudi and Mouhot~\cite{BedrossianMasmoudiMouhot2018}, but adapted to the quantum setting. A key refinement in the present setting is the inclusion of the weight $\jb{\hbar k}$ inside the norm in \eqref{B2}, which has no analogue in the classical framework. Precisely, the additional norm \begin{align*}
    \norm{\hbar k\ak^{1/2}\jb{k,kt}^{\sigma_4}\widehat{\rho_Q}(t,k)}_{L^2_{t}([0,T^*],L^2_k)}^2\leq 4K_2\ep^2
\end{align*} is propagated, reflecting a structural difference between the Hartree and Vlasov equations. While the quantum setting allows control of one more derivative, uniformity in the semiclassical limit requires pairing it with a factor of $\hbar$. This refinement enables the result in Theorem~\ref{maintheorem} to cover the screened Coulomb kernel. As $\hbar\to0$, the norm in \eqref{B2} reduces to the classical counterpart used in~\cite{BedrossianMasmoudiMouhot2018}, and the norm in \eqref{B4} already coincides with its classical analogue.
\end{remark}
\begin{remark}The decay statement \eqref{T1} in Theorem~\ref{maintheorem} follows solely from the control \eqref{B5}, and the scattering statement \eqref{T2} follows from \eqref{B3} and \eqref{B5}, as shown in Lemmas~\ref{rhodecayfromB5} and~\ref{scatteringlemma}.
\end{remark}

\subsubsection{Justification of the bootstrap norms}
The bootstrap assumptions \eqref{B1}, \eqref{B3}, and \eqref{B5} are initially satisfied under the smallness condition \( \norm{\Qin}_{\mathcal H^{N_0}_{M,M}} \leq \ep \) in Theorem~\ref{maintheorem}. For instance, at \( t = 0 \), assumption \eqref{B5} satisfies
\begin{align*}\jb{k,\eta}^{2\sigma_1}\abs{\widehat{W[\Qin]}(k,\eta)}^2=\abs{\reallywidehat{W[\jb{\bm{\nabla}_x,\bm{\nabla}_\xi}^{\sigma_1}\Qin]}(k,\eta)}^2\lesssim\norm{\Qin}_{\mathcal H^{\sigma_1}_{M,M}}^2\leq\ep^2,
\end{align*}where Lemma~\ref{FTofWignertransform} has been used. The density estimates \eqref{B2} and \eqref{B4} are trivially satisfied at \( t = 0 \), but their long-time structure is less obvious. The next lemma shows that these norms are natural: for the free Schrödinger evolution, they are directly controlled by the initial data \( \Qin \). This motivates their appearance in the bootstrap framework.

\begin{lemma}\label{b2b4arenatural}
    Fix $\sigma>0$ and let $\rho_{Q_{\mathrm{FS}}}$ be the density of the solution to the free Schrödinger equation \eqref{FS}, defined in \eqref{FSsol}. Then, for $s_1>(d+1)/2$ and $s_2>1/2$,
    \begin{align*}
        \norm{\ak^{1/2}\jb{k,kt}^{\sigma}\widehat{\rho_{Q_{\mathrm{FS}}}}(t,k)}_{L^2_{t}([0,\infty),L^2_k)}&\lesssim_{s_1}\norm{\Qin}_{{\mathcal H^{\sigma+s_1}_{M,M}}}\\
        \norm{\ak^{1/2}\jb{k,kt}^{\sigma}\widehat{\rho_{Q_{\mathrm{FS}}}}(t,k)}_{L^\infty_kL^2_{t}([0,\infty))}&\lesssim_{s_2}\norm{\Qin}_{{\mathcal H^{\sigma+s_2}_{M,M}}}.
    \end{align*}
\end{lemma}
\begin{proof}
    First recall that $\widehat{\rho_{Q_{\mathrm{FS}}}}(t,k)=\widehat{W[\Qin]}(k,kt)$. 
    \begin{align*}
        \norm{\ak^{1/2}\jb{k,kt}^{\sigma}\widehat{\rho_{Q_{\mathrm{FS}}}}(t,k)}_{L^2_{t}([0,\infty),L^2_k)}&=\br{\int_{t=0}^{\infty}\int_k\ak\jb{k,kt}^{2\sigma}\abs{\widehat{W[\Qin]}(k,kt)}^2\,dkdt}^{1/2}\\
        &\leq\sup_{k,\eta}\abs{\jb{k,\eta}^{\sigma+s_1}\widehat{W[\Qin]}(k,\eta)}\br{\int_k\int_{t=0}^{\infty}\frac{\ak}{\jb{k,kt}^{2s_1}}\,dtdk}^{1/2}.
    \end{align*}
Then, by making the change of variables $t\mapsto t/\ak$, and denoting $\overline{\ep}=s_1-\frac{d+1}{2},$ 
\begin{align*}
    \int_k\int_{t=0}^{\infty}\frac{\ak}{\jb{k,kt}^{2s_1}}\,dtdk=\int_k\int_{t=0}^\infty \frac{1}{\jb{k,t}^{d+1+2\overline{\ep}}}\,dtdk\lesssim\int_k\frac{1}{\jb{k}^{d+\overline{\ep}}}\,dk\int_\R\frac{1}{\jb{t}^{1+\overline{\ep}}}\,dt\lesssim1.
\end{align*}The bound then follows by applying Lemma~\ref{FTofWignertransform}. The second estimate is proved similarly.
\end{proof}

\subsubsection{Bootstrap proposition and strategy}\label{conclusionsection}
The following proposition forms the backbone of the nonlinear analysis. The inequalities \eqref{I1}--\eqref{I5} below are referred to as the \textit{bootstrap improvements}: they coincide with the bootstrap controls \eqref{B1}--\eqref{B5}, but with each constant $4K_i$ halved to $2K_i$. The proposition thus asserts that, on any interval on which the controls hold, the solution in fact obeys strictly stronger bounds. By continuity, these strict improvements allow the bootstrap controls to be propagated throughout the maximal lifespan. The continuation argument is given at the end of the proof.
\begin{proposition}[Bootstrap]\label{bootstrapprop}
Let the assumptions on $w$, $g$ and $\Qin$ be as in Theorem~\ref{maintheorem}. There exist $\ep_0=\ep_0(N_0,w,g,{c_{\mathrm{P}}},d,\delta_0,\sigma_i)$ and $K_i=K_i(N_0,w,g,{c_{\mathrm{P}}},d,\delta_0,\sigma_i)$ such that, for every $T^*\in(0,T_{\max})$, if the bootstrap controls \eqref{B1}--\eqref{B5} hold on $[0,T^*]$ and $\ep<\ep_0$, then 
\begin{align}
\sum_{\aal\leq M}\norm{\jb{\bm{\nabla}_x,\bm{\nabla}_\xi}^{\sigma_4}\jb{t\bm{\nabla}_x,\bm{\nabla}_\xi}\br{\bm{\xi}^\alpha P(t)}}_{\mathcal L^2}^2&\leq 2K_1\jb{t}^{2d-1}\ep^2,\vphantom{\int_a^b}\tag{I1}\label{I1}\\
     \norm{\jb{\hbar k}\ak^{1/2}\jb{k,kt}^{\sigma_4}\widehat{\rho_Q}(t,k)}_{L^2_{t}([0,T^*],L^2_k)}^2&\leq 2K_2\ep^2,\vphantom{\int_a}\tag{I2}\label{I2}\\
     \sum_{\aal\leq M}\norm{\jb{\bm{\nabla}_x,\bm{\nabla}_\xi}^{\sigma_3}\abs{\bm{\nabla}_x}^\delta \br{\bm{\xi}^\alpha P(t)}}_{\mathcal L^2}^2&\leq 2K_3\ep^2,\vphantom{\int_a^b}\tag{I3}\label{I3}\\
    \norm{\ak^{1/2}\jb{k,kt}^{\sigma_2}\widehat{\rho_{Q}}(t,k)}_{L^\infty_kL^2_{t}([0,T^*])}^2&\leq 2K_4\ep^2,\vphantom{\int_a}\tag{I4}\label{I4}\\
    \jb{k,\eta}^{2\sigma_1}\abs{\widehat{W[P]}(t,k,\eta)}^2&\leq 2K_5\ep^2,\vphantom{\int_a^b}\tag{I5}\label{I5}
\end{align}for all $t\in[0,T^*]$. Consequently, the maximal solution is global and the bootstrap controls hold for all $t\geq0$.
\end{proposition}

\begin{remark}
Among the improvements, \eqref{I2} presents the main analytical difficulty, as it permits neither regularity loss nor time growth. Precise control is required over the nonlinear interaction term in \eqref{Nb}, which includes a bilinear contribution in the perturbation and allows earlier density modes $\widehat{\rho_Q}(s,\ell)$ to influence the evolution of $\widehat{\rho_Q}(t,k)$. In certain regimes, this interaction may accumulate and produce \textit{resonances} at large times, see e.g.~\cite{Bedrossian2021Echoes}. In the Hartree equation on $\R^d$ with $d\geq 3$ and sufficiently regular $w$, such resonances are controlled by a dispersive mechanism. The mechanism relies on separating the nearly collinear interactions, for which
\begin{align*}
    \abs{\ell_\perp}\lesssim\frac{\ak^{b}}{\jb{s}^{\zeta}},\qquad \ell_\perp:=\ell-\frac{k\cdot\ell}{\ak^2}k,
\end{align*}
from the non-resonant interactions, for suitable exponents $b,\zeta\in(0,1)$. The resonant set is a shrinking cylinder around the line spanned by $k$, whose decreasing cross-section provides sufficient smallness. On its complement, the bound $\abs{kt-\ell s}\geq s\abs{\ell_\perp}$ yields additional decay. The resulting time-response kernel matches that arising in the Vlasov--Poisson system and is handled similarly in Section~\ref{resonancessection}.
\end{remark}

\begin{figure}[ht]
\vspace{0cm}
\centering
\begin{tikzpicture}[domain=-5.8:3.6,xscale=0.57,yscale=0.27,samples=150]
\node [left] at (-5,0) {$\mathrm{(B2)}$};
\node [above] at (-1.55,4.76) {$\mathrm{(B1)}$};
\node [below] at (-1.55,-4.76) {$\mathrm{(B5)}$};
\node [above] at (4.04,2.94) {$\mathrm{(B3)}$};
\node [below] at (4.04,-2.94) {$\mathrm{(B4)}$};
\draw [<->] (-5,+0.5)--(-1.55-0.75,4.76+0.6); 
\draw [<->] (-1.55+0.8,4.76+0.9)--(4.04-0.8,2.94+1.1); 
\draw [<->] (-1.55+0.4,-4.76+0.25)--(4.04-0.5,2.94); 
\draw [<->] (-1.55+0.7,-4.76-0.8)--(4.04-0.8,-2.94-1.1); 
\draw [->] (-1.55-0.75,-4.76-0.5)--(-5,0-0.6); 
\draw [->] (-1.55,-4.76+0.25)--(-1.55,4.76); 
\draw [->] (4.04,2.94)--(4.04,-2.94); 
\draw [<->] plot (\x,{-0.3*(\x+1.2+5)*(\x+1.2-6.7)}); 
\draw [->] plot (\x,{0.3*(\x+1.1+5)*(\x+1.1-6.5)}); 
\end{tikzpicture}
\captionsetup{width=.7\linewidth}
\caption{Bootstrap dependencies: each arrow indicates an assumption used to establish an improvement, e.g.\ \eqref{B1} $\rightarrow$ \eqref{B2} reflects the use of \eqref{B1} in proving \eqref{I2}.}

\label{bootstrapfigure}
\end{figure}

The proposition is proved below using the estimates established in Sections~\ref{sectiondensity} and~\ref{sectiondensitymatrix}.

\begin{proof}[Proof of Proposition~\ref{bootstrapprop}]
Fix dimension $d\geq 3$, constants $\sigma_0$, $\sigma_1$, $N_0$, kernel $w$, steady state $g$ and initial perturbation $\Qin$ as in Theorem~\ref{maintheorem} and choose parameters $\sigma_2$, $\sigma_3$, $\sigma_4$ satisfying \eqref{sigmai}. Section~\ref{sectiondensity} establishes that the constants $K_2$ and $K_4$ appearing in the bootstrap assumptions \eqref{B2} and \eqref{B4} may be chosen depending only on the fixed setup. Section~\ref{sectiondensitymatrix} then shows that the remaining constants $K_1$, $K_3$, and $K_5$, can also be selected depending only on the setup and on $K_2, K_4$. If $\ep_0$ satisfies
\begin{align*}
    0<\ep_0\leq\min\set{\frac{1}{\sqrt{K_3K_5+K_2K_5}},\frac{1}{\sqrt{K_4K_5}+\sqrt{K_3K_5}},\frac{1}{\sqrt{K_2K_3}+\sqrt{K_1K_5}+\sqrt{K_3K_5}}},
\end{align*}
then, for any $0<\ep\leq \ep_0$, all bootstrap improvements \eqref{I1}-\eqref{I5} hold by Sections \ref{sectionb2}, \ref{sectionb4}, \ref{sectionb5}, \ref{sectionb1} and \ref{sectionb3}.

The proof is concluded by a standard continuity and continuation argument. Local well-posedness and the initial estimates imply that the bootstrap controls hold for sufficiently small times. Proposition~\ref{bootstrapprop} gives strict improvements on every interval on which the controls hold. By continuity, the bootstrap estimates therefore persist throughout $[0,T_{\max})$. In particular, \eqref{I1} bounds the lower-order Sobolev norm entering the local continuation criterion on every finite interval. Hence $T_{\max}=\infty$. Persistence of regularity retains the regularity required above on every finite time interval.
\end{proof}

\subsection{Deduction of Theorem~\ref{maintheorem}}This section completes the proof of Theorem~\ref{maintheorem} by showing how the decay and scattering statements \eqref{T1} and \eqref{T2} follow from the bootstrap estimates.

The decay estimate on the density \eqref{T1} follows directly from the bootstrap control \eqref{B5}.
\begin{lemma}\label{rhodecayfromB5}
The global bootstrap control \eqref{B5} implies that, for all $t\geq0$ and $k\in\R^d$,
\begin{align*}
    \abs{\widehat{\rho_Q}(t,k)}\leq\frac{2\sqrt{K_5}\,\ep}{\jb{k,kt}^{\sigma_1}}.
\end{align*}
In particular, Proposition~\ref{bootstrapprop} implies the damping estimate \eqref{T1} in Theorem~\ref{maintheorem}.
\end{lemma}
\begin{proof}
Setting $\eta=0$ in the Fourier transform of the Wigner transform gives $\widehat{W[Q]}(t,k,0)=\widehat{\rho_Q}(t,k)$, while the conjugation identity of Section~\ref{conjugationSection} gives $\widehat{W[Q]}(t,k,0)=\widehat{W[P]}(t,k,kt)$. Hence
\begin{align*}
    \abs{\widehat{\rho_Q}(t,k)}=\abs{\widehat{W[P]}(t,k,kt)}\leq\frac{2\sqrt{K_5}\,\ep}{\jb{k,kt}^{\sigma_1}},
\end{align*}
by bootstrap assumption \eqref{B5} evaluated at $\eta=kt$. Since Proposition~\ref{bootstrapprop} shows that the bootstrap control~\eqref{B5} holds globally, the estimate~\eqref{T1} follows with $C=2\sqrt{K_5}$.
\end{proof}

Bootstrap assumptions \eqref{B3} and \eqref{B5} suffice to deduce the scattering statement \eqref{T2}.
\begin{lemma}\label{scatteringlemma}
    Proposition \ref{bootstrapprop} implies the scattering statement \eqref{T2} in Theorem~\ref{maintheorem}.
\end{lemma}
\begin{proof}Let $\aal\leq M$. To prove convergence of $P(t)$ in $\mathcal{H}^{\sigma_0}_M$, the Duhamel formula \eqref{Na} is analysed in Fourier--Wigner variables. The aim is to show that both the linear and nonlinear contributions decay sufficiently fast to be absolutely integrable in time in the $\mathcal H^{\sigma_0}_M$ norm, yielding the existence of an asymptotic state $Q_\infty := \lim_{t \to \infty} P(t)$.

For the linear term, apply Lemma~\ref{L2normofWigner} and Plancherel's theorem to obtain
\begin{align*}
    \norm{\jb{\bm{\nabla}_x,\bm{\nabla}_{\xi}}^{\sigma_0}(\bm{\xi}^\alpha {\mathrm{L}_P}(t))}^2_{\mathcal L^2} \leq \norm{\jb{k,\eta}^{\sigma_0} D^\alpha_\eta(\widehat{W[{\mathrm{L}_P}]}(t))}^2_{L^2_{k,\eta}}.
\end{align*}
Lemma~\ref{etalemma} is used to extract the highest-order $D^\alpha_\eta$ derivative falling on $\widehat{W[{\mathrm{L}_P}]}$. Using assumption \eqref{wassumption} on $w$, which gives $\jb{\hbar k}^{\aal - 1} \abs{\widehat{w}(k)} \lesssim 1$, and the inequality $\jb{k,\eta}^{\sigma_0} \lesssim \jb{k,kt}^{\sigma_0} \jb{\eta - kt}^{\sigma_0}$, one obtains
\begin{align*}
    &\norm{\jb{k,\eta}^{\sigma_0} D^\alpha_\eta(\widehat{W[{\mathrm{L}_P}]}(t))}^2_{L^2_{k,\eta}} \\
    &\lesssim \int_{k,\eta} \abs{\widehat{w}(k)}^2 \ak^2 \jb{k,kt}^{2\sigma_0} \abs{\widehat{\rho_Q}(t,k)}^2 \jb{\eta - kt}^{2\sigma_0} \abs{\eta - kt}^2 \abs{D^\alpha_\eta \widehat{g}(\eta - kt)}^2 \,dk\,d\eta \\
    &\hspace{0.5cm} + \sum_{\abs{\beta} \leq \aal - 1} \int_{k,\eta} \jb{\hbar k}^{2\aal - 2} \abs{\widehat{w}(k)}^2 \ak^2 \jb{k,kt}^{2\sigma_0} \abs{\widehat{\rho_Q}(t,k)}^2 \jb{\eta - kt}^{2\sigma_0} \abs{D^\beta_\eta \widehat{g}(\eta - kt)}^2 \,dk\,d\eta \\
    &\lesssim \norm{g}_{H^{\sigma_0 + 1}_M}^2 \int_k \abs{k}^2 \jb{k,kt}^{2\sigma_0} \abs{\widehat{\rho_Q}(t,k)}^2 \,dk \\
    &\lesssim K_5 \ep^2 \norm{g}_{H^{\sigma_0 + 1}_M}^2 \int_k \frac{\abs{k}^2}{\jb{k,kt}^{2\sigma_1 - 2\sigma_0}} \,dk,
\end{align*}
where the final bound follows from the bootstrap estimate \eqref{B5} (via Lemma \ref{rhodecayfromB5}). Provided $2\sigma_1-2\sigma_0>d+2$, this integral is uniformly bounded in time after the change of variables $k\mapsto k/t$, leading to
\begin{align*}
     \norm{\jb{\bm{\nabla}_x,\bm{\nabla}_{\xi}}^{\sigma_0}(\bm{\xi}^\alpha {\mathrm{L}_P}(t))}_{\mathcal L^2}\lesssim\frac{\sqrt{K_5}\ep}{\jb{t}^{d/2+1}}.
\end{align*}

The nonlinear term is handled similarly. Using the inequalities $\jb{k,\eta}^{\sigma_0}\lesssim\jb{\ell ,\ell t}^{\sigma_0}\jb{k-\ell ,\eta-\ell t}^{\sigma_0}$ and $\abs{\eta-kt}\lesssim\jb{t}\jb{k-\ell ,\eta-\ell t}$, one obtains
\begin{align*}
&\norm{\jb{k,\eta}^{\sigma_0}D^\alpha_\eta(\widehat {W[{\mathrm{N}_P}]}(t))}^2_{L^2_{k,\eta}}\\
&\lesssim\int_{k,\eta}\br{\jb{k,\eta}^{\sigma_0}\int_\ell \abs{\ell }\abs{\widehat w(\ell )}\abs{\widehat{\rho_Q}(t,\ell )}\abs{\eta-kt}\abs{D_\eta^\alpha \widehat{W[P]}(t,k-\ell ,\eta-\ell t)}\,d\ell }^2\,dkd\eta\\
    &\hspace{0.5cm}+\sum_{\abs{\beta}\leq \aal-1}\int_{k,\eta}\br{\jb{k,\eta}^{\sigma_0}\int_\ell \jb{\hbar\ell }^{\aal-1}\abs{\ell }\abs{\widehat w(\ell )}\abs{\widehat{\rho_Q}(t,\ell )}\abs{D_\eta^\beta\widehat{W[P]}(t,k-\ell ,\eta-\ell t)}\,d\ell }^2\,dkd\eta\\
    &\lesssim K_5\ep^2\jb{t}^2\sum_{\abs{\beta}\leq M}\int_{k,\eta}\br{\int_\ell \frac{\abs{\ell }}{\jb{\ell ,\ell t}^{\sigma_1-\sigma_0}}\jb{k-\ell ,\eta-\ell t}^{\sigma_0+1}\abs{D_\eta^\beta\widehat{W[P]}(t,k-\ell ,\eta-\ell t)}\,d\ell }^2\,dkd\eta.
\end{align*}
By Lemma~\ref{operatorsWigner}, multipliers in the Fourier variables $(k,\eta)$ and derivatives in $\eta$ correspond, respectively, to the quantum derivatives $\jb{\bm{\nabla}_x,\bm{\nabla}_\xi}$ and the velocity weights $\bm{\xi}$ acting on $P$. In particular, 
\begin{align*}
    \jb{k-\ell ,\eta-\ell t}^{\sigma_0+1}D_\eta^\beta\widehat{W[P]}(t,k-\ell ,\eta-\ell t)
    =(-i)^{\abs{\beta}}\reallywidehat{W[\jb{\bm{\nabla}_x,\bm{\nabla}_\xi}^{\sigma_0+1}(\bm{\xi}^\beta P)]}(t,k-\ell ,\eta-\ell t).
\end{align*}
In order to apply bootstrap assumption \eqref{B3}, Cauchy--Schwarz in $\ell$ is applied with weight $\abs{k-\ell}^{\delta}$, for $0<\delta\ll1/2$ as introduced in Section~\ref{bootstrapassumptionssection}, yielding\begin{align*}
    &\br{\int_\ell \frac{\abs{\ell }}{\jb{\ell ,\ell t}^{\sigma_1-\sigma_0}}\abs{\reallywidehat{W[\jb{\bm{\nabla}_x,\bm{\nabla}_\xi}^{\sigma_0+1}(\bm{\xi}^\beta P)]}(t,k-\ell ,\eta-\ell t)}\,d\ell}^{2} \\
    &\leq \sbr{\int_\ell \frac{1}{\abs{k-\ell }^{2\delta}}\frac{\abs{\ell }}{\jb{\ell ,\ell t}^{\sigma_1-\sigma_0}}\,d\ell }\int_\ell \frac{\abs{\ell }}{\jb{\ell ,\ell t}^{\sigma_1-\sigma_0}}\abs{k-\ell }^{2\delta}\abs{\reallywidehat{W[\jb{\bm{\nabla}_x,\bm{\nabla}_\xi}^{\sigma_0+1}(\bm{\xi}^\beta P)]}(t,k-\ell ,\eta-\ell t)}^2d\ell.
\end{align*}
To bound the first factor by its supremum over $k$, set $q:=\jb{t}k$. The change of variables $\ell\mapsto\ell/\jb{t}$ sends $\jb{\ell,\ell t}$ to $\jb{\ell}$, so that
\begin{align}\label{beforebound}
    \int_\ell \frac{1}{\abs{k-\ell }^{2\delta}}\frac{\abs{\ell }}{\jb{\ell ,\ell t}^{\sigma_1-\sigma_0}}\,d\ell=\frac{1}{\jb{t}^{d+1-2\delta}}\int_\ell \frac{1}{\abs{q-\ell}^{2\delta}}\frac{\abs{\ell }}{\jb{\ell}^{\sigma_1-\sigma_0}}\,d\ell.
\end{align}
Splitting the remaining integral into the regions $\abs{q-\ell}\geq1$, where $\abs{q-\ell}^{-2\delta}\leq1$, and $\abs{q-\ell}<1$, where the singularity is integrable as $2\delta<d$, shows that
\begin{align*}
    \sup_{k}\int_\ell \frac{1}{\abs{k-\ell }^{2\delta}}\frac{\abs{\ell }}{\jb{\ell ,\ell t}^{\sigma_1-\sigma_0}}\,d\ell\lesssim\frac{1}{\jb{t}^{d+1-2\delta}},
\end{align*}provided $\sigma_1-\sigma_0>d+1$.
Inserting this bound yields
\begin{align*}
   &\norm{\jb{k,\eta}^{\sigma_0}D^\alpha_\eta(\widehat {W[{\mathrm{N}_P}]}(t))}^2_{L^2_{k,\eta}}\\
   &\lesssim K_5\ep^2\jb{t}^2\frac{1}{\jb{t}^{d+1-2\delta}}\int_\ell \frac{\abs{\ell }}{\jb{\ell ,\ell t}^{\sigma_1-\sigma_0}}\,d\ell\sum_{\abs{\beta}\leq {M}}\norm{\reallywidehat{W[\abs{\bm{\nabla}_x}^\delta\jb{\bm{\nabla}_x,\bm{\nabla}_\xi}^{\sigma_0+1}(\bm{\xi}^\beta P)]}(t,k ,\eta)}^2_{L^2_{k,\eta}}\\
   &\lesssim K_5\ep^2\jb{t}^2\frac{1}{\jb{t}^{d+1-2\delta}}\frac{1}{\jb{t}^{d+1}}K_3\ep^2,
\end{align*}where the final step uses bootstrap assumption \eqref{B3}, together with Plancherel and Lemma~\ref{L2normofWigner}. Hence, for the nonlinear term,
\begin{align*}
    \norm{\jb{\bm{\nabla}_x,\bm{\nabla}_{\xi}}^{\sigma_0}(\bm{\xi}^\alpha {\mathrm{N}_P}(t))}_{\mathcal L^2}\lesssim\frac{\sqrt{K_5K_3}\ep^2}{\jb{t}^{d-\delta}}.
\end{align*}

Consequently, define
\begin{align*}
    Q_\infty:=Q_{\mathrm{in}}+\int_0^\infty \mathrm{L}_P(s)\,ds+\int_0^\infty \mathrm{N}_P(s)\,ds\in \mathcal H^{\sigma_0}_M.
\end{align*}
Inequality \eqref{T2} then follows from
\begin{align*}
    \norm{P(t)-Q_\infty}_{\mathcal H^{\sigma_0}_M}&\lesssim\sum_{\aal\leq M}\int_{t}^\infty \norm{\jb{\bm{\nabla}_x,\bm{\nabla}_{\xi}}^{\sigma_0}(\bm{\xi}^\alpha {\mathrm{L}_P}(s))}_{\mathcal L^2_{x,y}}\,ds+\sum_{\aal\leq M}\int_{t}^\infty \norm{\jb{\bm{\nabla}_x,\bm{\nabla}_{\xi}}^{\sigma_0}(\bm{\xi}^\alpha {\mathrm{N}_P}(s))}_{\mathcal L^2_{x,y}}\,ds\\
    &\lesssim\sqrt{K_5}\ep\int_{t}^\infty\frac{1}{\jb{s}^{d/2+1}}\,ds+\sqrt{K_5 K_3}\ep^2\int_t^\infty \frac{1}{\jb{s}^{d-\delta}}\,ds\\
    &\lesssim\frac{\ep}{\jb{t}^{d/2}},
\end{align*}as the contribution from the nonlinear term satisfies
\[
\int_t^\infty \jb{s}^{-(d - \delta)}\,ds \lesssim \jb{t}^{-d + \delta+1} \leq \jb{t}^{-d/2}
\]for $d\geq 3$ and $0<\delta\ll 1/2$, ensuring the overall decay rate $\jb{t}^{-d/2}$.
\end{proof}

This completes the proof of the scattering statement \eqref{T2}, and hence of Theorem~\ref{maintheorem}, subject to the bootstrap estimates established in the following sections.

\section{Linearised phase mixing}\label{linearisedsection}
The \textit{linearised Hartree equation} is obtained by removing the quadratic-in-$Q$ term $\sbr{w*\rho_Q,Q}$ from \eqref{NH}, which is expected to be small:\begin{align}
    \begin{mycases}\label{LH}
    i\hbar \p_t{Q_{\mathrm{L}}}=\sbr{-{\frac{\hbar^2}{2}}\Delta,{Q_{\mathrm{L}}}}+\sbr{w*\rho_{{Q_{\mathrm{L}}}},\gamma_g},\\
    {Q_{\mathrm{L}}}(0)=Q_{\mathrm{in}}.
\end{mycases}
\end{align}

\subsection{Fourier--Laplace transform}
Writing \eqref{LH} in kernel form gives
\begin{align*}
    i\hbar\p_t{Q_{\mathrm{L}}}(t,x,y)=-{\frac{\hbar^2}{2}}\br{\Delta_x-\Delta_y}{Q_{\mathrm{L}}}(t,x,y)+\br{w*\rho_{Q_{\mathrm{L}}}(t,x)-w*\rho_{Q_{\mathrm{L}}}(t,y)}\widehat g\br{\frac{y-x}{\hbar}}.
\end{align*}
By Duhamel's formula, the solution of the linearised Hartree equation can be written as\begin{equation}\begin{split}
    {Q_{\mathrm{L}}}(t,x,y)&=e^{i{\frac{\hbar}{2}} (\Delta_x-\Delta_y)t}Q_{\mathrm{in}}(x,y)\\
    &\hspace{0.5cm}-
    \frac{i}{\hbar}\int_0^te^{i{\frac{\hbar}{2}}(\Delta_x-\Delta_y)(t-s)}\sbr{w*\rho_{Q_{\mathrm{L}}}(s,x)-w*\rho_{Q_{\mathrm{L}}}(s,y)}\widehat g\br{\frac{y-x}{\hbar}}\,ds.\end{split}\label{linearHartreeimplicit}
\end{equation}
Taking the Fourier transform of \eqref{linearHartreeimplicit} yields \begin{equation}\begin{split}\label{Fouriertransform}
    \widehat {Q_{\mathrm{L}}}(t,k,p)&=e^{-i{\frac{\hbar}{2}}(\ak^2-\ap^2)t}\widehat{Q_{\mathrm{in}}}(k,p)\\
    &\hspace{0.5cm}-i\hbar^{d-1}(2\pi)^d\widehat w(k+p)\sbr{g(-\hbar p)-g(\hbar k)}\int_0^te^{-i{\frac{\hbar}{2}}(\ak^2-\ap^2)(t-s)}\widehat{\rho_{Q_{\mathrm{L}}}}(s,k+p)\,ds.\end{split}
\end{equation}
Using Lemma~\ref{densityofoperator}, the corresponding expression for the density is
\begin{align*}
\widehat{\rho_{Q_{\mathrm{L}}}}(t,k)&=\frac{1}{(2\pi)^d}\int_p e^{-i\hbar  kt\cdot p}\widehat{Q_{\mathrm{in}}}\br{\mytextfrac{k}{2}+p,\mytextfrac{k}{2}-p}\,dp\\
    &\hspace{0.6cm}-i\hbar^{d-1}\widehat w(k)\int_p\int_0^t e^{-i\hbar (t-s)k\cdot p}\sbr{g\br{\hbar (p-\mytextfrac{k}{2})}-g\br{\hbar (p+\mytextfrac{k}{2})}}\widehat{\rho_{Q_{\mathrm{L}}}}(s,k)\,dsdp.
\end{align*}This is a Volterra equation of the form\begin{align}\label{linearHartreeFourierimplicit}
      \widehat{\rho_{Q_{\mathrm{L}}}}(t,k)&=\widehat{\rho_{{Q_{\mathrm{FS}}}}}(t,k)-\widehat {\mathscr L}*_{t}\widehat{\rho_{Q_{\mathrm{L}}}}(t,k),
\end{align}where $\rho_{{Q_{\mathrm{FS}}}}(t,k)$ is the density of the free Schrödinger equation \eqref{FS}, and the kernel $\widehat {\mathscr L}$ is given by\begin{align*}\begin{split}
    \widehat {\mathscr L}(t,k)&:= i\hbar^{d-1}\widehat w(k)\int_p e^{-i\hbar kt\cdot p}\sbr{g\br{\hbar (p-\mytextfrac{k}{2})}-g\br{\hbar (p+\mytextfrac{k}{2})}}\,dp\\
    &= i\hbar^{d-1}\widehat w(k)\mathcal F_p\sbr{g\br{\hbar (p-\mytextfrac{k}{2})}-g\br{\hbar (p+\mytextfrac{k}{2})}}(\hbar kt)\\
    &= \frac{i}{\hbar}\widehat w(k)\br{e^{-i\frac{\hbar}{2} \ak^2 t}-e^{i\frac{\hbar}{2} \ak^2 t}}\widehat g\br{kt}\\
    &=\frac{2}{\hbar}\widehat w(k)\sin(\mytextfrac{1}{2} \hbar t\ak^2)\widehat g\br{kt}.\end{split}
\end{align*}Observe that since $w,g\in L^1(\R^d)$, the zero mode vanishes: $\widehat {\mathscr L}(t,k=0)=0$ for all $t\geq0$.

In order to find an explicit equation for $\widehat{\rho_{Q_{\mathrm{L}}}}$, the Laplace transform in time is applied to \eqref{linearHartreeFourierimplicit}. Denoting, for $\lam\in\C$ with $\Re\lam>0$, $\widetilde f(\lam,k):=\int_0^\infty e^{-\lam t}\widehat f(t,k)\,dt$, the transformed equation reads \begin{align}\label{beforedivision}
    \br{1+\widetilde {\mathscr L}(\lam,k)}\widetilde{\rho_{Q_{\mathrm{L}}}}(\lam,k)=\widetilde{\rho_{{Q_{\mathrm{FS}}}}}(\lam,k),
\end{align}where\begin{align}\label{firstL}
    \widetilde {\mathscr L}(\lam,k)&=\frac{2}{\hbar}\widehat w(k)\int_0^\infty e^{-\lam t}\mathrm{sin}(\tfrac{1}{2}\hbar t\ak^2)\widehat g(kt)\,dt.\end{align}

\subsection{Uniform Penrose condition}\label{Penroseconditionsection}
This section formulates conditions on the interaction kernel \( w \) and the steady state \( g \) under which the uniform Penrose condition \eqref{Penrose} holds.

The following lemma shows that, in the space $H^\sigma_M$, moments and derivatives can be reordered freely, allowing them to be applied in whichever order is most convenient.
\begin{lemma}[Ordering of moments]\label{lemmaHMnorm}
Let $\sigma\geq0$, let $M\in\mathbb N$, and $g\in H^\sigma_M(\R^d)$. Then the norm\begin{align*}
    \norm{g}_{H^\sigma_M}:=\sum_{\abs{\alpha}\leq M}\norm{\jb{\nabla_{k}}^\sigma (k^\alpha g)}_{L^2}\sim\sum_{\abs{\alpha}\leq M}\norm{k^\alpha\jb{\nabla_{k}}^\sigma g}_{L^2}.
\end{align*}
\end{lemma}

The Sobolev $L^2$ trace lemma is also recalled. 
 \begin{lemma}[$L^2$ trace {\cite[Lemma~2.15]{BedrossianMasmoudiMouhot2018}}] \label{L2trace}
        Let $g\in H^s(\R^d)$ with $s>\frac{d-1}{2}$ and $C\subset\R^d$ be an arbitrary straight line. Then, there holds\begin{align*}
            \norm{g}_{L^2(C)}\lesssim_{s,d}\norm{g}_{H^s(\R^d)}.
        \end{align*}
    \end{lemma}
These lemmas imply the following estimate for directional integrals of derivatives of \( \widehat{g} \):
\begin{lemma}\label{gnorm}
Let \( n \geq 0 \), and let \( \alpha \) be a multi-index. Then
\begin{align*}
        \int_0^\infty\jb{s}^{n}\abs{\p^\alpha_k\widehat g\br{\mytextfrac{k}{\abs{k}}s}}\,ds\lesssim_{d,
        \delta_0,n,\aal}\norm{g}_{H^{n+{{1/2}}+\delta_0}_{\aal+\ceils{d/2}}}.\end{align*}
\end{lemma}
\begin{proof}
    Applying Cauchy--Schwarz followed by Lemma~\ref{L2trace},\begin{align*}
         \int_0^\infty\jb{s}^{n}\abs{\p^\alpha_k\widehat g\br{\mytextfrac{k}{\abs{k}}s}}\,ds&\lesssim_{\delta_0}\br{\int_0^\infty\jb{s}^{2n+1+2\delta_0}\abs{\p^\alpha_k\widehat g\br{\mytextfrac{k}{\abs{k}}s}}^2\,ds}^{{1/2}}\\
         &\lesssim_{d}\br{\intRdk\abs{\jb{\nabla_k}^{\ceils{d/2}}\br{\jb{k}^{n+{{1/2}}+\delta_0}\p^\alpha_k\widehat g(k)}}^2\,dk}^{{1/2}}.
    \end{align*}
    Then, by Plancherel and Lemma~\ref{lemmaHMnorm}:\begin{align*}
        \int_0^\infty\jb{s}^{n}\abs{\p^\alpha_k\widehat g\br{\mytextfrac{k}{\abs{k}}s}}\,ds&\lesssim_d\norm{\jb{x}^{\ceils{d/2}}\jb{\nabla_x}^{n+{{1/2}}+\delta_0}\br{x^\alpha g(x)}}_{L^2}\\
        &\lesssim\sum_{\abs{\beta}\leq\ceils{d/2}}\norm{x^\beta\jb{\nabla_x}^{n+{{1/2}}+\delta_0}\br{x^\alpha g(x)}}_{L^2}\\
&\lesssim_n\sum_{\abs{\beta}\leq\ceils{d/2}}\norm{\jb{\nabla_x}^{n+{{1/2}}+\delta_0}\br{x^{\beta+\alpha} g(x)}}_{L^2}\\
        &\lesssim_\alpha\norm{g}_{H_{\aal+\ceils{d/2}}^{n+{{1/2}}+\delta_0}}.\tag*{\qedhere}
    \end{align*}
\end{proof}

With Lemma~\ref{gnorm} in hand, estimates on the linearised operator are now derived. The notation \( \widetilde{\mathscr L}(\lambda,k;\hbar) \) is used to make the dependence on \( \hbar \in (0,1] \) explicit, facilitating the analysis of uniformity in the semiclassical limit. For \( k \neq 0 \), a change of variables in \eqref{firstL} yields
\begin{align}
    \widetilde {\mathscr L}(\lam,k;\hbar)
    = \frac{2}{\hbar\ak} \widehat w(k) \int_0^\infty e^{-\lam \mytextfrac{s}{\ak}} \sin(\mytextfrac{\hbar}{2}  \ak s) \widehat g\br{\mytextfrac{k}{\ak}s} \, ds. \label{Lingoodcoordinates}
\end{align}
To compare with the classical setting, consider the limit
\[
\widetilde{\mathscr{L}}(\lambda,k;0) := \lim_{\hbar \to 0} \widetilde{\mathscr{L}}(\lambda,k;\hbar),
\]
which evaluates to
\[
\widetilde {\mathscr{L}}(\lam,k;0)
=  \widehat w(k) \int_0^\infty e^{-\lam \frac{s}{\ak}} s \widehat g\br{\mytextfrac{k}{\ak}s} \, ds.
\]
This coincides with the linearised operator of the classical theory
{\cite[Eq.~(2.6)]{BedrossianMasmoudiMouhot2018}}.

To study the behaviour of \( \widetilde{\mathscr L}(\lam,k;\hbar) \) along $\Re\lam=0$, the Plemelj formula is applied. For $\hbar\in(0,1]$ and \( \lam = i\tau \), it gives
    \begin{equation}\label{plemelj}
\begin{aligned}
 \widetilde {\mathscr L}(i\tau,k;\hbar)&=\widehat w(k)\frac{1}{\hbar\ak}\Bigg(\mathrm{p.v.}\int_{\R}\Bigg[\frac{g_k(p_1)}{p_1-\frac{\tau}{\ak}+\frac{\hbar \ak}{2}}-\frac{g_k(p_1)}{p_1-\frac{\tau}{\ak}-\frac{\hbar \ak}{2}}\Bigg]\,dp_1\\
&\hspace{4.25cm}+i\pi\sbr{g_k\br{\mytextfrac{\tau}{\ak}+\mytextfrac{\hbar \ak}{2}}-g_k\br{\mytextfrac{\tau}{\ak}-\mytextfrac{\hbar \ak}{2}}}\Bigg),
\end{aligned}
\end{equation}where $\mathrm{p.v.}$ denotes the Cauchy principal value and \( g_k \) denotes the marginal of \( g \) along the direction \( -k \), defined by
\[
g_k(u) := \int_{\{k\}^\perp} g\left( -\mytextfrac{k}{\abs{k}} u + w \right) dw.
\]Similarly, applying the Plemelj formula to \( \widetilde {\mathscr L}(\lam,k;0) \) yields
\begin{align*}
    \widetilde {\mathscr{L}}(i\tau,k;0)=\widehat w(k)\br{-\,\mathrm{p.v.}\int_{\R}\frac{g_k'(p_1)}{p_1-\frac{\tau}{\ak}}\,dp_1+i\pi g_k'(\mytextfrac{\tau}{\ak})}.
\end{align*}

Using Lemma~\ref{gnorm}, uniform bounds and decay properties of \( \widetilde{\mathscr L}(\lam,k;\hbar) \) are now established for all \( \hbar \in [0,1] \). 
\begin{lemma}\label{boundedness and decay lemma}
Suppose $w\in L^1(\R^d)$ and $g\in H^{3/2+\delta_0}_{\ceils{d/2}}(\R^d)$. Then: \begin{enumerate}[label=(\alph*)]
    \item \label{(a)} As $\ak\to\infty$, the quantity $\widetilde {\mathscr L}(\lam,k;\hbar)\to0$, with decay uniform in $\Re\lam\geq0$ and $\hbar\in[0,1]$.
    
    \item \label{(b)}There exists a constant $C_1=C_1(d,\delta_0)$, independent of $\hbar\in [0,1]$, such that for all $\Re\lam\geq0$ and $k\in\R^d$, \begin{align*}
        \abs{\widetilde {\mathscr L}(\lam,k;\hbar)}\leq C_1\norm{w}_{L^1}\norm{g}_{H^{3/2+\delta_0}_{\ceils{d/2}}}.
    \end{align*} 
    \item \label{(c)} If in addition $g\in H^{3/2+\delta_0}_{1+\ceils{d/2}}(\R^d)$, then there exists a constant $C_2=C_2(d,\delta_0)$, independent of $\hbar\in [0,1]$, such that for all $\Re\lam\geq0$ and $k\in\R^d$, \begin{align*}
        \abs{\widetilde {\mathscr L}(\lam,k;\hbar)}\leq C_2\frac{\ak}{\abs{\lam}}\norm{w}_{L^1}\norm{g}_{H^{3/2+\delta_0}_{1+\ceils{d/2}}}.
    \end{align*}
\end{enumerate}
\end{lemma}
\begin{proof}
First observe that \( \widetilde{\mathscr L}(\lam,k=0;\hbar) = 0 \) for all \( \Re \lam \geq 0 \) and \( \hbar \in [0,1] \). Now fix \( \hbar \in (0,1] \). For \( k \neq 0 \) and \( \Re \lam \geq 0 \), the representation \eqref{Lingoodcoordinates} together with Lemma~\ref{gnorm} yields the estimate
\begin{align*}
\abs{\widetilde{\mathscr L}(\lam,k;\hbar)}
&\lesssim_d \abs{\widehat w(k)} \frac{1}{\hbar \ak} \int_0^\infty \abs{\sin\br{\mytextfrac{\hbar}{2} \ak s} \, \widehat g\br{\mytextfrac{k}{\ak}s}}\,ds \\
&\lesssim \abs{\widehat w(k)} \int_0^\infty s \abs{\widehat g\br{\mytextfrac{k}{\ak}s}}\,ds \\
&\lesssim_{d,\delta_0} \abs{\widehat w(k)} \norm{g}_{H^{3/2+\delta_0}_{\ceils{d/2}}}.
\end{align*}
Since \( w \in L^1(\R^d) \), the Fourier transform \( \widehat w(k) \to 0 \) as \( \abs{k} \to \infty \), establishing part~\ref{(a)}. Moreover, \( \abs{\widehat w(k)} \leq \norm{w}_{L^1} \) for all \( k \), yielding part~\ref{(b)}.

For part~\ref{(c)}, the expression \eqref{Lingoodcoordinates} can be integrated by parts in \( s \), with no boundary contributions, to obtain
\begin{align*}
\widetilde{\mathscr L}(\lam,k;\hbar)
&= 2 \widehat w(k) \frac{1}{\hbar \ak} \left(-\frac{\ak}{\lam}\right) \int_0^\infty \p_s \left( e^{-\lam \frac{s}{\ak}} \right) \sin(\mytextfrac{\hbar}{2} \ak s) \, \widehat g\br{\mytextfrac{k}{\ak}s} \, ds \\
&= 2 \widehat w(k) \frac{1}{\hbar \lam} \int_0^\infty e^{-\lam \frac{s}{\ak}} \left[ \mytextfrac{\hbar}{2} \ak \cos(\mytextfrac{\hbar}{2} \ak s) \widehat g\br{\mytextfrac{k}{\ak}s} + \sin(\mytextfrac{\hbar}{2} \ak s) \mytextfrac{k}{\ak} \cdot \nabla_k \widehat g\br{\mytextfrac{k}{\ak}s} \right] ds.
\end{align*}
Therefore,
\begin{align*}
\abs{\widetilde{\mathscr L}(\lam,k;\hbar)}
&\lesssim \frac{\ak}{\abs{\lam}} \norm{w}_{L^1} \int_0^\infty \left[ \abs{\widehat g\br{\mytextfrac{k}{\ak}s}} + s \abs{\nabla_k \widehat g\br{\mytextfrac{k}{\ak}s}} \right] ds \\
&\lesssim_{d,\delta_0} \frac{\ak}{\abs{\lam}} \norm{w}_{L^1} \norm{g}_{H^{3/2+\delta_0}_{1+\ceils{d/2}}}. 
\end{align*}
The same bounds hold for the classical limit \( \widetilde{\mathscr L}(\lam,k;0) \), with identical estimates.
\end{proof}

Next, using the argument principle, a condition on the behaviour of \( \widetilde{\mathscr L} \) along the imaginary axis is shown to imply the uniform Penrose condition. The low frequencies require separate attention: \( \widetilde{\mathscr L}(\lam,0;\hbar)=0 \), whereas \eqref{Lingoodcoordinates} and dominated convergence give, for every $\hbar\in[0,1]$, $\Re\mu\geq0$ and $\omega\in\mathbb S^{d-1}$,
\begin{align*}
    \lim_{r\downarrow0}\widetilde{\mathscr L}(r\mu,r\omega;\hbar)=\widetilde{\mathscr L}_0(\mu,\omega):=\widehat w(0)\int_0^\infty e^{-\mu s}s\,\widehat g(\omega s)\,ds,
\end{align*}
so \( \widetilde{\mathscr L} \) is not in general continuous at $(\lam,k)=(0,0)$.

\begin{lemma}
Suppose \( w \in L^1(\R^d) \), \( g \in H^{3/2+\delta_0}_{1+\lceil d/2 \rceil}(\R^d) \), and that the following implications hold:\begin{align}
   \label{impliesPenrose}\forall\hbar\in[0,1],\,\tau\in\R,\,k\in\R^d,\quad  \Im\widetilde {\mathscr L}(i\tau,k;\hbar)=0\implies \Re\widetilde {\mathscr L}(i\tau,k;\hbar)>-1.
\end{align}and, for the low-frequency response,
\begin{align}
   \label{impliesPenrose0}\forall\tau\in\R,\,\omega\in\mathbb S^{d-1},\quad  \Im\widetilde {\mathscr L}_0(i\tau,\omega)=0\implies \Re\widetilde {\mathscr L}_0(i\tau,\omega)>-1.
\end{align} Then the uniform Penrose condition \eqref{Penrose} is satisfied.
\end{lemma}
\begin{proof}
For each $\hbar\in[0,1]$ and $k\in\R^d$, the assumptions on $w$ and $g$ ensure that the map $\lam\mapsto \widetilde {\mathscr L}(\lam,k;\hbar)$ is analytic in the half-plane $\set{\Re\lam>0}$, continuous on $\set{\Re\lam\geq0}$ and, by Lemma~\ref{boundedness and decay lemma}~\ref{(c)}, tends to zero as $\abs{\lam}\to\infty$. By the argument principle, the number of zeros of the function $\lam\mapsto 1+\widetilde {\mathscr L}(\lam,k;\hbar)$ in $\set{\Re\lam>0}$ is equal to the winding number of the curve
    \begin{align*}
        \Gamma:=\{\widetilde {\mathscr L}(i\tau,k;\hbar):\tau\in \R\}
    \end{align*}
    around the point $-1$. The implication \eqref{impliesPenrose} ensures that $\Gamma$ does not meet the half-line $(-\infty,-1]$. It follows that $1+\widetilde {\mathscr L}(\lam,k;\hbar)$ has no zeros in $\set{\Re\lam\geq0}$. The same argument, applied to $\mu\mapsto1+\widetilde{\mathscr L}_0(\mu,\omega)$ for fixed $\omega\in\mathbb S^{d-1}$ and using \eqref{impliesPenrose0}, shows that $1+\widetilde{\mathscr L}_0(\mu,\omega)$ has no zeros in $\set{\Re\mu\geq0}$.

It remains to show that the infimum in \eqref{Penrose} is positive. For $k\neq0$, $\hbar\in[0,1]$ and $\Re\lam\geq0$, the representation \eqref{Lingoodcoordinates} reads
\[
    \widetilde{\mathscr L}(\lam,k;\hbar)=\widehat w(k)\,\Phi\br{\frac{\lam}{\ak},\hbar\ak,\frac{k}{\ak}},\qquad \Phi(\mu,h,\omega):=\int_0^\infty e^{-\mu s}\,\frac{2}{h}\sin\br{\frac{hs}{2}}\widehat g(\omega s)\,ds,
\]
where the factor $2h^{-1}\sin(hs/2)$ is understood as $s$ when $h=0$, so that $\widetilde{\mathscr L}_0(\mu,\omega)=\widehat w(0)\Phi(\mu,0,\omega)$. The function $\Phi$ is continuous on $\set{\Re\mu\geq0}\times[0,\infty)\times\mathbb S^{d-1}$. Indeed, the integrand is continuous and bounded by $s\norm{\widehat g}_{L^\infty}$, so the integral over $[0,R]$ is continuous for every $R>0$, while $\abs{2h^{-1}\sin(hs/2)}\leq s$ and the Cauchy--Schwarz and trace estimates in the proof of Lemma~\ref{gnorm} give, uniformly in $\Re\mu\geq0$, $h\geq0$ and $\omega\in\mathbb S^{d-1}$,
\begin{align*}
    \int_R^\infty\abs{e^{-\mu s}\,\frac{2}{h}\sin\br{\frac{hs}{2}}\widehat g(\omega s)}\,ds
    &\leq\br{\int_R^\infty\frac{s^2}{\jb{s}^{3+2\delta_0}}\,ds}^{\frac{1}{2}}\br{\int_0^\infty\jb{s}^{3+2\delta_0}\abs{\widehat g(\omega s)}^2\,ds}^{\frac{1}{2}}\\
    &\lesssim_{d,\delta_0} R^{-\delta_0}\norm{g}_{H^{3/2+\delta_0}_{\ceils{d/2}}},
\end{align*}
so $\Phi$ is a uniform limit of continuous functions. Now set
\[
    F(\mu,r,\omega,\hbar):=1+\widehat w(r\omega)\,\Phi(\mu,\hbar r,\omega),
\]
so that $F(\mu,r,\omega,\hbar)=1+\widetilde{\mathscr L}(r\mu,r\omega;\hbar)$ for $r>0$ and $F(\mu,0,\omega,\hbar)=1+\widetilde{\mathscr L}_0(\mu,\omega)$.

    By Lemma~\ref{boundedness and decay lemma}, there exist constants \( K > 0 \) and \( M > 0 \) such that, for \( k\neq0 \) and \( \Re \lam \geq 0 \), if either (i) \( \ak \geq K \), or (ii) \( \abs{\lam} \geq M\ak \), then
\[
\abs{1 + \widetilde{\mathscr L}(\lam,k;\hbar)} \geq \frac{1}{2}
\]uniformly in $\hbar\in[0,1]$, while $\abs{1 + \widetilde{\mathscr L}(\lam,0;\hbar)}=1$. Writing $\lam=r\mu$ and $k=r\omega$, the remaining region corresponds to the compact set \[
\mathcal K:=\left\{
(\mu, r, \omega, \hbar) : \,
\Re\mu\geq0,\,\abs{\mu} \leq M,\,
0\leq r \leq K,\,
\omega\in\mathbb S^{d-1},\,
\hbar \in [0,1]
\right\}.
\]
On this compact set, the function $F$ is continuous, since $\widehat w$ and $\Phi$ are, and it does not vanish, by the first paragraph. Hence
\[
{c_{\mathrm{P}}} := \min\set{\frac{1}{2},\,\min_{\mathcal K}\abs{F}} > 0
\]
is a lower bound for $\abs{1 + \widetilde{\mathscr L}(\lam,k;\hbar)}$ over $\Re \lam \geq 0$, $k\in\R^d$ and $\hbar \in [0,1]$.
This establishes the uniform Penrose condition \eqref{Penrose}.
\end{proof}

Three sufficient conditions on the interaction kernel \( w \) and the steady state \( g \) are now presented, each of which implies the uniform Penrose condition \eqref{Penrose}. These include a smallness assumption, a repulsive interaction with decreasing marginal, and a generalised Penrose-type criterion.
\begin{proposition}\label{sufficientforPenrose}
Suppose one of the following assumptions holds. Then the uniform Penrose condition \eqref{Penrose} follows.
\begin{enumerate}[label=(\roman*)]
\item \label{smallness} \textit{Smallness:}  
\( w \in L^1(\R^d) \), and \( g \in H^{3/2+\delta_0}_{\lceil d/2 \rceil}(\R^d) \), for some \( 0 < \delta_0 \ll 1/2 \), satisfy
\[
C_1 \norm{w}_{L^1} \norm{g}_{H^{3/2+\delta_0}_{\lceil d/2 \rceil}} < 1,
\]
where \( C_1 \) is the constant from Lemma~\ref{boundedness and decay lemma}.
\item \label{marginal} \textit{Repulsive interaction and decreasing marginal:}  
\( w \in L^1(\R^d) \) with \( \widehat w(k) \geq 0 \), and \( g \in H^{3/2+\delta_0}_{1 + \lceil d/2 \rceil}(\R^d) \) is radial and satisfies
\begin{align}\label{marginalcases}
\begin{mycases}
g_k'(u) < 0 & \text{if } u > 0, \\
g_k'(u) > 0 & \text{if } u < 0,
\end{mycases}
\end{align}
for all \( k \in \R^d \).

    \item \textit{Generalised uniform Penrose condition:} \label{generalisedPenrose(iii)} $w\in L^1(\R^d)$ and $g\in H^{3/2+\delta_0}_{1+\ceils{d/2}}(\R^d)$, and, for all $\tau\in\R$, $k\neq 0$, \begin{align*}
        g_k'\br{\mytextfrac{\tau}{\ak}}=0\implies -\widehat w(k)\,\mathrm{p.v.}\int_{\R}\frac{g_k'(p_1)}{p_1-\frac{\tau}{\ak}}\,dp_1>-1
    \end{align*}together with the same implication with $\widehat w(k)$ replaced by $\widehat w(0)$, and, for $\hbar\in (0,1]$,\begin{equation*}\begin{aligned}
        g_k\br{\mytextfrac{\tau}{\ak}+\mytextfrac{\hbar\ak}{2}}&=g_k\br{\mytextfrac{\tau}{\ak}-\mytextfrac{\hbar \ak}{2}}\\
        &\hspace{-0.75cm}\implies
        \widehat w(k)\frac{1}{\hbar \ak}\mathrm{p.v.}\int_{\R}g_k(p_1)\sbr{\frac{1}{p_1-\frac{\tau}{\ak}+\frac{\hbar\ak}{2}}-\frac{1}{p_1-\frac{\tau}{\ak}-\frac{\hbar \ak}{2}}}\,dp_1>- 1.
    \end{aligned}\end{equation*}

\end{enumerate}

\end{proposition}
\begin{remark}
    In dimensions \( d \geq 3 \), if \( w \in L^1(\R^d) \) with $\widehat w(k)\geq0$ and \( g \in H^{3/2+\delta_0}_{1+\lceil d/2 \rceil}(\R^d) \) is radial and positive-valued, then the condition \eqref{marginalcases} in part~\ref{marginal} is satisfied.
\end{remark}
\begin{proof}
In case~\ref{smallness}, apply Lemma~\ref{boundedness and decay lemma}~\ref{(b)} directly. For case~\ref{generalisedPenrose(iii)}, note that if \( g \in L^1(\R^d) \), then \( \widetilde{\mathscr L}(i\tau,k=0;\hbar) = 0 \) for all \( \tau \in \R \). For \( k \neq 0 \), the representation \eqref{plemelj} shows that the stated inequalities directly ensure \eqref{impliesPenrose} via the Plemelj formula. Moreover, with \( \omega := k/\ak \), the classical Plemelj formula gives
\[
\widetilde{\mathscr L}_0(i\tau,\omega) = \widehat w(0)\br{-\,\mathrm{p.v.}\int_{\R}\frac{g_k'(p_1)}{p_1-\tau}\,dp_1+i\pi g_k'(\tau)},
\]
so the implication with \( \widehat w(0) \) ensures \eqref{impliesPenrose0}. In case~\ref{marginal}, first consider $\hbar\in(0,1]$ and assume \( g \) is radial and satisfies \eqref{marginalcases}. Then for \( k \neq 0 \), the equality $g_k\left( \mytextfrac{\tau}{\abs{k}} + \mytextfrac{\hbar \abs{k}}{2} \right)
=
g_k\left( \mytextfrac{\tau}{\abs{k}} - \mytextfrac{\hbar \abs{k}}{2} \right)$
can only occur when \( \tau = 0 \), due to the strict monotonicity of \( g_k \). When \( \tau = 0 \), the Plemelj formula \eqref{plemelj} yields
\[
\Re \widetilde{\mathscr L}(0,k;\hbar)
= \widehat w(k)\frac{2}{\hbar \abs{k}}
\lim_{\varepsilon \downarrow 0} \int_\varepsilon^\infty
\frac{
g_k\left(p_1 - \frac{\hbar \abs{k}}{2}\right)
- g_k\left(p_1 + \frac{\hbar \abs{k}}{2}\right)
}{p_1} \, dp_1,
\]
which is non-negative by the assumption \( g_k'(u) < 0 \) for \( u > 0 \). For the low-frequency response, the formula for \( \widetilde{\mathscr L}_0 \) displayed above gives
\[
\Im \widetilde{\mathscr L}_0(i\tau,\omega) = \pi\widehat w(0)\,g_k'(\tau),
\qquad
\Re \widetilde{\mathscr L}_0(0,\omega) = \widehat w(0)\int_{\R}\br{-\frac{g_k'(p_1)}{p_1}}\,dp_1.
\]
If \( \widehat w(0)=0 \), then \( \widetilde{\mathscr L}_0\equiv0 \) and \eqref{impliesPenrose0} holds trivially. Otherwise \( \widehat w(0)>0 \), so the imaginary part vanishes only at \( \tau = 0 \), where the real part is non-negative by \eqref{marginalcases}. This gives \eqref{impliesPenrose0}. The same argument, with \( \widehat w(0) \) replaced by \( \widehat w(k) \), applies in the classical case \( \hbar = 0 \), see e.g.\ \cite{bed-mas-mou-2016}, again leading to \eqref{impliesPenrose}.
\end{proof}

\subsection{Inverse Laplace transform of Green function}
Throughout the remainder of the paper, the interaction kernel \( w \) and the steady state \( g \) are assumed to satisfy the uniform Penrose condition \eqref{Penrose} with constant \( {c_{\mathrm{P}}} > 0 \). Under this assumption, the formula
\begin{align}
    \widetilde{\rho_{Q_{\mathrm{L}}}}(\lambda,k) = \widetilde{\mathscr{G}}(\lambda,k)\,\widetilde{\rho_{Q_{\mathrm{FS}}}}(\lambda,k) \label{1/1+L}
\end{align}
is justified, where
\begin{align}
    \widetilde{\mathscr{G}}(\lambda,k) := \frac{1}{1 + \widetilde{\mathscr{L}}(\lambda,k)}. \label{definitionofwidetildeG}
\end{align}
The analysis of the inverse linearised operator, referred to as the \textit{Green function}, follows the approach of Nguyen and You~\cite{NguyenYou2023}. Specifically, this includes Lemma~\ref{mflemma}, Lemma~\ref{Grdecay}, and Proposition~\ref{widehatGrdecay}.

Define the \textit{Lindhard} function $m_g$ by \begin{align}\begin{split}
    m_g(\lam,k)&=\frac{2}{\hbar}\int_0^\infty e^{-\lam t}\mathrm{sin}( \mytextfrac{\hbar}{2} t\ak^2)\widehat g(kt)\,dt\\
   &=\frac{2}{ 
\hbar\ak}\int_0^\infty e^{-\lam \mytextfrac{s}{\ak}}\mathrm{sin}(\mytextfrac{\hbar}{2}\ak s)\widehat g(\mytextfrac{k}{\ak}s)\,ds,\end{split}\label{mf}
\end{align}if $k\neq 0$, such that $\widetilde {\mathscr L}(\lam,k)=\widehat w(k)m_g(\lam,k)$. The Lindhard function plays a central role in the physics literature; see~\cite{LewinSabinII2014}.

\begin{lemma}[Decay of $m_g(\lam,k)$]\label{mflemma}For any $n\geq0$, if $g\in{H^{n+3/2+\delta_0}_{2+\ceils{d/2}}}(\R^d)$, then
\begin{enumerate}[label=(\roman*)]
    \item $\abs{\p_\lam^nm_g(i\tau,k)}\lesssim_{n,d,\delta_0}\ak^{-n}\norm{g}_{H^{n+3/2+\delta_0}_{\ceils{d/2}}},$ and\label{mf1}
    \item $\abs{\p_\lam^nm_g(i\tau,k)}\lesssim_{n,d,\delta_0}\frac{\jb{\hbar k}^2\ak^{2-n}}{\jb{\hbar k}^2\ak^2+\tau^2}\norm{g}_{H^{n+3/2+\delta_0}_{2+\ceils{d/2}}}$.\label{mf2}
\end{enumerate}
\end{lemma}
\begin{proof}
    Differentiating \eqref{mf} \( n \) times in \( \lambda \) yields\begin{align*}
        (-1)^n\p_\lam^n m_g(\lam,k)&=\frac{1}{\ak^n}\frac{2}{ \hbar\ak}\int_0^\infty e^{-\lam\frac{s}{\ak}}s^n\mathrm{sin}(\mytextfrac{\hbar}{2}\ak s)\widehat g(\mytextfrac{k}{\ak}s)\,ds,
    \end{align*}Hence, for any \( \Re \lambda \geq 0 \),\begin{align*}
        \abs{\p_\lam^n m_g(\lam,k)}\lesssim_{n}\frac{1}{\ak^n}\int_0^\infty s^{n+1}\abs{\widehat g(\mytextfrac{k}{\ak}s)}\,ds\lesssim_{n,d,\delta_0}\frac{1}{\ak^n}\norm{g}_{H^{n+3/2+\delta_0}_{\ceils{d/2}}}
    \end{align*}using Lemma~\ref{gnorm}, which proves part \ref{mf1}.

    To establish part \ref{mf2}, consider \( \lambda = i\tau \in i\R \) and \( n \geq 2 \). Multiply both sides by \( \tau^2 \) and integrate by parts twice: \begin{align*}
        &\tau^2(-1)^{n}\p_\lam^nm_g(i\tau,k)\\
        &=\frac{1}{\ak^n}\frac{2}{ \hbar\ak}\int_0^\infty \tau^2e^{-i\tau\frac{s}{\ak}}s^n\mathrm{sin}(\mytextfrac{\hbar}{2}\ak s)\widehat g(\mytextfrac{k}{\ak}s)\,ds\\
        &=\frac{\ak^2}{\ak^{n}}\frac{2}{ \hbar\ak}\int_0^\infty (-\p_s^2)\br{e^{-i\tau\frac{s}{\ak}}}s^n\mathrm{sin}(\mytextfrac{\hbar}{2}\ak s)\widehat g(\mytextfrac{k}{\ak}s)\,ds\\
        &=\frac{\ak^2}{\ak^{n}}\frac{2}{ \hbar\ak}\int_0^\infty e^{-i\tau\frac{s}{\ak}}\Big[\mytextfrac{\hbar^2}{4}\ak^2s^n\mathrm{sin}(\mytextfrac{\hbar}{2}\ak s)\widehat g(\mytextfrac{k}{\ak}s)-n(n-1)s^{n-2}\mathrm{sin}(\mytextfrac{\hbar}{2}\ak s)\widehat g(\mytextfrac{k}{\ak}s)\\
        &\hspace{3.5cm}-2ns^{n-1}\mathrm{sin}(\mytextfrac{\hbar}{2}\ak s)\frac{k}{\ak}\cdot\nabla_k\widehat g(\frac{k}{\ak}s)- \hbar\ak ns^{n-1}\mathrm{cos}(\mytextfrac{\hbar}{2}\ak s)\widehat g(\frac{k}{\ak}s)\\
        &\hspace{3.5cm}- \hbar\ak s^n\mathrm{cos}(\mytextfrac{\hbar}{2}\ak s)\mytextfrac{k}{\ak}\cdot\nabla_k\widehat g(\frac{k}{\ak}s)-s^n\mathrm{sin}(\mytextfrac{\hbar}{2}\ak s)(\mytextfrac{k}{\ak}\cdot\nabla_k)^2\widehat g(\mytextfrac{k}{\ak}s)\Big]ds.
    \end{align*}The first term is exactly $(-1)^n\frac{\hbar^2}{4}\ak^4\p_\lam^nm_g(i\tau,k)$. Moreover, all the other terms in the integrand are $O(\hbar\ak)$, and hence
    \begin{align*}
        \tau^2&\abs{\p_\lam^nm_g(i\tau,k)}\\
        &\lesssim_n \hbar^2\ak^4\abs{\p_\lam^nm_g(i\tau,k)}+\ak^{2-n}\int_0^\infty \jb{s}^{n+1}\br{\abs{\widehat g(\mytextfrac{k}{\ak}s)}+\abs{\nabla_k\widehat g(\mytextfrac{k}{\ak}s)}+\abs{\nabla_k^2\widehat g(\mytextfrac{k}{\ak}s)}}\,ds\\
        &\lesssim_{n,d,\delta_0}\hbar^2\ak^{4-n}\norm{g}_{H^{n+3/2+\delta_0}_{\ceils{d/2}}}+\ak^{2-n}\norm{g}_{H^{n+3/2+\delta_0}_{2+\ceils{d/2}}},
    \end{align*}again using Lemma~\ref{gnorm}. The cases $n=0$ and $n=1$ give rise to fewer terms upon the integration by parts and are treated similarly. Finally, combining this with the bound from part \ref{mf1} gives
\[
(\jb{\hbar k}^2 \ak^2 + \tau^2) \abs{ \partial_\lambda^n m_g(i\tau, k) } \lesssim \jb{\hbar k}^2 \ak^{2 - n},
\]
which proves part \ref{mf2}.
\end{proof}

To quantify the deviation from the identity in the Green function, define the remainder \begin{align*}
    \widetilde {\mathscr G^r}(\lam,k):=\frac{\widetilde {\mathscr L}(\lam,k)}{1+\widetilde {\mathscr L}(\lam,k)}=\frac{\widehat w(k)m_g(\lam,k)}{1+\widehat w(k)m_g(\lam,k)},
\end{align*}so that $\widetilde {\mathscr G}(\lam,k)=1-\widetilde {\mathscr G^r}(\lam,k)$.

\begin{lemma}[Decay of $\widetilde {\mathscr G^r}(\lam,k)$]\label{Grdecay}
    Let $n\geq0$ and suppose $w\in L^1(\R^d)$ and $g\in H^{n+3/2+\delta_0}_{2+\ceils{d/2}}(\R^d)$ satisfy the uniform Penrose condition \eqref{Penrose} with constant ${c_{\mathrm{P}}}$. Then, there exists a constant $C=C(n,w,g,{c_{\mathrm{P}}},d,\delta_0)$ such that \begin{align*}
          \abs{\p_\lam^n\widetilde {\mathscr G^r}(i\tau,k)}\leq C\frac{\jb{\hbar k}^2\ak^{2-n}}{\jb{\hbar k}^2\ak^2+\tau^2}.
    \end{align*}
\end{lemma}
\begin{proof}
    Using the Faà di Bruno formula, for any $n\geq 1$,\begin{align*}
     \frac{d^n}{dx^n}\sbr{\frac{f(x)}{1+f(x)}}=\sum_{\pi\in\Pi_n}\frac{(-1)^{\abs{\pi}+1}\abs{\pi}!}{(1+f(x))^{\abs{\pi}+1}}\prod_{B\in\pi}f^{(\abs{B})}(x),
\end{align*}where $\pi$ runs through the set $\Pi_n$ of all the partitions of the set $\set{1,\ldots,n}$. So,\begin{align*}
    \p_\lam^n\widetilde {\mathscr G^r}(i\tau,k)&=\frac{d^n}{d\lam^n}\sbr{\frac{\widehat w(k)m_g(\lam,k)}{1+\widehat w(k)m_g(\lam,k)}}\Bigg|_{\lam=i\tau}\\
    &=\sum_{\pi\in\Pi_n}\frac{(-1)^{\abs{\pi}+1}\abs{\pi}!}{(1+\widehat w(k)m_g(i\tau,k))^{\abs{\pi}+1}}\prod_{B\in\pi}\widehat w(k)\p^{\abs{B}}_\lam m_g(i\tau,k).
\end{align*}Moreover, writing $\pi=\set{B_0}\cup(\pi\setminus \set{B_0})$ and using Lemma~\ref{mflemma} \ref{mf2} to bound the term $\p_\lam^{\abs{B_0}}m_g(i\tau,k)$ and Lemma~\ref{mflemma}~\ref{mf1} to bound the $\p^{\abs{B}}_\lam m_g(i\tau,k)$ terms in the (possibly empty) $\pi\setminus\set{B_0}$, \begin{align*}
    &\abs{\p_\lam^n\widetilde {\mathscr G^r}(i\tau,k)}\\
    &\quad\lesssim\sum_{\pi\in\Pi_n}\frac{1}{{c_{\mathrm{P}}}^{\abs{\pi}+1}}\frac{\jb{\hbar k}^2\ak^{2-\abs{B_0}}}{\jb{\hbar k}^2\ak^2+\tau^2}\abs{\widehat w(k)}\norm{g}_{H^{\abs{B_0}+3/2+\delta_0}_{2+\ceils{d/2}}}\prod_{B\in \pi\setminus\set{B_0}}\abs{\widehat w(k)}\ak^{-\abs{B}}\norm{g}_{H^{\abs{B}+3/2+\delta_0}_{\ceils{d/2}}}\\
    &\quad\leq\frac{\jb{\hbar k}^2\ak^{2-n}}{\jb{\hbar k}^2\ak^2+\tau^2}\sum_{\pi\in \Pi_n}\frac{\abs{\widehat w(k)}^{\abs{\pi}}}{{c_{\mathrm{P}}}^{\abs{\pi}+1}}\prod_{B\in\pi}\norm{g}_{H^{\abs{B}+3/2+\delta_0}_{2+\ceils{d/2}}}.\tag*{\qedhere}
\end{align*}
\end{proof}

The decay of the Green function on the Fourier side is now obtained by taking the inverse Laplace transform.
\begin{proposition}[Decay of $\widehat {\mathscr G^r}(t,k)$]\label{widehatGrdecay} Let $n\geq0$ and assume $w\in L^1(\R^d)$ and $g\in H^{n+3/2+\delta_0}_{2+\ceils{d/2}}(\R^d)$ satisfy the uniform Penrose condition \eqref{Penrose} with constant ${c_{\mathrm{P}}}$. Then, there exists a constant $C=C(n,w,g,{c_{\mathrm{P}}},d,\delta_0)$ such that\begin{align*}
    \abs{\widehat {\mathscr G^r}(t,k)}\leq C\frac{\jb{\hbar k}\ak}{\jb{kt}^n}.
\end{align*}
\end{proposition}
\begin{proof}
By the uniform Penrose condition \eqref{Penrose}, the map \( \lambda \mapsto \widetilde{\mathscr{G}^r}(\lambda, k) \) is analytic in \( \Re \lambda > 0 \) and extends continuously to \( \Re \lambda = 0 \). Moreover, Lemma~\ref{Grdecay} ensures sufficient decay in \( \Im \lambda \), so the inverse Laplace contour can be deformed to the imaginary axis:\begin{align*}
    \widehat {\mathscr G^r}(t,k)=\frac{1}{2\pi i}\int_{\Re\lam=0}e^{\lam t}\widetilde {\mathscr G^r}(\lam,k)\,d\lam=\frac{1}{2\pi}\int_\R e^{i\tau t}\widetilde {\mathscr G^r}(i\tau,k)\,d\tau.
\end{align*}Assume first \( \ak t \geq 1 \). Integrating by parts \( n \) times and applying Lemma~\ref{Grdecay},
\begin{align*}
          \abs{\widehat {\mathscr G^r}(t,k)}=\frac{1}{2\pi t^n}\abs{\int_\R \p_{\tau}^n\br{e^{i\tau t}}\widetilde {\mathscr G^r}(i\tau,k)\,d\tau}&=\frac{1}{2\pi t^n}\abs{\int_\R e^{i\tau t}\p_{\tau}^n\widetilde {\mathscr G^r}(i\tau,k)\,d\tau}\\
          &\lesssim\frac{1}{(\ak t)^n}\int_\R\frac{\jb{\hbar k }^2\ak^2
    }{\jb{\hbar k}^2\ak^2+\tau^2}\,d\tau\\
          &= \pi \frac{\jb{\hbar k}\ak}{(\ak t)^n},
      \end{align*}where it was used that $\int_\R\frac{a^2}{a^2+x^2}\,dx=\pi a$ for $a>0$. No boundary terms arise by Lemma~\ref{Grdecay}. For \( 0 \leq \ak t \leq 1 \), the estimate follows from Lemma~\ref{Grdecay} with \( n = 0 \), without integration by parts.
\end{proof}

\subsection{Linearised decay and scattering}
Taking the inverse Laplace transform of \eqref{1/1+L}, the density can be written as\begin{align}\label{widehatrhoQ}
    \widehat{\rho_{Q_{\mathrm{L}}}}(t,k)=\widehat{\rho_{Q_{\mathrm{FS}}}}(t,k)-\int_0^t\widehat {\mathscr G^r}(t-s,k)\widehat{\rho_{Q_{\mathrm{FS}}}}(s,k)\,ds.
\end{align}The next proposition quantifies decay of the density \( \widehat{\rho_{Q_{\mathrm{L}}}} \) and scattering for \( Q_{\mathrm{L}} \).

\begin{proposition}[Linearised phase mixing]
    Let $N_0,N_1\geq0$ be constants such that $0\leq N_1<N_0-d/2-1$.  Assume $w\in L^1(\R^d)$ and $g\in H^{N_0+9/2+\delta_0}_{2+\ceils{d/2}}(\R^d)$, for some $0<\delta_0\ll 1$, satisfy the uniform Penrose condition \eqref{Penrose} with constant ${c_{\mathrm{P}}}$. Assume the initial perturbation satisfies \[\norm{\Qin}_{\mathcal H^{N_0+1}_{M,M}}<\infty.\] Then, there exists a constant $C=C(N_0,w,g,{c_{\mathrm{P}}},d,\delta_0)$, independent of $\hbar\in (0,1]$, such that the solution ${Q_{\mathrm{L}}}$ to the linearised Hartree equation \eqref{LH} with initial data ${Q_{\mathrm{L}}}(0)=\Qin$, satisfies the following.
    \begin{itemize}
        \item Density decay: for all $k\in\R^d$ and $t\geq0$,\begin{align*}
    \jb{k,  kt}^{N_0}\abs{\widehat{\rho_{Q_{\mathrm{L}}}}(t,k)}\leq C\norm{Q_{\mathrm{in}}}_{\mathcal H^{N_0+1}_{M,M}}.
\end{align*}
\item Scattering: there exists an operator $Q_\infty\in \mathcal H^{N_1}(\R^d;\C)$ such that, for $t\geq 0$ and $\hbar\in (0,1]$,\begin{align}\label{scatteringstatement}
   \norm{\jb{t\bm{\nabla}_{x}, \bm{\nabla}_\xi}^{N_1}\sbr{e^{-i\frac{\hbar}{2}t\Delta}{Q_{\mathrm{L}}}(t)e^{i\frac{\hbar}{2}t\Delta}-Q_\infty}}_{\mathcal L^2}\leq C\frac{\norm{Q_{\mathrm{in}}}_{\mathcal H^{N_0+1}_{M,M}}}{\jb{t}^{d/2}}.
\end{align}
    \end{itemize}
\end{proposition}\begin{remark}
    Observe that each spatial derivative of $e^{-i\frac{\hbar}{2}t\Delta}{Q_{\mathrm{L}}}(t)e^{i\frac{\hbar}{2}t\Delta}$ improves the convergence to $Q_\infty$ by a power of $t$. 
\end{remark}
\begin{proof}
Apply Proposition~\ref{propfreephasemixing} to bound the free density \( \widehat{\rho_{Q_{\mathrm{FS}}}} \), and Proposition~\ref{widehatGrdecay} for decay of the Green function remainder \( \widehat{\mathscr{G}^r} \), in the representation \eqref{widehatrhoQ}. This yields \begin{align*}
      \abs{\widehat{\rho_{Q_{\mathrm{L}}}}(t,k)}&\lesssim\frac{\norm{Q_{\mathrm{in}}}_{\mathcal H^{N_0+1}_{M,M}}}{\jb{k,  kt}^{N_0+1}}+\int_0^t\frac{\jb{\hbar k}\ak}{\jb{k s}^{N_0+3}}\frac{\norm{Q_{\mathrm{in}}}_{\mathcal H^{N_0+1}_{M,M}}}{\jb{k,  k(t-s)}^{N_0+1}}\,ds\\
      &\lesssim\frac{\norm{Q_{\mathrm{in}}}_{\mathcal H^{N_0+1}_{M,M}}}{\jb{k,kt}^{N_0+1}}+\frac{\jb{\hbar k}\norm{Q_{\mathrm{in}}}_{\mathcal H^{N_0+1}_{M,M}}}{\jb{k,kt}^{N_0+1}}\int_0^t\frac{\ak}{\jb{k s}^{2}}\,ds\\
      &\lesssim\frac{\jb{\hbar k}}{\jb{k,kt}^{N_0+1}}\norm{Q_{\mathrm{in}}}_{\mathcal H^{N_0+1}_{M,M}},
\end{align*}where the inequalities \( \jb{k,kt} \lesssim \jb{k,k(t-s)}\jb{ks} \) and $\int_0^t\ak/\jb{ks}^2\,ds=\int_0^{\ak t}1/\jb{r}^2\,dr\lesssim1$ were used. Since \( \hbar \in (0,1] \), the desired density decay estimate follows.

To establish scattering, perform the change of variables $(k,p)\mapsto(\frac{k}{2}+p,\frac{k}{2}-p)$: 
\begin{align*}
    e^{i\hbar k\cdot p t}\widehat {Q_{\mathrm{L}}}(t,\tfrac{k}{2}+p,\tfrac{k}{2}-p)&-\widehat{Q_{\mathrm{in}}}(\tfrac{k}{2}+p,\tfrac{k}{2}-p)\\
    &\hspace{-1.5cm}=-i\hbar^{d-1}(2\pi)^d\widehat w(k)\sbr{g\big(\hbar (p-\tfrac{k}{2})\big)-g\big(\hbar(p+ \tfrac{k}{2})\big)}\int_0^te^{i\hbar k\cdot ps}\widehat {\rho_{Q_{\mathrm{L}}}}(s,k)\,ds.
\end{align*}
For any $\sigma\geq0$ and multi-index $\alpha$ satisfying $\sigma+\abs{\alpha}\leq N_1$, define
\begin{align*}
    &\mathcal A_{\sigma,\alpha}(s)\\
    &:=\frac{1}{\hbar^{d/2}}\norm{\ak^\sigma\br{\mytextfrac{1}{\hbar}\nabla_p}^{\alpha}\sbr{-i\hbar^{d-1}(2\pi)^d\widehat w(k)\sbr{g\big(\hbar(p-\tfrac{k}{2})\big)-g\big(\hbar(p+\tfrac{k}{2})\big)}e^{i\hbar k\cdot ps}\widehat{\rho_{Q_{\mathrm{L}}}}(s,k)}}_{L^2_{k,p}}.
\end{align*}
Applying the Leibniz rule, the bound $\abs{\widehat w(k)}\leq\norm{w}_{L^1}$, the density decay estimate and the inequality
\begin{align*}
    \int_x\abs{f(x+\ell)-f(x)}^2\,dx\lesssim\abs{\ell}^2\int_x\abs{\nabla f(x)}^2\,dx
\end{align*}
gives
\begin{align*}
    \mathcal A_{\sigma,\alpha}(s)\lesssim\norm{Q_{\mathrm{in}}}_{\mathcal H^{N_0+1}_{M,M}}\norm{w}_{L^1}\sum_{\beta\leq\alpha}\norm{\nabla^\beta g}_{\dot H^1}\br{\int_k\frac{\ak^{2\sigma+2}\abs{ks}^{2\abs{\alpha-\beta}}}{\jb{k,ks}^{2N_0}}\,dk}^{1/2}.
\end{align*}
For $m=\abs{\alpha-\beta}$, the change of variables $k\mapsto k/\jb{s}$ yields
\begin{align*}
    \br{\int_k\frac{\ak^{2\sigma+2}\abs{ks}^{2m}}{\jb{k,ks}^{2N_0}}\,dk}^{1/2}=\frac{1}{\jb{s}^{d/2+\sigma+1}}\br{\frac{s}{\jb{s}}}^{m}\br{\int_k\frac{\ak^{2\sigma+2+2m}}{\jb{k}^{2N_0}}\,dk}^{1/2}\lesssim\frac{1}{\jb{s}^{d/2+\sigma+1}},
\end{align*}
where the final integral is finite since $2N_0-2\sigma-2m-2>d$, which follows from $m\leq\abs{\alpha}$, $\sigma+\abs{\alpha}\leq N_1$ and $N_1<N_0-d/2-1$. Consequently,
\begin{align}\label{linearscatteringintegrand}
    \mathcal A_{\sigma,\alpha}(s)\lesssim\frac{\norm{Q_{\mathrm{in}}}_{\mathcal H^{N_0+1}_{M,M}}\norm{w}_{L^1}\norm{g}_{H^{1+\abs{\alpha}}}}{\jb{s}^{d/2+\sigma+1}}.
\end{align}
Since the right-hand side of \eqref{linearscatteringintegrand} is integrable in $s$, the asymptotic state $Q_\infty$ is well defined by
\begin{align*}
    \widehat Q_\infty(k,p)&:=\widehat{Q_{\mathrm{in}}}(k,p)\\
    &\quad-i\hbar^{d-1}(2\pi)^d\widehat w(k+p)\sbr{g(-\hbar p)-g(\hbar k)}\int_0^\infty e^{i\frac{\hbar}{2}(\ak^2-\ap^2)s}\widehat{\rho_{Q_{\mathrm{L}}}}(s,k+p)\,ds.
\end{align*}
Minkowski's inequality and \eqref{linearscatteringintegrand} then give
\begin{align*}
    \frac{1}{\hbar^{d/2}}\norm{\ak^{\sigma}\br{\mytextfrac{1}{\hbar}\nabla_p}^\alpha\sbr{e^{i\hbar k\cdot pt}\widehat{Q_{\mathrm{L}}}\br{t,\mytextfrac{k}{2}+p,\mytextfrac{k}{2}-p}-\widehat Q_{\infty}\br{\mytextfrac{k}{2}+p,\mytextfrac{k}{2}-p}}}_{L^2_{k,p}}&\leq\int_t^\infty\mathcal A_{\sigma,\alpha}(s)\,ds\\
    &\lesssim\frac{\norm{Q_{\mathrm{in}}}_{\mathcal H^{N_0+1}_{M,M}}}{\jb{t}^{d/2+\sigma}},
\end{align*}
where the dependence on $w$ and $g$ is absorbed into the implicit constant. Since this holds for all $\sigma+\abs{\alpha}\leq N_1$, the scattering estimate \eqref{scatteringstatement} follows.
\end{proof}

\section{Energy estimates on the density}\label{sectiondensity}  
This section initiates the nonlinear analysis by establishing the bootstrap improvements \eqref{I2} and \eqref{I4} on the density $\rho_Q$. The approach applies part of the linear theory from Section~\ref{linearisedsection} to the nonlinear equation.

\subsection{Linear damping in $L^2_t$} The first step is to derive an estimate for a general Volterra equation of the form \begin{align}\label{generalsystem}
    \phi(t,k)=H(t,k)-\int_0^t\widehat {\mathscr L}(t-s,k)\phi(s,k)\,ds,\quad \widehat {\mathscr L}(t,k)=\frac{2}{\hbar}\widehat w(k)\mathrm{sin}(\mytextfrac{\hbar}{2} t\ak^2)\widehat g\br{kt}.
\end{align}
\begin{proposition}\label{linearquantumdampingproposition}
    Fix $T^*>0$. 
    Let $m\in\R$ and $\sigma\geq0$. Suppose $w\in L^1(\R^d)$ and $g\in H^{\ceils{\sigma}+3/2+\delta_0}_{2+\ceils{d/2}}(\R^d)$ satisfy the uniform Penrose condition \eqref{Penrose} with constant ${c_{\mathrm{P}}}$. Assume \begin{align*}
        \norm{\ak^m\jb{k,kt}^\sigma H(t,k)}_{L^2_{t}([0,T^*])}<\infty.
    \end{align*}
    Then there exists a constant $C_{\mathrm{LD}}=C_{\mathrm{LD}}(\sigma,d,\delta_0,w,g,{c_{\mathrm{P}}})$, such that the solution $\phi(t,k)$ to the system \eqref{generalsystem} satisfies the pointwise-in-$k$ estimate
    \begin{align*}
        \norm{\ak^m\jb{k,kt}^\sigma\phi(t,k)}_{L^2_{t}([0,T^*])}\leq C_{\mathrm{LD}}\norm{\ak^m\jb{k,kt}^\sigma H(t,k)}_{L^2_t([0,T^*])}.
    \end{align*}
\end{proposition}
\begin{proof}The result follows by taking Laplace transforms and applying Plancherel's theorem. The key point is to control the resolvent \(\widetilde{\mathscr G}(\lambda, k)\) uniformly in \(k\) along the imaginary axis.

    Taking the Laplace transform of \eqref{generalsystem} yields\begin{align}\label{marnie}
        \widetilde\phi(\lam,k)=\widetilde {\mathscr G}(\lam,k)\widetilde H(\lam,k),
    \end{align}where $\widetilde{\mathscr G}$ is defined in \eqref{definitionofwidetildeG}. Assume first that $\sigma$ is an integer. For any $0\leq n\leq \sigma$,
    \begin{align}\nonumber
        \ak^n\p_\lam^n\widetilde\phi(\lam,k)&=\sum_{j\leq n}\begin{pmatrix}
            n\\j
        \end{pmatrix}\ak^j\p_\lam^j\widetilde {\mathscr G}(\lam,k)\ak^{n-j}\p_{\lam}^{n-j}\widetilde H(\lam,k)\\
        &=\widetilde {\mathscr G}(\lam,k)\ak^n\p_\lam^n \widetilde H(\lam,k)-\sum_{1\leq j\leq n}\begin{pmatrix}
            n\\j
        \end{pmatrix}\ak^j\p_\lam^j\widetilde {\mathscr G^r}(\lam,k)\ak^{n-j}\p_{\lam}^{n-j}\widetilde H(\lam,k),   \label{L2tlemma} \end{align} recalling that $\widetilde {\mathscr G}(\lam,k)=1-\widetilde {\mathscr G^r}(\lam,k)$.
For the first term on the right-hand side, note that $\abs{\widetilde {\mathscr G}(i\tau,k)}\leq {c_{\mathrm{P}}}^{-1}$, which follows from $w$ and $g$ satisfying the uniform Penrose condition (\ref{Penrose}). Along the imaginary axis $\lam=i\tau$, Lemma~\ref{Grdecay} can be applied to the other terms in \eqref{L2tlemma} to bound $\abs{k}^j\abs{\partial^j_\lam \widetilde{\mathscr G^r}(i\tau,k)}\lesssim1$. Therefore,
\begin{align*}
    \ak^n\abs{\p_\lam^n\widetilde\phi(i\tau,k)}&\lesssim\sum_{j\leq n} \ak^{j}\abs{\p_\lam^{j}\widetilde H(i\tau,k)}.
\end{align*}
        Taking the $L^2$ norm in $\tau$ and applying Plancherel's theorem yields\begin{align*}
            \norm{\abs{kt}^n\phi(t,k)}_{L^2_t([0,T^*])}\lesssim\norm{\jb{kt}^nH(t,k)}_{L^2_t([0,T^*])}.
        \end{align*}Moreover,\begin{align*}
\jb{k}^\sigma\abs{\widetilde\phi(i\tau,k)}=\jb{k}^\sigma\abs{\widetilde {\mathscr G}(i\tau,k)\widetilde H(i\tau,k)}\lesssim{c_{\mathrm{P}}}^{-1}\jb{k}^\sigma\abs{\widetilde H(i\tau,k)}.
        \end{align*}
        Using $\jb{k,kt}^\sigma\sim \jb{k}^\sigma+\abs{kt}^\sigma$, it follows that\begin{align*}
            \norm{\jb{k,kt}^\sigma \phi(t,k)}_{L^2_t([0,T^*])}\lesssim\norm{\jb{k,kt}^\sigma H(t,k)}_{L^2_t([0,T^*])}.
        \end{align*} Multiplying both sides of \eqref{marnie} by $\ak^m$ yields the result for integer $\sigma$. For non-integer $\sigma$, the result follows by interpolation between the integer cases $\lfloor\sigma\rfloor$ and $\ceils{\sigma}$ \cite{SteinWeiss1958}.
\end{proof}
    
\subsection{Bootstrap improvement \eqref{I2}}\label{sectionb2}
To prove the bootstrap improvement \eqref{I2}, the weighted norm is split into two contributions according to $\jb{\hbar k}^2 = 1 + \hbar^2 \ak^2$. Each contribution is treated separately. First, it is shown that under the bootstrap hypotheses \eqref{B1} and \eqref{B5}, there exists a constant $K_2'>0$ depending only on the fixed parameters in Proposition~\ref{bootstrapprop} and independent of $\hbar$, $T^*$, $\ep$, $\Qin$ and the bootstrap constants $K_i$, such that, for $\ep$ sufficiently small,
\begin{align}
    \norm{\ak^{1/2}\jb{k,kt}^{\sigma_4}\widehat{\rho_Q}(t,k)}_{L^2_t([0,T^*],L^2_k)}^2&\leq 2K_2'\ep^2.\label{B2original}
\end{align} Following this, it will be shown that there exists $K_2''>0$ such that, for $\ep$ sufficiently small, 
\begin{align}
    \norm{\hbar\ak^{3/2}\jb{k,kt}^{\sigma_4}\widehat{\rho_Q}(t,k)}_{L^2_{t}([0,T^*],L^2_k)}^2&\leq 2K_2''\ep^2.\label{B2new}
\end{align}Choosing $K_2:=K_2'+K_2''$ establishes the bootstrap improvement \eqref{I2}.

\subsubsection{Estimate for \eqref{B2original}}\label{3.2.1}
From \eqref{Nb}, the nonlinear equation for the density satisfies
\begin{align*}
    \widehat{\rho_Q}(t,k)=H(t,k)-\int_0^t \widehat{\mathscr L}(t-s,k)\widehat{\rho_Q}(s,k)\,ds,\quad  \widehat{\mathscr L}(t,k)=\frac{2}{\hbar}\widehat w(k)\mathrm{sin}(\mytextfrac{\hbar}{2} t\ak^2)\widehat g\br{kt},
\end{align*}
with
\begin{align*}
    H(t,k)&=\widehat{\rho_{{Q_{\mathrm{FS}}}}}(t,k)\\
    &\hspace{0.4cm}-\frac{2\hbar^{-1}}{(2\pi)^d}\int_0^t\int_\ell\widehat w(\ell)\widehat{\rho_Q}(s,\ell) \sin\br{\mytextfrac{\hbar}{2}\ell\cdot k(t-s)}\widehat{W[P]}(s,k-\ell,kt-\ell s)\,d\ell ds.
\end{align*} Under the assumptions of Theorem~\ref{maintheorem}, Proposition~\ref{linearquantumdampingproposition} implies
\begin{align*}
    \norm{\ak^{1/2}\jb{k,kt}^{\sigma_4}\widehat{\rho_Q}(t,k)}_{L^2_{t}([0,T^*],L^2_k)}^2&\leq C_{\mathrm{LD}}\br{\norm{\ak^{1/2}\jb{k,kt}^{\sigma_4}\widehat{\rho_{{Q_{\mathrm{FS}}}}}(t,k)}_{L^2_{t}([0,T^*],L^2_k)}^2+\mathcal {N}^{\mathrm{(B2),1}}},
\end{align*}where\begin{align*}
    \mathcal {N}^{\mathrm{(B2),1}}&=\int_{t=0}^{T^*}\int_k\Big|\ak^{1/2}\jb{k,kt}^{\sigma_4}\int_0^t\int_\ell\widehat w(\ell)\widehat{\rho_Q}(s,\ell) \hbar^{-1}\\
    &\hspace{2cm}\times\sin\br{\mytextfrac{\hbar}{2}\ell\cdot k(t-s)}\widehat {W[P]}\br{s,k-\ell,kt-\ell s}\,d\ell ds\Big|^2dkdt.
\end{align*}
The initial data term is controlled by Lemma~\ref{b2b4arenatural} and the assumption $\norm{\Qin}_{\mathcal H^{N_0}_{M,M}}\leq\ep$,
\begin{align*}
    \norm{\ak^{1/2}\jb{k,kt}^{\sigma_4}\widehat{\rho_{{Q_{\mathrm{FS}}}}}(t,k)}_{L^2_{t}([0,T^*],L^2_k)}^2\lesssim\sum_{\aal,\abs{\beta}\leq M}\norm{\bm{\xi}^\alpha\bm{x}^\beta \Qin}_{\mathcal H^{N_0}}^2\lesssim\ep^2,
\end{align*}provided that $\sigma_4<N_0-(d+1)/2$. For the nonlinear term $\mathcal {N}^{\mathrm{(B2),1}}$, the triangle inequality $\jb{k,kt}^{\sigma_4}\lesssim \jb{k-\ell,kt-\ell s}^{\sigma_4}+\jb{\ell,\ell s}^{\sigma_4}$ is used to divide into \textit{transport} and \textit{reaction} parts $\mathcal{N}^{\mathrm{(B2),1}}\lesssim \mathcal{T}^{\mathrm{(B2),1}}+\mathcal{R}^{\mathrm{(B2),1}}$, where
\begin{align*}
    \mathcal{T}^{\mathrm{(B2),1}}&\lesssim\int_{t=0}^{T^*}\int_k\Big[\int_0^t\int_\ell\ak^{1/2}\jb{k-\ell,kt-\ell s}^{\sigma_4}\abs{\widehat w(\ell)}\abs{\widehat{\rho_Q}(s,\ell)} \abs{\ell}\\
    &\hspace{3cm}\times\abs{k(t-s)}\abs{\widehat {W[P]}\br{s,k-\ell,kt-\ell s}}\,d\ell ds\Big]^2dkdt,\\
    \mathcal{R}^{\mathrm{(B2),1}}&\lesssim\int_{t=0}^{T^*}\int_k\Big[\int_0^t\int_\ell\ak^{1/2}\jb{\ell,\ell s}^{\sigma_4}\abs{\widehat w(\ell)}\abs{\widehat{\rho_Q}(s,\ell)} \abs{\ell}\\
    &\hspace{3cm}\times\abs{k(t-s)}\abs{\widehat {W[P]}\br{s,k-\ell,kt-\ell s}}\,d\ell ds\Big]^2dkdt.
\end{align*}The regularity of the solution moves along the flow, allowing control of the transport term through the bootstrap assumption \eqref{B1}. The reaction term, however, introduces the \textit{time response kernel} and requires careful handling of resonances.

\subsubsection{Control of transport term}\label{controloftransportB2original}
For $\mathcal{T}^{\mathrm{(B2),1}}$, by controlling $\abs{k(t-s)}=\abs{kt-\ell s-s(k-\ell)}\lesssim\jb{s(k-\ell),kt-\ell s}$, the weights are absorbed onto $P$ using Definition \ref{toolboxdefinition},
\begin{align*}
    &\hspace{-1cm}\jb{s(k-\ell),kt-\ell s}\jb{k-\ell,kt-\ell s}^{\sigma_4}\abs{\widehat {W[P]}\br{s,k-\ell,kt-\ell s}}\\
    &\hspace{2cm}=\abs{\reallywidehat{W\sbr{\jb{s\bm{\nabla}_x,\bm{\nabla}_\xi}\jb{\bm{\nabla}_x,\bm{\nabla}_\xi}^{\sigma_4}P}}\br{s,k-\ell,kt-\ell s}}.
\end{align*}This leads to
\begin{equation}\label{TforB2original}
\begin{split}\mathcal{T}^{\mathrm{(B2),1}}&\lesssim\int_{t=0}^{T^*}\int_k\Big[\int_0^t\int_\ell\ak^{1/2}\abs{\widehat w(\ell)}\abs{\widehat{\rho_Q}(s,\ell)} \abs{\ell}\\
    &\hspace{2cm}\times\abs{\reallywidehat {W[\jb{s\bm{\nabla}_x,\bm{\nabla}_\xi}\jb{\bm{\nabla}_x,\bm{\nabla}_\xi}^{\sigma_4}P]}\br{s,k-\ell,kt-\ell s}}\,d\ell ds\Big]^2dkdt.
    \end{split}
\end{equation}
Applying Cauchy--Schwarz in $s$ and $\ell$ and the bound $\abs{\widehat w(\ell)}\lesssim1$,
\begin{align*}
    &\mathcal{T}^{\mathrm{(B2),1}}\\
    &\lesssim\int_{t=0}^{T^*}\int_k\sbr{\int_{s=0}^t\int_\ell\jb{s}^{d-1/2}\abs{\ell}\abs{\widehat{\rho_Q}(s,\ell)}\,d\ell ds}\\
    &\times\sbr{\int_{s=0}^t\int_\ell\jb{s}^{-(d-1/2)}\abs{\widehat{\rho_Q}(s,\ell)} \abs{\ell}\ak\abs{\reallywidehat{W\sbr{\jb{s\bm{\nabla}_x,\bm{\nabla}_\xi}\jb{\bm{\nabla}_x,\bm{\nabla}_\xi}^{\sigma_4}P}}\br{s,k-\ell,kt-\ell s}}^2\,d\ell ds}dkdt.
\end{align*}
Bootstrap assumption \eqref{B5} (via Lemma~\ref{rhodecayfromB5}) implies \begin{align*}
    \int_{s=0}^t\jb{s}^{d-1/2}\int_\ell\abs{\ell}\abs{\widehat{\rho_Q}(s,\ell)}\,d\ell ds\lesssim\sqrt{K_5}\ep\int_0^\infty\frac{\jb{s}^{d-1/2}}{\jb{s}^{d+1}}\,ds\lesssim\sqrt{K_5}\ep
\end{align*}provided that $\sigma_1>d+1$, which allows control over the first $s$-$\ell$ integral. For the second, Fubini's theorem is applied:
\begin{align*}
    \mathcal{T}^{\mathrm{(B2),1}}&\lesssim\sqrt{K_5}\ep\int_{s=0}^{T^*} \int_\ell \jb{s}^{d-1/2}\abs{\ell}\abs{\widehat{\rho_Q}(s,\ell)} d\ell ds\\
&\times\sup_{s\in[0,T^*],\ell\in\R^d}\jb{s}^{-2d+1}\int_k\int_{t=-\infty}^\infty\ak\abs{\reallywidehat{W\sbr{\jb{s\bm{\nabla}_x,\bm{\nabla}_\xi}\jb{\bm{\nabla}_x,\bm{\nabla}_\xi}^{\sigma_4}P}}\br{s,k-\ell,kt-\ell s}}^2\,dtdk.
\end{align*}
The $s$-$\ell$ integral is controlled again using \eqref{B5}. Finally, the change of variable $t\mapsto t/\ak$ is made, and the $L^2$ trace lemma (Lemma~\ref{L2trace}) is applied to the $\eta$-variable:
\begin{align*}
    &\mathcal{T}^{\mathrm{(B2),1}}\\
    &\lesssim K_5\ep^2\sup_{s\in[0,T^*],\ell\in\R^d}\jb{s}^{-2d+1}\int_k\int_{t=-\infty}^\infty\abs{\reallywidehat{W\sbr{\jb{s\bm{\nabla}_x,\bm{\nabla}_\xi}\jb{\bm{\nabla}_x,\bm{\nabla}_\xi}^{\sigma_4}P}}\br{s,k-\ell,\mytextfrac{k}{\ak}t-\ell s}}^2\,dtdk\\
    &\lesssim K_5\ep^2\sup_{s\in[0,T^*]}\jb{s}^{-2d+1}\int_k\sup_{\omega\in \mathbb S^{d-1},x\in\R^d}\int_{t=-\infty}^\infty\abs{\reallywidehat{W\sbr{\jb{s\bm{\nabla}_x,\bm{\nabla}_\xi}\jb{\bm{\nabla}_x,\bm{\nabla}_\xi}^{\sigma_4}P}}\br{s,k,\omega t-x}}^2\,dtdk\\
    &\lesssim K_5\ep^2\sup_{s\in[0,T^*]}\jb{s}^{-2d+1}\int_k\norm{\jb{\nabla_\eta}^{q}\reallywidehat{W\sbr{\jb{s\bm{\nabla}_x,\bm{\nabla}_\xi}\jb{\bm{\nabla}_x,\bm{\nabla}_\xi}^{\sigma_4}P}}\br{s,k,\eta}}_{L^2_\eta}^2\,dk\\
    &\lesssim K_5\ep^2\sup_{s\in[0,T^*]}\jb{s}^{-2d+1}\sum_{\aal\leq M}\norm{v^\alpha W\sbr{\jb{s\bm{\nabla}_x,\bm{\nabla}_\xi}\jb{\bm{\nabla}_x,\bm{\nabla}_\xi}^{\sigma_4}P}(s,z,v)}_{L^2_{z,v}}^2,
\end{align*}where $(d-1)/2<q\leq M$.
Using Definition~\ref{toolboxdefinition} and Lemma~\ref{L2normofWigner}, the resulting norm coincides with the quantity controlled by the bootstrap assumption \eqref{B1}. Hence, it follows that the transport term satisfies\begin{align*}
    \mathcal{T}^{\mathrm{(B2),1}}\lesssim K_1K_5\ep^4.
\end{align*}

\subsubsection{Control of resonances}\label{resonancessection}
For the reaction term $\mathcal{R}^{\mathrm{(B2),1}}$, the bound $\abs{k(t-s)}\lesssim\jb{s}\jb{k-\ell,kt-\ell s}$ is utilised, and bootstrap assumption \eqref{B5} is used to control
\begin{align*}
    \mathcal{R}^{\mathrm{(B2),1}}&\lesssim K_5\ep^2\int_{t=0}^{T^*}\int_k\sbr{\int_0^t\int_\ell\frac{\jb{s}\abs{\widehat w(\ell)}\ak^{1/2} \abs{\ell}^{1/2}}{\jb{k-\ell,kt-\ell s}^{\sigma_1-1}}\abs{\ell}^{1/2}\jb{\ell,\ell s}^{\sigma_4}\abs{\widehat{\rho_Q}(s,\ell)}\,d\ell ds}^2dkdt.
\end{align*} Consequently,  \begin{align}\label{tobeapplied}\mathcal{R}^{\mathrm{(B2),1}}\lesssim K_5\ep^2\norm{\overline R\sbr{\abs{k}^{1/2}\jb{k,kt}^{\sigma_4}\abs{\widehat{\rho_Q}(t,k)}}}_{L^2_{t}([0,T^*],L^2_k)}^2\end{align} with the integral operator $\overline R\colon L^2_{t}([0,T^*],L^2_k)\to L^2_{t}([0,T^*],L^2_k)$ defined by\begin{align}\label{Rintegraloperator}
    \overline R[f](t,k)=\int_0^t\int_\ell\frac{\jb{s}\abs{\widehat w(\ell)}\ak^{1/2} \abs{\ell}^{1/2}}{\jb{k-\ell,kt-\ell s}^{\sigma_1-1}}f(s,\ell)\,d\ell ds.
\end{align}
Through this kernel, density modes at earlier times $0\leq s\leq t$ can influence the $k$-th mode of the density at time $t$. Preventing a cascade of resonances is crucial: the $\ell$-th mode of the density at time $s$ may sequentially drive the $k$-th mode at time $t$, and so on, potentially causing growth in the nonlinear interaction at arbitrarily large times. Schur's test implies that 
\begin{align}\label{Schurintegrals}\begin{split}
    \norm{\overline R}_{L^2_{t}([0,T^*],L^2_k)\to L^2_{t}([0,T^*],L^2_k)}^2&\leq \br{\sup_{t\in[0,T^*],k\in\R^d}\int_0^t\int_\ell \frac{\jb{ s }\abs{\widehat w(\ell)}\ak^{1/2}\abs{\ell}^{1/2}}{\jb{k-\ell,kt-\ell  s }^{\sigma_1-1}}\,d\ell d s }\\
    &\hspace{0.5cm}\times\br{\sup_{ s \in[0,T^*],\ell\in\R^d}\int_{ s }^{T^*}\int_k \frac{\jb{ s }\abs{\widehat w(\ell)}\ak^{1/2}\abs{\ell}^{1/2}}{\jb{k-\ell,kt-\ell  s }^{\sigma_1-1}}\,dk d t }.\end{split}
\end{align}
The following lemma provides control over these integrals, analogous to~\cite[Lemmas 3.1 and 3.2]{BedrossianMasmoudiMouhot2018}, with the proof extended to $d \geq 3$. The main challenge in managing the integrals in \eqref{Schurintegrals} arises from the $\jb{s}$-growth in the numerator, which becomes difficult to handle due to potential cancellations in $\jb{k-\ell, kt-\ell s}$ when $k$ and $\ell$ are nearly collinear. To address this, a \textit{resonant region} in $\R^d$ is identified, where $k$ and $\ell$ are nearly collinear, and which diminishes over time. The integral over this resonant region converges as it shrinks at a sufficient rate to ensure convergence in the time variable. Outside the resonant region, where $\ell$ and $k$ are sufficiently non-collinear, additional decay is extracted from the denominator through powers of $\abs{kt-\ell s}$. 
\begin{lemma}[Control of resonances]\label{lemmaschurcontrolofresonances} Fix dimension $d\geq 3$ and let $T^*\geq0$, $m\geq\max\{1,(d-3/2)/2\}$ and $\beta\geq d+6$. The following integral inequalities hold:
\begin{align}\label{one}
    \sup_{t\in[0,T^*],k\in\R^d}\int_0^t\int_{\ell\in\R^d}\frac{\jb{ s }\ak^{1/2}\abs{\ell}^{1/2}}{\jb{\ell}^m\jb{k-\ell,kt-\ell s }^{\beta}}d\ell d s &\lesssim1
    \intertext{and}
\label{two}\sup_{s\in[0,T^*],\ell\in\R^d}\int_{s}^{T^*}\int_{k\in\R^d}\frac{\jb{ s }\ak^{1/2}\abs{\ell}^{1/2}}{\jb{\ell}^m\jb{k-\ell,kt-\ell s }^{\beta}}dk d t &\lesssim1.    
\end{align}
Schur's test thus shows that the operator $\overline{R}$, defined in \eqref{Rintegraloperator}, satisfies $\norm{\overline R}_{L^2_{t}([0,T^*],L^2_k)\to L^2_{t}([0,T^*],L^2_k)}\lesssim1$.
\end{lemma}
\begin{proof}
To establish inequality \eqref{one}, fix $t\in[0,T^*]$ and $k\in\R^d$. If $k=0$, the integral vanishes, so assume $k\neq0$. The integral over early times is controlled as follows:
\begin{align*}
    \int_0^1\int_\ell\frac{\jb{ s }\ak^{1/2}\abs{\ell}^{1/2}}{\jb{\ell}^m\jb{k-\ell,kt-\ell s }^{\beta}}d\ell d s \lesssim\int_\ell\frac{\ak^{1/2}\abs{\ell}^{1/2}}{\jb{\ell}^m\jb{k-\ell}^{\beta}}\,d\ell\lesssim\int_\ell\frac{\abs{\ell}^{1/2}}{\jb{\ell}^{m-{1/2}}\jb{k-\ell}^{\beta-\half}}\,d\ell\lesssim1,
\end{align*}where $\jb{k}\lesssim\jb{\ell}\jb{k-\ell}$ was used. 

Assume $T^*>1$. Decompose $\ell\in\R^d$ into parallel and orthogonal components relative to $k$ as $\ell_\para:=(k\cdot\ell)k/\ak^2$ and $\ell_\perp:=\ell-\ell_\para$. Define $I_R\subset\R^d$ as a cylinder around $k$ with radius shrinking over time, and set $I_{NR}=I_R^c$, where
\begin{align*}
    I_R=\set{\ell:\abs{\ell_\perp}<\frac{\ak^b}{\jb{ s }^\zeta}}\qaq I_{NR}=\set{\ell:\abs{\ell_\perp}\geq\frac{\ak^b}{\jb{ s }^\zeta}},
\end{align*}with $b:=(1-\zeta)/2$, where $\zeta$ is chosen to satisfy
\begin{align*}
    \frac{2}{d-\mytextfrac{1}{2}}<\zeta<1\qaq x:=\frac{1}{1-\zeta}<\frac{11}{2}.
\end{align*}
Such a choice is possible since $(1-\zeta)^{-1}\to(2d-1)/(2d-5)\leq5$ as $\zeta\downarrow2/(d-1/2)$, for all $d\geq3$.

The integral over the cylinder $I_R$ satisfies, using $\abs{\ell}^{1/2}\lesssim\abs{\ell_\para}^{1/2}+\abs{\ell_\perp}^{1/2}$,
\begin{align*}
    \mathscr{I}_R:=\int_1^t\int_{I_R}\frac{\jb{ s }\ak^{1/2}\abs{\ell}^{1/2}}{\jb{\ell}^m\jb{k-\ell,kt-\ell s}^{\beta}}\,d\ell d s\lesssim\int_1^t\int_{I_R}\frac{\jb{ s }\ak^{1/2}}{\jb{\ell_\para}^m\jb{k-\ell_\para,kt-\ell_\para  s }^\beta}\br{\abs{\ell_\para}^{1/2}+\abs{\ell_\perp}^{1/2}}\,d\ell d s.
    \end{align*}
Applying the bound for $\abs{\ell_\perp}$ on $I_R$ and the volume of a cylinder on $\R^d$ gives
\begin{align*}
    \mathscr I_R\lesssim\int_1^t\int_{-\infty}^\infty\frac{\jb{ s }\ak^{1/2}}{\jb{\ell_\para}^m\jb{k-\ell_\para,kt-\ell_\para  s }^\beta}\frac{\ak^{(d-1)b}}{\jb{ s }^{(d-1)\zeta}}\br{\abs{\ell_\para}^{1/2}+\frac{\ak^{\frac{b}{2}}}{\jb{ s }^{\frac{\zeta}{2}}}}\,d\ell_\para d s,
\end{align*}leading to two contributions $\mathscr I_1$ and $\mathscr I_2$. For $\mathscr I_1$, the change of variable $ s \mapsto s /\abs{\ell_\para}$ and $(d-1)\zeta\geq1$ yield convergence under the conditions $\beta+m-2(d-1)b>3/2$ and $\beta,m\geq1/2+(d-1)b$. The second term $\mathscr I_2$ converges due to the choice of $\zeta$, ensuring that $1/\jb{ s }^{(d-1/2)\zeta-1}$ is integrable, for $\beta+m>2+2(d-1/2)b$ and $m,\beta\geq 1/2+(d-1/2)b$. 

On $I_{NR}$, for $t\geq1$, it holds that $\abs{kt-\ell s }\gtrsim s \abs{\ell_\perp}\gtrsim\ak^b\jb{ s }^{1-\zeta}$, so\begin{align*}
    \mathscr I_{NR}&:=\int_1^t\int_{I_{NR}}\frac{\jb{ s }\ak^{1/2}\abs{\ell}^{1/2}}{\jb{\ell}^m\jb{k-\ell,kt-\ell s}^{\beta}}\,d\ell d s \\
    &\lesssim\int_1^t\int_{I_{NR}}\frac{\jb{s}\ak^{1/2}\abs{\ell}^{1/2}}{\jb{\ell}^m\jb{k-\ell,kt-\ell s}^{\beta-\frac{1}{2b}}\ak^{1/2}\jb{s}^{\frac{1}{2b}(1-\zeta)}}\,d\ell d s .
\end{align*}
Using $(1-\zeta)/(2b)-1=0$, and applying the change of variable $s\mapsto s/\abs{\ell}$, the integral $\mathscr I_{NR}$ converges for $\beta+m-1/(2b)-1/2>d$. 

To establish inequality \eqref{two}, it is observed that:\begin{align*}
    \sup_{s\in[0,T^*],\ell\in\R^d}\int_{s}^{T^*}&\int_{k\in\R^d}\frac{\jb{ s }\ak^{1/2}\abs{\ell}^{1/2}}{\jb{\ell}^m\jb{k-\ell,kt-\ell s }^{\beta}}\,dk d t \\
    &\leq \sup_{s\in[0,T^*],\ell\in\R^d}\int_{0}^{T^*}\int_{k\in\R^d}\frac{\jb{t }\ak^{1/2}\abs{\ell}^{1/2}}{\jb{\ell}^m\jb{k-\ell,kt-\ell s }^{\beta}}\,dk d t. 
\end{align*}
For fixed $s \in [0, T^*]$ and $\ell \in \mathbb{R}^d$, the integral over early times is estimated as
\begin{align*}
    \int_{0}^{1}\int_{k}\frac{\jb{t }\ak^{1/2}\abs{\ell}^{1/2}}{\jb{\ell}^m\jb{k-\ell,kt-\ell s }^{\beta}}dk d t\lesssim\int_k\frac{\ak^{1/2}\abs{\ell}^{1/2}}{\jb{\ell}^m\jb{k-\ell}^{\beta}}\,dk\lesssim\int_k\frac{\ak^{1/2}}{\jb{k}^{\beta}}\,dk\lesssim1.
\end{align*}

Assume $T^* > 1$. If $\ell=0$, the integral vanishes, so assume $\ell\neq0$. Decompose $k\in\R^d$ into parallel and orthogonal components relative to $\ell$, defining $k_\para:=(\ell\cdot k)\ell/\abs{\ell}^2$ and $k_\perp:=k-k_\para$, and partition
\begin{align*}
J_R=\set{k:\abs{k_\perp}<\frac{\abs{\ell}^b}{\jb{ t }^\zeta}}\qaq J_{NR}=\set{k:\abs{k_\perp}\geq\frac{\abs{\ell}^b}{\jb{ t }^\zeta}},
\end{align*}with $\zeta$, $b$ and $x$ as in the proof of~\eqref{one}.

The integral over $J_R$ satisfies
\begin{align*}
    \mathscr J_{R}&:=\int_1^{T^*}\int_{J_{R}} \frac{\jb{t}\abs{\ell}^{1/2}}{\jb{\ell}^m\jb{k_\para-\ell,k_\para t-\ell s}^\beta}\br{\abs{k_\para}^{1/2}+\abs{k_\perp}^{1/2}}\,dk dt\\
    &\lesssim\int_1^{T^*}\int_{\R} \frac{\abs{\ell}^{1/2}}{\jb{\ell}^m\jb{k_\para-\ell,k_\para t-\ell s}^\beta}\frac{\abs{\ell}^{(d-1)b}}{\jb{t}^{(d-1)\zeta-1}}\br{\abs{k_\para}^{1/2}+\frac{\abs{\ell}^{\frac{b}{2}}}{\jb{t}^{\frac{\zeta}{2}}}}\,dk_\para dt,
    \end{align*}leading to two contributions, $\mathscr J_1$ and $\mathscr J_2$. Following the change of variable $t\mapsto t/\abs{k_\para}$ in $\mathscr J_1$, the integrals converge under the conditions $m\geq(d-1/2)b+1/2$ and $(d-1/2)\zeta - 1 > 1$.

    On $J_{NR}$, for $t\geq1$, it holds that $\abs{kt-\ell s}\gtrsim t\abs{k_\perp}\gtrsim\abs{\ell}^b\jb{t}^{1-\zeta}$, so
    \begin{align*}
        \mathscr J_{NR}&:=\int_1^{T^*}\int_{J_{NR}}\frac{\jb{t}\abs{\ell}^{1/2}\ak^{1/2}}{\jb{\ell}^m\jb{k-\ell,kt-\ell s}^{\beta}}\,dkdt\\
        &\lesssim\int_1^{T^*}\int_{J_{NR}}\frac{\abs{\ell}^{1/2}\ak^{1/2}}{\jb{\ell}^m\jb{k-\ell,kt-\ell s}^{\beta-x}\abs{\ell}^{bx}\jb{t}^{x(1-\zeta)-1}}\,dkdt.
    \end{align*}Since $x(1-\zeta)=1$ and $bx=1/2$, the change of variable $t\mapsto t/\ak$ shows that the integral is uniformly bounded provided $\beta-x-1/2>d$. By the choice $x<11/2$, this condition follows from $\beta\geq d+6$. The remaining conditions obtained above follow from $\beta\geq d+6$ and $m\geq\max\{1,(d-3/2)/2\}$.
\end{proof}
Applying Lemma~\ref{lemmaschurcontrolofresonances} to \eqref{tobeapplied} yields the bound\begin{align*}\mathcal{R}^{\mathrm{(B2),1}}\lesssim K_5\ep^2\norm{\abs{k}^{1/2}\jb{k,kt}^{\sigma_4}\widehat{\rho_Q}(t,k)}_{L^2_{t}([0,T^*],L^2_k)}^2.\end{align*}

\subsubsection{Conclusion of \eqref{B2original}}
It has been shown that there exists a constant $K_{\mathrm{LD}}' >0$, depending only on the fixed parameters in Proposition~\ref{bootstrapprop} and independent of $\hbar$, $T^*$, $\ep$, $\Qin$ and the bootstrap constants $K_i$, such that 
\begin{align*}
&\norm{\ak^{1/2}\jb{k,kt}^{\sigma_4}\widehat{\rho_Q}(t,k)}_{L^2_{t}([0,T^*],L^2_k)}^2\\
&\qquad\leq K_{\mathrm{LD}}'\br{\ep^2+K_1K_5\ep^4+K_5\ep^2\norm{\abs{k}^{1/2}\jb{k,kt}^{\sigma_4}\widehat{\rho_Q}(t,k)}_{L^2_{t}([0,T^*],L^2_k)}^2}.
\end{align*}
For $\ep$ sufficiently small so that $K_{\mathrm{LD}}' K_5 \ep^2 \leq 1/2$, the final term may be absorbed into the left-hand side, yielding
\begin{align*}
\norm{\ak^{1/2}\jb{k,kt}^{\sigma_4}\widehat{\rho_Q}(t,k)}_{L^2_{t}([0,T^*],L^2_k)}^2\leq 2K_{\mathrm{LD}}'\br{\ep^2+K_1K_5\ep^4}.
\end{align*} Choosing $K_1 K_5 \ep^2 \leq 1$ and defining $K_2' := 2K_{\mathrm{LD}}'$, the estimate \eqref{B2original} follows.

\subsubsection{Estimate for \eqref{B2new}}
Following the strategy of Section~\ref{3.2.1},
\begin{align*}
    \norm{\hbar\ak^{3/2}\jb{k,kt}^{\sigma_4}\widehat{\rho_Q}(t,k)}_{L^2_{t}([0,T^*],L^2_k)}^2&\leq C_{\mathrm{LD}}\br{\norm{\hbar\ak^{3/2}\jb{k,kt}^{\sigma_4}\widehat{\rho_{{Q_{\mathrm{FS}}}}}(t,k)}_{L^2_{t}([0,T^*],L^2_k)}^2+\mathcal {N}^{\mathrm{(B2),2}}},
\end{align*}where\begin{align*}
    \mathcal {N}^{\mathrm{(B2),2}}&=\int_{t=0}^{T^*}\int_k\Big|\hbar\ak^{3/2}\jb{k,kt}^{\sigma_4}\int_0^t\int_\ell\widehat w(\ell)\widehat{\rho_Q}(s,\ell) \hbar^{-1}\\
    &\hspace{2cm}\times\sin\br{\mytextfrac{\hbar}{2}\ell\cdot k(t-s)}\widehat {W[P]}\br{s,k-\ell,kt-\ell s}\,d\ell ds\Big|^2dkdt.
\end{align*} The initial data term is controlled by Lemma~\ref{b2b4arenatural} and the assumption on $\Qin$. For the nonlinear term, observe that the additional factor of $\hbar\ak$ in \eqref{B2new} cancels the $\hbar^{-1}$ denominator. The triangle inequality $\jb{k,kt}^{\sigma_4}\lesssim \jb{k-\ell,kt-\ell s}^{\sigma_4}+\jb{\ell,\ell s}^{\sigma_4}$ is used to divide into $\mathcal{N}^{\mathrm{(B2),2}}\lesssim \mathcal{T}^{\mathrm{(B2),2}}+\mathcal{R}^{\mathrm{(B2),2}}$, where
\begin{align*}
    \mathcal {T}^{\mathrm{(B2),2}}&=\int_{t=0}^{T^*}\int_k\Big[\int_0^t\int_\ell\ak^{3/2}\jb{k-\ell,kt-\ell s}^{\sigma_4}\abs{\widehat w(\ell)}\abs{\widehat{\rho_Q}(s,\ell)}\\
    &\hspace{2cm}\times\abs{\sin\br{\mytextfrac{\hbar}{2}\ell\cdot k(t-s)}}\abs{\widehat {W[P]}\br{s,k-\ell,kt-\ell s}}\,d\ell ds\Big]^2dkdt,\\
    \mathcal {R}^{\mathrm{(B2),2}}&=\int_{t=0}^{T^*}\int_k\Big[\int_0^t\int_\ell\ak^{3/2}\jb{\ell,\ell s}^{\sigma_4}\abs{\widehat w(\ell)}\abs{\widehat{\rho_Q}(s,\ell)} \\
    &\hspace{2cm}\times\abs{\sin\br{\mytextfrac{\hbar}{2}\ell\cdot k(t-s)}}\abs{\widehat {W[P]}\br{s,k-\ell,kt-\ell s}}\,d\ell ds\Big]^2dkdt.
\end{align*}

To control the transport part, the bound $\abs{\sin\br{\mytextfrac{\hbar}{2}\ell\cdot k(t-s)}}\leq 1$ is used. The additional factor of $\ak$ in \eqref{B2new} is decomposed into $\ak\leq \abs{\ell}+\abs{k-\ell}$. This leads to two terms, with $\mathcal {T}^{\mathrm{(B2),2}}\lesssim\mathcal {T}^{\mathrm{(B2),2}}_1+\mathcal {T}^{\mathrm{(B2),2}}_2$ where
\begin{align*}
    \mathcal {T}^{\mathrm{(B2),2}}_1&=\int_{t=0}^{T^*}\int_k\Big[\int_0^t\int_\ell\ak^{{1/2}}\abs{\ell}\abs{\widehat w(\ell)}\abs{\widehat{\rho_Q}(s,\ell)}\abs{\reallywidehat {W[\jb{\bm{\nabla}_x,\bm{\nabla}_{\xi}}^{\sigma_4}P]}\br{s,k-\ell,kt-\ell s}}\,d\ell ds\Big]^2dkdt,\\
    \mathcal {T}^{\mathrm{(B2),2}}_2&=\int_{t=0}^{T^*}\int_k\Big[\int_0^t\int_\ell\ak^{{1/2}}\abs{\widehat w(\ell)}\abs{\widehat{\rho_Q}(s,\ell)}\abs{\reallywidehat {W[\abs{\bm{\nabla}_x}\jb{\bm{\nabla}_x,\bm{\nabla}_{\xi}}^{\sigma_4}P]}\br{s,k-\ell,kt-\ell s}}\,d\ell ds\Big]^2dkdt.
\end{align*}
The contribution $\mathcal {T}^{\mathrm{(B2),2}}_1$ is controlled by \eqref{TforB2original}, and hence 
\begin{align*}
    \mathcal {T}^{\mathrm{(B2),2}}_1\lesssim K_1K_5\ep^4.
\end{align*}
For $\mathcal {T}^{\mathrm{(B2),2}}_2$, observing that
\begin{align*}
    \int_0^t\int_\ell \jb{s}^{d-3/2}\abs{\widehat{\rho_Q}(s,\ell)}\,d\ell ds\lesssim\sqrt{K_5}\ep\int_0^t\int_\ell \frac{\jb{s}^{d-3/2}}{\jb{\ell,\ell s}^{\sigma_1}}\,d\ell ds\lesssim \sqrt{K_5}\ep
\end{align*}and proceeding as in Section~\ref{controloftransportB2original} leads to
\begin{align*}
    \mathcal {T}^{\mathrm{(B2),2}}_2\lesssim K_5\ep^2\sup_{s\in[0,T^*]}\jb{s}^{-2d+3}\sum_{\aal\leq M}\norm{v^\alpha W\sbr{\abs{\bm{\nabla}_x}\jb{\bm{\nabla}_x,\bm{\nabla}_\xi}^{\sigma_4}P}(s,z,v)}_{L^2_{z,v}}^2.
\end{align*}
Using bootstrap assumption \eqref{B1}, 
\begin{align*}
    \sup_{s\in[1,T^*]}\jb{s}^{-2d+3}\sum_{\aal\leq M}&\norm{\vphantom{2_2^2}v^\alpha W\sbr{\abs{\bm{\nabla}_x}\jb{\bm{\nabla}_x,\bm{\nabla}_\xi}^{\sigma_4}P}(s,z,v)}_{L^2_{z,v}}^2\\
    &\lesssim\sup_{s\in[1,T^*]}\jb{s}^{-2d+1}\sum_{\aal\leq M}\norm{\vphantom{2_2^2}v^\alpha W\sbr{\abs{s\bm{\nabla}_x}\jb{\bm{\nabla}_x,\bm{\nabla}_\xi}^{\sigma_4}P}(s,z,v)}_{L^2_{z,v}}^2\\
    &\lesssim K_1\ep^2,
\end{align*}
while, for $s\in[0,\min\set{1,T^*}]$, the spatial-derivative estimate \eqref{B1zconc}, which is derived from the bootstrap assumptions \eqref{B1}, \eqref{B2}, \eqref{B3} and \eqref{B5} alone, together with the choice of $K_1$ and the smallness condition $\ep\leq(\sqrt{K_2K_3}+\sqrt{K_1K_5})^{-1}$ made in Section~\ref{sectionb1}, gives
\begin{align*}
    \sup_{s\in[0,\min\set{1,T^*}]}\jb{s}^{-2d+3}\sum_{\aal\leq M}\norm{\vphantom{2_2^2}v^\alpha W\sbr{\abs{\bm{\nabla}_x}\jb{\bm{\nabla}_x,\bm{\nabla}_\xi}^{\sigma_4}P}(s,z,v)}_{L^2_{z,v}}^2\lesssim K_1\ep^2.
\end{align*}
Altogether, the transport term satisfies
\begin{align*}
    \mathcal {T}^{\mathrm{(B2),2}}\lesssim K_1K_5\ep^4.
\end{align*}

For the reaction term $\mathcal{R}^{\mathrm{(B2),2}}$, bounding
\begin{align*}
    \abs{\sin\br{\mytextfrac{\hbar}{2}\ell\cdot k(t-s)}}\lesssim\hbar\abs{\ell}\abs{k(t-s)}
\end{align*} and arguing as in Section~\ref{resonancessection} yields
\begin{align*}
    \mathcal{R}^{\mathrm{(B2),2}}&\lesssim K_5\ep^2\int_{t=0}^{T^*}\int_k\sbr{\int_0^t\int_\ell\frac{\jb{s}\abs{\widehat w(\ell)}\hbar\ak^{3/2} \abs{\ell}^{1/2}}{\jb{k-\ell,kt-\ell s}^{\sigma_1-1}}\abs{\ell}^{1/2}\jb{\ell,\ell s}^{\sigma_4}\abs{\widehat{\rho_Q}(s,\ell)}\,d\ell ds}^2dkdt.
\end{align*}
Splitting the additional weight using the triangle inequality $\hbar\ak \leq \hbar\abs{\ell} + \hbar\abs{k - \ell}$ yields two contributions: $\mathcal{R}^{\mathrm{(B2),2}}\lesssim\mathcal{R}^{\mathrm{(B2),2}}_1+\mathcal{R}^{\mathrm{(B2),2}}_2$. The first contribution is exactly
\begin{align*}
    \mathcal{R}^{\mathrm{(B2),2}}_1=K_5\ep^2\norm{\overline{R}\sbr{\hbar\abs{k}^{3/2}\jb{k,kt}^{\sigma_4}\abs{\widehat{\rho_Q}(t,k)}}}_{L^2_{t}([0,T^*],L^2_k)}^2,
\end{align*}
where the integral operator $\overline R\colon L^2_{t}([0,T^*],L^2_k)\to L^2_{t}([0,T^*],L^2_k)$ is defined in \ref{Rintegraloperator}.  Lemma~\ref{lemmaschurcontrolofresonances} implies $\norm{\overline R}_{L^2_{t}([0,T^*],L^2_k)\to L^2_{t}([0,T^*],L^2_k)}\lesssim1$, and hence \begin{align*}
    \mathcal{R}^{\mathrm{(B2),2}}_1\lesssim K_5\ep^2\norm{\hbar\abs{k}^{3/2}\jb{k,kt}^{\sigma_4}\widehat{\rho_Q}(t,k)}_{L^2_{t}([0,T^*],L^2_k)}^2.
\end{align*} This term will be absorbed into the left-hand side. For the second contribution, use that $\hbar\in (0,1]$ and $\abs{\widehat w(\ell)}\lesssim\jb{\ell}^{-m}$ with $m=M-1/2$, by \eqref{wassumption}, and absorb $\abs{k-\ell}$ into the denominator:
\begin{align*}
    \mathcal{R}^{\mathrm{(B2),2}}_2\lesssim K_5\ep^2\int_{t=0}^{T^*}\int_k\sbr{\int_0^t\int_\ell\frac{\jb{s}\ak^{1/2} \abs{\ell}^{1/2}}{\jb{\ell}^m\jb{k-\ell,kt-\ell s}^{\sigma_1-2}}\abs{\ell}^{1/2}\jb{\ell,\ell s}^{\sigma_4}\abs{\widehat{\rho_Q}(s,\ell)}\,d\ell ds}^2dkdt.
\end{align*} By Lemma~\ref{lemmaschurcontrolofresonances} and \eqref{B2original}, 
\begin{align*}
    \mathcal{R}^{\mathrm{(B2),2}}_2\lesssim K_5\ep^2\norm{\abs{k}^{{1/2}}\jb{k,kt}^{\sigma_4}\widehat{\rho_Q}(t,k)}_{L^2_{t}([0,T^*],L^2_k)}^2\lesssim K_2'K_5\ep^4,
\end{align*}provided that $\sigma_1\geq d+8$.

\subsubsection{Conclusion of \eqref{B2new}}
There exists a constant $K_{\mathrm{LD}}''>0$, depending only on the fixed parameters in Proposition~\ref{bootstrapprop} and independent of $\hbar$, $T^*$, $\ep$, $\Qin$ and the bootstrap constants $K_i$, such that 
\begin{align*}
    &\norm{\hbar\ak^{3/2}\jb{k,kt}^{\sigma_4}\widehat{\rho_Q}(t,k)}_{L^2_{t}([0,T^*],L^2_k)}^2\\
    &\leq K_{\mathrm{LD}}''\br{\ep^2+K_1K_5\ep^4+K_2'K_5\ep^4+K_5\ep^2\norm{\hbar\abs{k}^{3/2}\jb{k,kt}^{\sigma_4}\widehat{\rho_Q}(t,k)}_{L^2_{t}([0,T^*],L^2_k)}^2}.
\end{align*}
Choosing $\ep$ sufficiently small such that $K_{\mathrm{LD}}'' K_5 \ep^2 \leq 1/2$ permits absorption into the left-hand side,
\begin{align*}
    \norm{\hbar\ak^{3/2}\jb{k,kt}^{\sigma_4}\widehat{\rho_Q}(t,k)}_{L^2_{t}([0,T^*],L^2_k)}^2\leq 2K_{\mathrm{LD}}''\br{\ep^2+K_1K_5\ep^4+K_2'K_5\ep^4}.
\end{align*}
Finally, for $\ep$ chosen such that $K_1 K_5 \ep^2 \leq 1$ and $K_2' K_5 \ep^2 \leq 1$, the improvement \eqref{B2new} follows upon setting $K_2'' := 3 K_{\mathrm{LD}}''$. This completes the proof of the bootstrap improvement \eqref{I2}.

\subsection{Bootstrap improvement \eqref{I4}}\label{sectionb4}
This subsection establishes the improvement of bootstrap assumption \eqref{B4}. The argument combines the linear damping result with a decomposition of the nonlinear term.

As in Section~\ref{sectionb2}, the linear damping result (Proposition~\ref{linearquantumdampingproposition}) is applied to the nonlinear equation, yielding
\begin{align*}
    \norm{\ak^{1/2}\jb{k,kt}^{\sigma_2}\widehat{\rho_Q}(t,k)}_{L^2_{t}([0,T^*])}^2&\lesssim \norm{\ak^{1/2}\jb{k,kt}^{\sigma_2}\widehat{\rho_{{Q_{\mathrm{FS}}}}}(t,k)}_{L^2_{t}([0,T^*])}^2+\mathcal{N}^{\mathrm{(B4)}}(k),
\end{align*}where\begin{align*}
    \mathcal{N}^{\mathrm{(B4)}}(k)&=\int_{t=0}^{T^*}\Big|\ak^{1/2}\jb{k,kt}^{\sigma_2}\int_0^t\int_\ell\widehat w(\ell)\widehat{\rho_Q}(s,\ell) \\
    &\hspace{1.5cm}\times\hbar^{-1}\sin\br{\mytextfrac{\hbar}{2}\ell\cdot k(t-s)}\widehat {W[P]}\br{s,k-\ell,kt-\ell s}\,d\ell ds\Big|^2dt.
\end{align*}The initial data term is controlled by Lemma~\ref{b2b4arenatural}, while the nonlinear contribution is decomposed into transport and reaction terms, with $\mathcal{N}^{\mathrm{(B4)}}\lesssim \mathcal{T}^{\mathrm{(B4)}}(k)+\mathcal{R}^{\mathrm{(B4)}}(k)$, where
\begin{align*}
    \mathcal{T}^{\mathrm{(B4)}}(k)&\lesssim\int_{t=0}^{T^*}\Big[\int_0^t\int_\ell\ak^{1/2}\jb{k-\ell,kt-\ell s}^{\sigma_2}\abs{\widehat w(\ell)}\abs{\widehat{\rho_Q}(s,\ell)} \\
    &\hspace{3.5cm}\times\abs{\ell}\abs{k(t-s)}\abs{\widehat {W[P]}\br{s,k-\ell,kt-\ell s}}\,d\ell ds\Big]^2dt\\
    \mathcal{R}^{\mathrm{(B4)}}(k)&\lesssim\int_{t=0}^{T^*}\Big[\int_0^t\int_\ell\ak^{1/2}\jb{\ell,\ell s}^{\sigma_2}\abs{\widehat w(\ell)}\abs{\widehat{\rho_Q}(s,\ell)} \\
    &\hspace{3.5cm}\times\abs{\ell}\abs{k(t-s)}\abs{\widehat {W[P]}\br{s,k-\ell,kt-\ell s}}\,d\ell ds\Big]^2dt.
\end{align*}

To control the transport term, apply the bounds $\abs{\widehat w(\ell)}\lesssim1$ and $\abs{k(t-s)}\lesssim\ak t$. Then 
\begin{align*}
    \mathcal{T}^{\mathrm{(B4)}}(k)&\lesssim K_5\ep^2\ak^3\int_{t=0}^{T^*}t^2\sbr{\int_{s=0}^t\int_\ell\frac{\abs{\ell}}{\jb{\ell,\ell s}^{\sigma_1}} \jb{k-\ell,kt-\ell s}^{\sigma_2}\abs{\widehat {W[P]}\br{s,k-\ell,kt-\ell s}}\,d\ell ds}^2dt\\
    &\lesssim K_5\ep^2\ak^3\int_{t=0}^{T^*}t^2\Bigg[\int_{s=0}^t\br{\int_\ell\frac{\abs{\ell}^2}{\jb{\ell,\ell s}^{2\sigma_1}}\frac{1}{\abs{k-\ell}^{2\delta}\jb{k-\ell,kt-\ell s}^{2(\sigma_3-\sigma_2)}}\,d\ell}^{1/2} \\
    &\hspace{2.5cm}\times\br{\int_\ell\abs{k-\ell}^{2\delta}\jb{k-\ell,kt-\ell s}^{2\sigma_3}\abs{\widehat {W[P]}\br{s,k-\ell,kt-\ell s}}^2\,d\ell}^{1/2} ds\Bigg]^2dt,
\end{align*}where the second line uses Cauchy–Schwarz in $\ell$. To estimate the first integral in $\ell$, fix a parameter $j>0$ such that $j\leq \min\set{\sigma_1,\sigma_3-\sigma_2}$. Using the inequality $\jb{\ell,\ell s}\jb{k-\ell,kt-\ell s}\gtrsim\jb{k,kt}$, it follows that\begin{align*}
    \frac{1}{\jb{\ell,\ell s}^{2\sigma_1}\jb{k-\ell,kt-\ell s}^{2(\sigma_3-\sigma_2)}}\lesssim\frac{1}{\jb{\ell,\ell s}^{2\sigma_1-2j}\jb{k,kt}^{2j}}.
\end{align*} To estimate the second integral in $\ell$, the Sobolev embedding and Lemma~\ref{L2normofWigner} imply, for $d/2<s\leq \ceils{(d+1)/2}$,
\begin{align*}
    \int_\ell\sup_{\eta}\abs{\reallywidehat{ W[\abs{\bm{\nabla}_x}^\delta\jb{\bm{\nabla}_x,\bm{\nabla}_\xi}^{\sigma_3}P]}(\ell,\eta)}^2\,d\ell&\lesssim\int_\ell\int_\eta\abs{\jb{\nabla_\eta}^{s}\reallywidehat{ W[\abs{\bm{\nabla}_x}^\delta\jb{\bm{\nabla}_x,\bm{\nabla}_\xi}^{\sigma_3}P]}(\ell,\eta)}^2\,d\eta d\ell\\
    &\lesssim\sum_{\aal\leq M}\norm{W\sbr{\bm{\xi}^\alpha\abs{\bm{\nabla}_x}^\delta\jb{\bm{\nabla}_x,\bm{\nabla}_\xi}^{\sigma_3}P}}_{L^2_{z,v}}^2\\
    &\lesssim K_3\ep^2,
\end{align*}where the final step uses bootstrap assumption \eqref{B3}. Substituting into the earlier bound yields
\begin{align*}
    \mathcal{T}^{\mathrm{(B4)}}(k)&\lesssim K_3K_5\ep^4\ak^3\int_{t=0}^{T^*}\frac{t^2}{\jb{k,kt}^{2j}}dt\sbr{\int_{s=0}^t\br{\int_\ell\frac{\abs{\ell}^2}{\jb{\ell,\ell s}^{2\sigma_1-2j}}\frac{1}{\abs{k-\ell}^{2\delta}}\,d\ell}^{1/2} ds}^2\\
    &\lesssim K_3K_5\ep^4\int_{t=0}^{\ak T^*}\frac{t^2}{\jb{k,t}^{2j}}dt\sbr{\int_{s=0}^\infty\frac{1}{\jb{s}^{{d/2}+1-\delta}} \,ds}^2\\
    &\lesssim K_3K_5\ep^4,
\end{align*}provided that $\sigma_1-j>d/2+1-\delta$ and $j>3/2$. These conditions are satisfied under the assumption $\sigma_3-\sigma_2>d/2$ and $\sigma_1>d/2+5/2-\delta$.

To control the reaction term, apply the bound $\abs{\widehat w(\ell)}\lesssim1$ and perform Cauchy--Schwarz in both $\ell$ and $s$:
\begin{align*}
    \mathcal{R}^{\mathrm{(B4)}}(k)&\lesssim \int_{t=0}^{T^*}\abs{k}^3t^2\br{\int_{s=0}^t\int_\ell\abs{\ell}\jb{\ell,\ell s}^{2\sigma_4}\abs{\widehat{\rho_Q}(s,\ell)}^2\,d\ell ds}\\ &\hspace{2cm}\times\br{\int_{s=0}^t\int_\ell\abs{\ell}\jb{\ell,\ell s}^{-2(\sigma_4-\sigma_2)}\abs{\widehat {W[P]}\br{s,k-\ell,kt-\ell s}}^2\,d\ell ds}dt\\
    &\lesssim K_2K_5 \ep^4\int_{t=0}^{T^*}\abs{k}^3t^2\int_{s=0}^t\int_\ell \frac{\abs{\ell}}{\jb{\ell,\ell s}^{2(\sigma_4-\sigma_2)}\jb{k-\ell,kt-\ell s}^{2\sigma_1}}\,d\ell dsdt.
\end{align*}Fix a parameter $j>0$ such that $j\leq \min\set{\sigma_1,\sigma_4-\sigma_2}$. Using the inequality $\jb{\ell,\ell s}\jb{k-\ell,kt-\ell s}\gtrsim\jb{k,k t}$, it follows that \begin{align*}
    \frac{1}{\jb{\ell,\ell s}^{2(\sigma_4-\sigma_2)}\jb{k-\ell,kt-\ell s}^{2\sigma_1}}\lesssim\frac{1}{\jb{\ell,\ell s}^{2(\sigma_4-\sigma_2-j)}\jb{k,kt}^{2j}},
\end{align*}
and hence,\begin{align*}
    \mathcal{R}^{\mathrm{(B4)}}(k)&\lesssim  K_2K_5 \ep^4\int_{t=0}^{T^*}\frac{\abs{k}^3t^2}{\jb{k,kt}^{2j}}dt\int_{s=0}^t\int_\ell \frac{\abs{\ell}}{\jb{\ell,\ell s}^{2(\sigma_4-\sigma_2-j)}}\,d\ell ds \\
    &\lesssim K_2K_5 \ep^4,
\end{align*}where the $\ell$-$s$ integral converges for $j<\sigma_4-\sigma_2-(d+1)/2$, and the $t$ integral converges for $j>3/2$. For such a $j$ to exist, the condition $\sigma_4-\sigma_2>d/2+2$ is required.

\subsubsection*{Conclusion of bootstrap improvement \eqref{I4}}
Under the bootstrap assumptions \eqref{B2}, \eqref{B3}, and \eqref{B5}, there exists a constant $\widetilde K>0$, depending only on the fixed parameters in Proposition~\ref{bootstrapprop} and independent of $\hbar$, $T^*$, $\ep$, $\Qin$ and the bootstrap constants $K_i$, such that\begin{align*}
    \norm{\ak^{1/2}\jb{k,kt}^{\sigma_2}\widehat{\rho_Q}(t,k)}_{L^\infty_kL^2_{t}([0,T^*])}^2\leq \widetilde K\br{\ep^2+(K_3K_5+K_2K_5)\ep^4}.
\end{align*} Define $K_4:=\widetilde K$. Then for $\ep^2\leq(K_3K_5+K_2K_5)^{-1}$, the improvement \eqref{I4} in Proposition~\ref{bootstrapprop} follows.

\section{Energy estimates on the density matrix}\label{sectiondensitymatrix}
This section establishes the bootstrap improvements \eqref{I1}, \eqref{I3} and \eqref{I5}. The estimates are derived from energy bounds on \( P \) in weighted quantum Sobolev spaces, formulated via the Wigner transform. Throughout this section the Wigner transform of $P$ is written $W[P](z,v)$, in the free-flow phase-space variables $(z,v)$ introduced in Section~\ref{conjugationSection}.

The following lemma is used repeatedly in the analysis.
\begin{lemma}[{\cite[Lemma~2.9]{BedrossianMasmoudiMouhot2018}}]\label{Lemma2.9}
    \begin{enumerate}[label=(\alph*)]
        \item \label{Lemma2.9a}Let $f_{1},f_{2}\in L^2(\R^d_k\times\R_\eta^d)$ and $r\in L^1(\R^d)$. Then \begin{align*}
            \abs{\int_\eta\int_k\int_\ell  f_1(k,\eta)r(\ell )f_2(k-\ell ,\eta-\ell t)\,d\ell dk d\eta}\lesssim\norm{f_1}_{L^2_{k,\eta}}\norm{f_2}_{L^2_{k,\eta}}\norm{r}_{L^1_k}.
        \end{align*} 
\item \label{Lemma2.9b}Let $f_1\in L^2(\R^d_k\times\R^d_\eta)$ and $f_2\in L^1(\R^d_k;L^2(\R^d_\eta))$ and $r\in L^2(\R^d)$. Then
\begin{align*}
            \abs{\int_\eta\int_k\int_\ell  f_1(k,\eta)r(\ell )f_2(k-\ell ,\eta-\ell t)\,d\ell dkd\eta}\lesssim\norm{f_1}_{L^2_{k,\eta}}\norm{f_2}_{L^1_k(L^2_\eta)}\norm{r}_{L^2_k}.
        \end{align*} 
    \end{enumerate}
\end{lemma}
\begin{proof}
    The proof of \ref{Lemma2.9a} relies on the Cauchy--Schwarz inequality, while the proof of \ref{Lemma2.9b} incorporates Young's convolution inequality.
\end{proof}

\subsection{Bootstrap improvement \eqref{I1}}\label{sectionb1}
This subsection establishes the improvement of the bootstrap assumption \eqref{B1}. The analysis is carried out on the Wigner transform, using the equivalence
\begin{align*}
   \sum_{\aal\leq M}\norm{\jb{\bm{\nabla}_x,\bm{\nabla}_\xi}^{\sigma_4}\jb{t\bm{\nabla}_x,\bm{\nabla}_\xi}\br{\bm{\xi}^\alpha P(t)}}_{\mathcal L^2}^2=\sum_{\aal\leq M}\norm{\jb{t\nabla_z,\nabla_v}\br{v^\alpha W[P]}}_{H^{\sigma_4}}^2
\end{align*}as given by Lemma~\ref{L2normofWigner}.

\subsubsection{Estimate on the velocity derivative}\label{tricksection}
Let $\alpha$ be a multi-index such that \(\abs{\alpha} \leq M\). An energy estimate yields\begin{align*}
    (2\pi)^{2d}\half\frac{d}{dt}\norm{\jb{\nabla_v}(v^\alpha W[P])}_{H^{\sigma_4}}^2&=\mathcal{L}^{\mathrm{(B1)},V} + \mathcal{N}^{\mathrm{(B1)},V},
\end{align*}where
\begin{align*}
    \mathcal{L}^{\mathrm{(B1)},V}&=\Re\int_{k,\eta}\jb{\eta}\jb{k,\eta}^{\sigma_4} \overline{D_\eta^\alpha \widehat{W[P]}(t,k,\eta)}\jb{\eta}\jb{k,\eta}^{\sigma_4} D_\eta^\alpha\widehat{W[{\mathrm{L}_P}]}(t,k,\eta)\,d\eta dk,\\
    \mathcal{N}^{\mathrm{(B1)},V} &=\Re \int_{k,\eta}\jb{\eta}\jb{k,\eta}^{\sigma_4} \overline{D_\eta^\alpha \widehat{W[P]}(t,k,\eta)}\jb{\eta}\jb{k,\eta}^{\sigma_4} D_\eta^\alpha\widehat{W[{\mathrm{N}_P}]}(t,k,\eta)\,d\eta dk.
\end{align*}The terms $\widehat{W[{\mathrm{L}_P}]}(t,k,\eta)$ and $\widehat{W[{\mathrm{N}_P}]}(t,k,\eta)$ are defined in \eqref{WL}.

To estimate \(\mathcal{L}^{\mathrm{(B1)},V}\), Lemma~\ref{etalemma} is applied to isolate the highest-order contribution, in which all \( D_\eta^\alpha \)-derivatives fall on \(\widehat{g}(\eta - kt)\), from the remainder terms. This gives the decomposition $\abs{\mathcal{L}^{\mathrm{(B1)},V}}\lesssim \mathcal{L}_0^{\mathrm{(B1)},V} + \mathcal{L}_M^{\mathrm{(B1)},V}$, where
\begin{align*}
    \mathcal{L}_0^{\mathrm{(B1)},V}&=\int_{k,\eta}\jb{\eta}\jb{k,\eta}^{\sigma_4} \pigl|D_\eta^\alpha \widehat{W[P]}(t,k,\eta)\pigr|\jb{\eta}\jb{k,\eta}^{\sigma_4} \\
    &\hspace{4cm}\times\ak\abs{\widehat w(k)}\abs{\widehat{\rho_Q}(t,k)}\abs{\eta-kt}\abs{D^\alpha_\eta\widehat g(\eta-kt)}\,d\eta dk,\\
    \mathcal{L}_M^{\mathrm{(B1)},V}&=\sum_{\abs{\beta}\leq \aal-1}\int_{k,\eta}\jb{\eta}\jb{k,\eta}^{\sigma_4} \pigl|D_\eta^\alpha \widehat{W[P]}(t,k,\eta)\pigr|\jb{\eta}\jb{k,\eta}^{\sigma_4}\\
    &\hspace{4cm}\times\jb{\hbar k}^{\abs{\alpha}-1}\ak\abs{\widehat w(k)}\abs{\widehat{\rho_Q}(t,k)}\pigl|{D^\beta_\eta\widehat g(\eta-kt)}\pigr|\,d\eta dk.
\end{align*}

For \(\mathcal{L}_0^{\mathrm{(B1)},V}\), observe that \(\jb{\eta} \lesssim \jb{kt} \jb{\eta - kt}\) and \(\jb{k,\eta}^{\sigma_4} \lesssim \jb{k,kt}^{\sigma_4} \jb{\eta - kt}^{\sigma_4}\). The decay assumption \eqref{wassumption} on \(w\) implies $\jb{kt} \ak \abs{\widehat{w}(k)} \lesssim \jb{t} \ak^{1/2}.$ Substituting these bounds gives
\begin{align*}
\abs{\mathcal{L}_0^{\mathrm{(B1)},V}}&\lesssim\jb{t}\int_{k,\eta}\jb{\eta}\jb{k,\eta}^{\sigma_4} \abs{D_\eta^\alpha \widehat{W[P]}(t,k,\eta)}\ak^{1/2}\jb{k,kt}^{\sigma_4}\abs{\widehat{\rho_Q}(t,k)}\jb{\eta-kt}^{\sigma_4+2}\abs{D^\alpha_\eta\widehat g(\eta-kt)}\,d\eta dk\\
    &\lesssim\jb{t}\norm{\jb{\nabla_v}W[P]}_{H^{\sigma_4}_M}\norm{\ak^{1/2}\jb{k,kt}^{\sigma_4}\widehat{\rho_Q}(t,k)}_{L^2_{k}}\norm{g}_{H^{\sigma_4+2}_M}.
\end{align*}

To estimate \(\mathcal{L}_M^{\mathrm{(B1)},V}\), the full decay assumption \eqref{wassumption} is used to ensure \(\jb{k}^{M - 1/2} \abs{\widehat{w}(k)} \lesssim 1\), which allows the bound
\[
\jb{kt} \jb{\hbar k}^{\abs{\alpha} - 1} \abs{k}\abs{\widehat w(k)} \lesssim \jb{t} \jb{\hbar k} \ak^{1/2} \jb{k}^{M - 1/2}\abs{\widehat w(k)} \lesssim\jb{t} \jb{\hbar k} \ak^{1/2},
\]
valid uniformly for \(\hbar \in (0,1]\). Applying Cauchy--Schwarz then gives
\begin{align*}
    \mathcal{L}_M^{\mathrm{(B1)},V}\lesssim\jb{t}\norm{\jb{\nabla_v}W[P]}_{H^{\sigma_4}_M}\norm{\jb{\hbar k}\ak^{1/2}\jb{k,kt}^{\sigma_4}\widehat{\rho_Q}(t,k)}_{L^2_{k}}\norm{g}_{H^{\sigma_4+1}_M}.
\end{align*}

To control the nonlinear contribution \(\mathcal{N}^{\mathrm{(B1)},V}\), Lemma~\ref{etalemma} is used to isolate the highest-order term in which all \(D_\eta^\alpha\)-derivatives fall on \(\widehat{W[P]}\), yielding the decomposition $\mathcal{N}^{\mathrm{(B1)},V}= \mathcal{N}_0^{\mathrm{(B1)},V}+\mathcal{N}_M^{\mathrm{(B1)},V}.$ The leading term will be treated explicitly below. The remainder is controlled using Lemma~\ref{etalemma}. Explicitly,
\begin{align*}
\mathcal{N}^{\mathrm{(B1)},V}_0&=\Im\int_{k,\eta}\jb{\eta}\jb{k,\eta}^{\sigma_4}\overline{D_\eta^\alpha \widehat{W[P]}(t,k,\eta)}\jb{\eta}\jb{k,\eta}^{\sigma_4}\\
&\times\sbr{\frac{\hbar^{-1}}{(2\pi)^d}\int_\ell\widehat{w}(\ell)\widehat{\rho_Q}(t,\ell )\br{e^{-i\frac{\hbar}{2}\ell \cdot(\eta-kt)}-e^{i\frac{\hbar}{2}\ell \cdot(\eta-kt)}}D_\eta^\alpha\widehat{W[P]}(t,k-\ell ,\eta-\ell t)\,d\ell }dkd\eta\end{align*} and
\begin{align*}
    \abs{\mathcal{N}_M^{\mathrm{(B1)},V}}&\lesssim\sum_{\abs{\beta}\leq \aal-1}\int_{k,\eta,\ell }\jb{\eta}\jb{k,\eta}^{\sigma_4}\pigl|{D_\eta^\alpha \widehat{W[P]}(t,k,\eta)}\pigr|\jb{\eta}\jb{k,\eta}^{\sigma_4}\\
    &\hspace{2.5cm}\times\abs{\widehat{w}(\ell )}\abs{\ell }\jb{\hbar\ell }^{\aal-1}\abs{\widehat{\rho_Q}(t,\ell )}\pigl|{D_\eta^\beta\widehat{W[P]}(t,k-\ell ,\eta-\ell t)}\pigr|\,d\ell dkd\eta.
\end{align*}

To estimate \(\mathcal{N}_0^{\mathrm{(B1)},V}\), a cancellation is used by subtracting an integral that evaluates to zero:
\begin{equation}\label{N0B1}
\begin{split}
&\mathcal{N}_0^{\mathrm{(B1)},V}\\
&=\Im\int_{k,\eta,\ell}\jb{\eta}\jb{k,\eta}^{\sigma_4}\overline{D_\eta^\alpha \widehat{W[P]}(t,k,\eta)}\br{\jb{\eta}\jb{k,\eta}^{\sigma_4}-\jb{\eta-\ell t}\jb{k-\ell ,\eta-\ell t}^{\sigma_4}}\\
&\hspace{0.23cm}\times\sbr{\frac{\hbar^{-1}}{(2\pi)^d} \widehat{w}(\ell )\widehat{\rho_Q}(t,\ell )\br{e^{-i\frac{\hbar}{2}\ell \cdot(\eta-kt)}-e^{i\frac{\hbar}{2}\ell \cdot(\eta-kt)}}D_\eta^\alpha\widehat{W[P]}(t,k-\ell ,\eta-\ell t)}d\ell dkd\eta.
\end{split}
\end{equation}
The integral that has been subtracted is now shown to vanish. Observe that the Wigner transform of \(P\) satisfies the equation\begin{align}\label{Wphysical}
\begin{mycases}
\p_tW[P](t,z,v)=W[{\mathrm{L}_P}](t,z,v)+W[{\mathrm{N}_P}](t,z,v)\\
    W[{\mathrm{L}_P}](t,z,v)=-\frac{i\hbar^{-1}}{(2\pi)^d}\int_ke^{ik\cdot(z+vt)}\widehat w(k)\widehat{\rho_Q}(t,k)\br{g(v-\mytextfrac{\hbar}{2}k)-g(v+\mytextfrac{\hbar}{2}k)}\,dk,\\
    W[{\mathrm{N}_P}](t,z,v)=-\frac{i\hbar^{-1}}{(2\pi)^d}\int_\ell  e^{i\ell \cdot(z+vt)}\widehat w(\ell )\widehat{\rho_Q}(t,\ell )\\
    \hspace{4.25cm}\times\br{W[P](t,z+\mytextfrac{\hbar}{2}\ell t,v-\mytextfrac{\hbar}{2}\ell )-W[P](t,z-\mytextfrac{\hbar}{2}\ell t,v+\mytextfrac{\hbar}{2}\ell )}\,d\ell.
\end{mycases}
\end{align}
Using Plancherel's theorem, the subtracted integral corresponds to an expression on the physical side of the form:
\begin{align*}
&\int_{k,\eta,\ell}\jb{\eta}\jb{k,\eta}^{\sigma_4}\overline{D_\eta^\alpha \widehat{W[P]}(t,k,\eta)}\jb{\eta-\ell t}\jb{k-\ell ,\eta-\ell t}^{\sigma_4}\\
&\hspace{0.35cm}\times\Bigg[\frac{\hbar^{-1}}{(2\pi)^d} \widehat{w}(\ell )\widehat{\rho_Q}(t,\ell )\br{e^{-i\frac{\hbar}{2}\ell \cdot(\eta-kt)}-e^{i\frac{\hbar}{2}\ell \cdot(\eta-kt)}}D_\eta^\alpha\widehat{W[P]}(t,k-\ell ,\eta-\ell t) \Bigg]d\ell dkd\eta\\
&=(2\pi)^d\int_{z,v}\overline{\jb{\nabla_v}\jb{\nabla}^{\sigma_4}\br{(-iv)^\alpha W[P]}(t,z,v)}\hbar^{-1}\int_\ell e^{i\ell \cdot(z+vt)}\widehat w(\ell )\widehat{\rho_Q}(t,\ell )\\
&\hspace{0.35cm}\times\Big[\br{\jb{\nabla_v}\jb{\nabla}^{\sigma_4}((-iv)^\alpha W[P])}(t,z+\mytextfrac{\hbar}{2}\ell t,v-\mytextfrac{\hbar}{2}\ell )\\
&\hspace{4.5cm}-\br{\jb{\nabla_v}\jb{\nabla}^{\sigma_4}((-iv)^\alpha W[P])}(t,z-\mytextfrac{\hbar}{2}\ell t,v+\mytextfrac{\hbar}{2}\ell )\Big]\,d\ell dzdv,
\end{align*}where $\jb{\nabla}=\jb{\nabla_z,\nabla_v}$.
To see that this integral has zero imaginary part, define \[f(z,v)=\jb{\nabla_v}\jb{\nabla_z,\nabla_v}^{\sigma_4}\br{(-iv)^\alpha W[P]}(t,z,v),\] so that the integral above equals $-(2\pi)^dI$, where \begin{align*}
    I:=\hbar^{-1}\int_{z,v,\ell }\overline{f(z,v)}e^{i\ell \cdot(z+tv)}\widehat{\rho_Q}(t,\ell )\widehat w(\ell )\sbr{f(z-\mytextfrac{\hbar}{2}\ell t,v+\mytextfrac{\hbar}{2}\ell )-f(z+\mytextfrac{\hbar}{2}\ell t,v-\mytextfrac{\hbar}{2}\ell )}\,d\ell dzdv.
\end{align*} Apply the change of variables \(z \mapsto z + \tfrac{\hbar}{2} \ell t\), \(v \mapsto v - \tfrac{\hbar}{2} \ell\) in the first term and \(z \mapsto z - \tfrac{\hbar}{2} \ell t\), \(v \mapsto v + \tfrac{\hbar}{2} \ell\) in the second. Then\begin{align*}
    I&=\hbar^{-1}\int_{z,v,\ell }\overline{f(z,v)}e^{i\ell \cdot(z+tv)}\widehat{\rho_Q}(t,\ell )\widehat w(\ell )\sbr{f(z-\mytextfrac{\hbar}{2}\ell t,v+\mytextfrac{\hbar}{2}\ell )-f(z+\mytextfrac{\hbar}{2}\ell t,v-\mytextfrac{\hbar}{2}\ell )}\,d\ell dzdv\\
&=\hbar^{-1}\int_{z,v,\ell }\overline{\br{f(z+\mytextfrac{\hbar}{2}t\ell ,v-\frac{\hbar}{2}\ell )-f(z-\frac{\hbar}{2}\ell t,v+\mytextfrac{\hbar}{2}\ell )}}e^{i\ell \cdot(z+tv)}\widehat{\rho_Q}(t,\ell )\widehat w(\ell )f(z,v)\,d\ell dvdz\\
    &=\hbar^{-1}\int_{z,v,\ell }\overline{\br{f(z-\mytextfrac{\hbar}{2}t\ell ,v+\frac{\hbar}{2}\ell )-f(z+\frac{\hbar}{2}\ell t,v-\mytextfrac{\hbar}{2}\ell )}}e^{-i\ell \cdot(z+tv)}\widehat{\rho_Q}(t,-\ell )\widehat w(-\ell )f(z,v)\,d\ell dvdz\\
    &=\hbar^{-1}\overline{\int_{z,v,\ell }\overline{f(z,v)}e^{i\ell \cdot(z+tv)}\widehat{\rho_Q}(t,\ell )\widehat w(\ell )\sbr{f(z-\mytextfrac{\hbar}{2}\ell t,v+\mytextfrac{\hbar}{2}\ell )-f(z+\mytextfrac{\hbar}{2}\ell t,v-\mytextfrac{\hbar}{2}\ell )}\,d\ell dzdv}\\
    &=\overline{I}.
\end{align*}The penultimate equality uses that $w$ and $\rho_Q(t)$ are real-valued. Indeed, self-adjointness of $Q(t)$ implies that $\rho_Q(t,x)$ is real-valued, and hence $\widehat w(-\ell)=\overline{\widehat w(\ell)}$ and $\widehat{\rho_Q}(t,-\ell)=\overline{\widehat{\rho_Q}(t,\ell)}$.
Therefore, $I=\overline I$, so $\Im I=0$.

Therefore, $\mathcal{N}_0^{\mathrm{(B1)},V}$ can be controlled using \eqref{N0B1}:
\begin{align*}
  \abs{\mathcal{N}_0^{\mathrm{(B1)},V}}&\lesssim  \int_{k,\eta,\ell }\jb{\eta}\jb{k,\eta}^{\sigma_4}\abs{D_\eta^\alpha \widehat{W[P]}(t,k,\eta)}\abs{\jb{\eta}\jb{k,\eta}^{\sigma_4}-\jb{\eta-\ell t}\jb{k-\ell ,\eta-\ell t}^{\sigma_4}}\\
&\hspace{1.5cm}\times\abs{\widehat{w}(\ell )}\abs{\ell }\abs{\widehat{\rho_Q}(t,\ell )}\abs{\eta-kt}\abs{D_\eta^\alpha\widehat{W[P]}(t,k-\ell ,\eta-\ell t)}\,d\ell dkd\eta.
\end{align*}
To estimate the difference of weights, the Mean Value Theorem gives\begin{align*}
    &\abs{\jb{\eta}\jb{k,\eta}^{\sigma_4}-\jb{\eta-\ell t}\jb{k-\ell ,\eta-\ell t}^{\sigma_4}}\\
    &\leq\vphantom{\int_A}\abs{\jb{\eta}-\jb{\eta-\ell t}}\jb{k,\eta}^{\sigma_4}+\jb{\eta-\ell t}\abs{\jb{k,\eta}^{\sigma_4}-\jb{k-\ell ,\eta-\ell t}^{\sigma_4}}\\
    &\lesssim\jb{\ell t}\jb{k,\eta}^{\sigma_4}+\jb{\eta-\ell t}\abs{(\ell ,\ell t)}\sup_{\abs{\ell '}\leq \abs{\ell }}\jb{k-\ell ',\eta-\ell 't}^{\sigma_4-1}\\
    &\lesssim\jb{\ell t}\br{\jb{\ell ,\ell t}^{\sigma_4}+\jb{k-\ell ,\eta-\ell t}^{\sigma_4}}+\jb{\eta-\ell t}\abs{(\ell ,\ell t)}\sup_{\abs{\ell '}\leq \abs{\ell }}\br{\jb{k-\ell ,\eta-\ell t}^{\sigma_4-1}+\jb{\ell -\ell ',(\ell -\ell ')t}^{\sigma_4-1}}\\
    &\lesssim\jb{\ell t}\jb{\ell ,\ell t}^{\sigma_4}+\jb{\eta-\ell t}\abs{(\ell ,\ell t)}\jb{\ell ,\ell t}^{\sigma_4-1}+\jb{\ell t}\jb{k-\ell ,\eta-\ell t}^{\sigma_4}+\jb{\eta-\ell t}\abs{(\ell ,\ell t)}\jb{k-\ell ,\eta-\ell t}^{\sigma_4-1}.
\end{align*}
This leads to the decomposition, with $\mathcal R$ denoting the reaction contribution and $\mathcal T$ the transport contribution: \[\abs{\mathcal{N}_0^{\mathrm{(B1)},V}}\lesssim \mathcal{R}_{0,1}^{\mathrm{(B1)},V}+\mathcal{R}_{0,2}^{\mathrm{(B1)},V}+\mathcal{T}_{0,1}^{\mathrm{(B1)},V}+\mathcal{T}_{0,2}^{\mathrm{(B1)},V},\] corresponding to the four terms above.

For \(\mathcal{R}_{0,1}^{\mathrm{(B1)},V}\), note that \(\abs{\eta - kt} \lesssim \jb{t} \jb{k - \ell, \eta - \ell t}\) and the bound \(\jb{\ell t} \abs{\ell} \abs{\widehat{w}(\ell)} \lesssim \jb{t} \abs{\ell}^{1/2}\). These yield the bound
\begin{align*}
    \mathcal{R}_{0,1}^{\mathrm{(B1)},V}
    &\lesssim\jb{t}^2\int_{k,\eta,\ell }\jb{\eta}\jb{k,\eta}^{\sigma_4}\abs{D_\eta^\alpha \widehat{W[P]}(t,k,\eta)}\jb{\ell ,\ell t}^{\sigma_4}\abs{\ell }^{1/2}\abs{\widehat{\rho_Q}(t,\ell )}\\
    &\hspace{4.5cm}\times\jb{k-\ell ,\eta-\ell t}\abs{D_\eta^\alpha\widehat{W[P]}(t,k-\ell ,\eta-\ell t)}\,d\ell dkd\eta
\end{align*}
Lemma~\ref{Lemma2.9}\ref{Lemma2.9b} then implies
\begin{align*}
    \mathcal{R}_{0,1}^{\mathrm{(B1)},V}\lesssim\jb{t}^2\norm{\jb{\nabla_v}W[P]}_{H^{\sigma_4}_M}\norm{\ak^{1/2}\jb{k,kt}^{\sigma_4}\widehat{\rho_Q}(t,k)}_{L^2_k}\int_k\norm{\jb{k,\eta}D^\alpha_\eta\widehat{W[P]}(t,k,\eta)}_{L^2_\eta}\,dk.
\end{align*}
The remaining integral is bounded using Cauchy--Schwarz in $k$:
\begin{align*}
    \int_k\norm{\jb{k,\eta}D^\alpha_\eta\widehat{W[P]}(t,k,\eta)}_{L^2_\eta}\,dk\lesssim\norm{\ak^\delta\jb{k}^{d/2}\jb{k,\eta}D^\alpha_\eta\widehat{W[P]}(t,k,\eta)}_{L^2_{k,\eta}}\lesssim \sqrt{K_3}\ep.
\end{align*}
The second reaction term \(\mathcal{R}_{0,2}^{\mathrm{(B1)},V}\) is estimated similarly. It follows that
\begin{align*}
    \mathcal{R}_{0,1}^{\mathrm{(B1)},V}+\mathcal{R}_{0,2}^{\mathrm{(B1)},V}\lesssim\sqrt{K_3}\ep\jb{t}^2\norm{\jb{\nabla_v}W[P]}_{H^{\sigma_4}_M}\norm{\ak^{1/2}\jb{k,kt}^{\sigma_4}\widehat{\rho_Q}(t,k)}_{L^2_k}.
\end{align*}

The transport term $\mathcal{T}_{0,1}^{\mathrm{(B1)},V}$ is estimated using the inequality $\abs{\eta - kt} \lesssim \jb{t(k - \ell), \eta - \ell t}$, together with Lemma~\ref{Lemma2.9}\ref{Lemma2.9a}:
\begin{align*}
    \mathcal{T}_{0,1}^{\mathrm{(B1)},V}&\lesssim\int_{k,\eta,\ell }\jb{\eta}\jb{k,\eta}^{\sigma_4}\abs{D_\eta^\alpha \widehat{W[P]}(t,k,\eta)}\jb{\ell t}\jb{k-\ell ,\eta-\ell t}^{\sigma_4}\\
    &\hspace{3cm}\times\abs{\widehat{w}(\ell )}\abs{\ell }\abs{\widehat{\rho_Q}(t,\ell )}\abs{\eta-kt}\abs{D_\eta^\alpha\widehat{W[P]}(t,k-\ell ,\eta-\ell t)}\,d\ell dkd\eta\\
    &\lesssim\norm{\jb{\nabla_v}W[P]}_{H^{\sigma_4}_M}\norm{\jb{t\nabla_z,\nabla_v}W[P]}_{H^{\sigma_4}_M}\int_\ell\jb{\ell t}\abs{\ell }\abs{\widehat{\rho_Q}(t,\ell )}\abs{\widehat w(\ell )}\,d\ell \\
    &\lesssim \frac{\sqrt{K_5}\ep}{\jb{t}^{d+1}}\norm{\jb{\nabla_v}W[P]}_{H^{\sigma_4}_M}\norm{\jb{t\nabla_z,\nabla_v}W[P]}_{H^{\sigma_4}_M},
\end{align*}where the last step uses the decay assumption \eqref{wassumption} and the bootstrap assumption \eqref{B5}, provided that $\sigma_1 > d + 2$. The second transport term $\mathcal{T}_{0,2}^{\mathrm{(B1)},V}$ is treated analogously.

For $\mathcal{N}_M^{\mathrm{(B1)},V}$, the weight $\jb{\eta}\jb{k,\eta}^{\sigma_4}$ is decomposed into transport and reaction components using \begin{align*}
    \jb{\eta}\jb{k,\eta}^{\sigma_4}\lesssim\br{\vphantom{\jb{\cdot}_1^{\sigma}}\jb{\ell t}+\jb{\eta-\ell t}}\jb{\ell ,\ell t}^{\sigma_4}+\br{\vphantom{\jb{\cdot}_1^{\sigma}}\jb{\ell t}+\jb{\eta-\ell t}}\jb{k-\ell ,\eta-\ell t}^{\sigma_4},
\end{align*}
which leads to the bound \[\abs{\mathcal{N}_M^{\mathrm{(B1)},V}}\lesssim \mathcal{R}_{M,1}^{\mathrm{(B1)},V}+\mathcal{R}_{M,2}^{\mathrm{(B1)},V}+\mathcal{T}_{M,1}^{\mathrm{(B1)},V}+\mathcal{T}_{M,2}^{\mathrm{(B1)},V}.\] To estimate the reaction terms, apply assumption \eqref{wassumption} and proceed as in the estimate for $\mathcal{R}_{0,1}^{\mathrm{(B1)},V}$:
\begin{align*}
    \mathcal{R}_{M,1}^{\mathrm{(B1)},V}+ \mathcal{R}_{M,2}^{\mathrm{(B1)},V}&\lesssim\sqrt{K_3}\ep\jb{t}^2\norm{\jb{\nabla_v}W[P]}_{H_M^{\sigma_4}}\norm{\jb{\hbar k}\ak^{1/2}\jb{k,kt}^{\sigma_4}\widehat{\rho_Q}(t,k)}_{L^2_k}.
\end{align*}The transport terms $\mathcal{T}_{M,1}^{\mathrm{(B1)},V}$ and $\mathcal{T}_{M,2}^{\mathrm{(B1)},V}$ are treated analogously to $\mathcal{T}_{0,1}^{\mathrm{(B1)},V}$, yielding
\begin{align*}
    \mathcal{T}_{M,1}^{\mathrm{(B1)},V}+\mathcal{T}_{M,2}^{\mathrm{(B1)},V}&\lesssim\frac{\sqrt{K_5}\ep}{\jb{t}^{d+1}}\norm{\jb{\nabla_v}W[P](t)}_{H^{\sigma_4}_M}^2.
\end{align*}

\subsubsection{Conclusion of the velocity derivative}\label{concluvelocity}
The preceding estimates imply
\begin{align*}
    \hspace{-0.2cm}\half\frac{d}{dt}\norm{\jb{\nabla_v}{W[P]}(t)}_{H^{\sigma_4}_M}^2&\lesssim\jb{t}\norm{\jb{\nabla_v}{W[P](t)}}_{H^{\sigma_4}_M}\norm{\jb{\hbar k}\ak^{1/2}\jb{k,kt}^{\sigma_4}\widehat{\rho_Q}(t,k)}_{L^2_{k}}\\
    &\hspace{0.5cm}+\sqrt{K_3}\ep\jb{t}^2\norm{\jb{\nabla_v}W[P](t)}_{H^{\sigma_4}_M}\norm{\jb{\hbar k}\ak^{1/2}\jb{k,kt}^{\sigma_4}\widehat{\rho_Q}(t,k)}_{L^2_k}\\
    &\hspace{0.5cm}+\frac{\sqrt{K_5}\ep}{\jb{t}^{d+1}}\norm{\jb{\nabla_v}W[P](t)}_{H^{\sigma_4}_M}\norm{\jb{t\nabla_z,\nabla_v}W[P](t)}_{H^{\sigma_4}_M}.
    \end{align*}
Dividing through and estimating the derivative of the norm gives\begin{align*}
    \frac{d}{dt}&\norm{\jb{\nabla_v}W[P](t)}_{H^{\sigma_4}_M}\\
    &\lesssim\br{\jb{t}+\sqrt{K_3}\ep\jb{t}^2}\norm{\jb{\hbar k}\ak^{1/2}\jb{k,kt}^{\sigma_4}\widehat{\rho_Q}(t,k)}_{L^2_{k}}+\frac{\sqrt{K_5}\ep}{\jb{t}^{d+1}}\norm{\jb{t\nabla_z,\nabla_v}W[P](t)}_{H^{\sigma_4}_M}.
\end{align*}
Integration in time, using the bootstrap assumptions \eqref{B1} and \eqref{B2}, yields
\begin{align*}
    &\norm{\jb{\nabla_v}W[P](t)}_{H^{\sigma_4}_M}-\norm{\jb{\nabla_v}W[\Qin]}_{H^{\sigma_4}_M}\\
    &\lesssim\int_0^t\br{\jb{s}+\sqrt{K_3}\ep\jb{s}^2}\norm{\jb{\hbar k}\ak^{1/2}\jb{k,ks}^{\sigma_4}\widehat{\rho_Q}(s,k)}_{L^2_{k}}\,ds+\int_0^t\frac{\sqrt{K_5}\ep}{\jb{s}^{d+1}}\norm{\jb{s\nabla_z,\nabla_v}W[P](s)}_{H^{\sigma_4}_M}\,ds\\
    &\lesssim\br{\int_0^t\br{\jb{s}+\sqrt{K_3}\ep\jb{s}^2}^2\,ds}^{1/2}\norm{\jb{\hbar k}\ak^{1/2}\jb{k,ks}^{\sigma_4}\widehat{\rho_Q}(s,k)}_{L^2_{t}([0,T^*],L^2_k)}+\sqrt{K_1K_5}\ep^2\int_0^t\frac{\jb{s}^{d-1/2}}{\jb{s}^{d+1}}\,ds\\
    &\lesssim\sqrt{K_2}\ep\br{1+\sqrt{K_3}\ep\jb{t}}\jb{t}^{3/2}+\sqrt{K_1K_5}\ep^2.
\end{align*}
Hence, there exists a constant $\widetilde K_V>0$, depending only on the fixed parameters in Proposition~\ref{bootstrapprop} and independent of $\hbar$, $T^*$, $\ep$, $\Qin$ and the bootstrap constants $K_i$ such that
\begin{align}
      \hbar^{-d}\sum_{\aal\leq M}\norm{\bm{\xi}^\alpha\jb{\bm{\nabla}_x,\bm{\nabla}_\xi}^{\sigma_4}\jb{\bm{\nabla}_\xi}P(t)}_{L^2_{x,y}}^2\leq \widetilde K_V\ep^2\sbr{\sqrt{K_2}+\ep\br{\sqrt{K_2K_3}+\sqrt{K_1K_5}}}^2\jb{t}^{5}.\label{B1vconc}
\end{align}

\subsubsection{Estimate on the spatial derivative}
Let $\aal\leq M$ be a multi-index. As in the velocity case, an energy estimate yields\begin{align*}
    (2\pi)^{2d}\frac{1}{2}\frac{d}{dt}\norm{\jb{\nabla_z}v^\alpha{W[P]}}_{H^{\sigma_4}}^2=\mathcal{L}^{\mathrm{(B1)},Z} + \mathcal{N}^{\mathrm{(B1)},Z},
\end{align*}where the linear and nonlinear terms are defined analogously. The linear contribution is controlled as in Section~\ref{tricksection}, giving
\begin{align*}
    \mathcal{L}^{\mathrm{(B1)},Z}\lesssim\norm{\jb{\nabla_z}{W[P]}}_{H^{\sigma_4}_M}\norm{\jb{\hbar k}\ak^{1/2}\jb{k,kt}^{\sigma_4}\widehat{\rho_Q}(t,k)}_{L^2_{k}}\norm{g}_{H^{\sigma_4+1}_M}.
\end{align*}
For the nonlinear term, the decomposition $\mathcal N^{\mathrm{(B1)},Z}=\mathcal{N}_0^{\mathrm{(B1)},Z}+\mathcal{N}_M^{\mathrm{(B1)},Z}$ is used, following the same approach as in Section~\ref{tricksection}, with
\begin{align*}
    \abs{\mathcal{N}_0^{\mathrm{(B1)},Z}}&\lesssim\int_{k,\eta,\ell }\jb{k}\jb{k,\eta}^{\sigma_4}\abs{D_\eta^\alpha \widehat{W[P]}(t,k,\eta)}\abs{\jb{k}\jb{k,\eta}^{\sigma_4}-\jb{k-\ell }\jb{k-\ell ,\eta-\ell t}^{\sigma_4}}\\
&\hspace{3cm}\times\abs{\widehat{w}(\ell )}\abs{\ell }\abs{\widehat{\rho_Q}(t,\ell )}\abs{\eta-kt}\abs{D_\eta^\alpha\widehat{W[P]}(t,k-\ell ,\eta-\ell t)}\,d\ell dkd\eta,\\
\abs{\mathcal{N}_M^{\mathrm{(B1)},Z}}&\lesssim\sum_{\abs{\beta}\leq \aal-1}\int_{k,\eta,\ell }\jb{k}\jb{k,\eta}^{\sigma_4}\abs{D_\eta^\alpha \widehat{W[P]}(t,k,\eta)}\jb{k}\jb{k,\eta}^{\sigma_4}\\
&\hspace{3cm}\times\abs{\widehat{w}(\ell )}\abs{\ell }\jb{\hbar\ell}^{\aal-1}\abs{\widehat{\rho_Q}(t,\ell )}\abs{D_\eta^\beta\widehat{W[P]}(t,k-\ell ,\eta-\ell t)}\,d\ell dkd\eta.
\end{align*}
To control $\mathcal{N}_0^{\mathrm{(B1)},Z}$, the Mean Value Theorem yields
\begin{align*}&\abs{\jb{k}\jb{k,\eta}^{\sigma_4}-\jb{k-\ell }\jb{k-\ell ,\eta-\ell t}^{\sigma_4}}\vphantom{2_{A_A}}\\
    &\lesssim\jb{\ell }\jb{\ell ,\ell t}^{\sigma_4}+\jb{k-\ell }\abs{(\ell ,\ell t)}\jb{\ell ,\ell t}^{\sigma_4-1}+\jb{\ell }\jb{k-\ell ,\eta-\ell t}^{\sigma_4}+\jb{k-\ell }\abs{(\ell ,\ell t)}\jb{k-\ell ,\eta-\ell t}^{\sigma_4-1},\vphantom{2^{B^B}}
\end{align*}which splits the estimate into four contributions: $\abs{\mathcal{N}_0^{\mathrm{(B1)},Z}}\lesssim \mathcal{R}_{0,1}^{\mathrm{(B1)},Z}+\mathcal{R}_{0,2}^{\mathrm{(B1)},Z}+\mathcal{T}_{0,1}^{\mathrm{(B1)},Z}+\mathcal{T}_{0,2}^{\mathrm{(B1)},Z}$. These terms are analogous to those in Section~\ref{tricksection}, and are treated using the same estimates, except that the leading-order reaction term $\mathcal{R}_{0,1}^{\mathrm{(B1)},Z}$ yields one less power of $\jb{t}$. In conclusion, provided $\sigma_3 > d + 2$, the estimate
\begin{align*}
    \abs{\mathcal{N}_0^{\mathrm{(B1)},Z}}&\lesssim\frac{\sqrt{K_5}\ep}{\jb{t}^{d+1}}\norm{\jb{\nabla_z}W[P]}_{H^{\sigma_4}_M}\norm{\jb{t\nabla_z,\nabla_v}W[P]}_{H^{\sigma_4}_M}\\
    &\hspace{3cm}+\sqrt{K_3}\ep\jb{t}\norm{\jb{\nabla_z}W[P]}_{H^{\sigma_4}_M}\norm{\ak^{1/2}\jb{k,kt}^{\sigma_4}\widehat{\rho_Q}(t,k)}_{L^2_k}
\end{align*}follows. The same argument applies to $\mathcal{N}_M^{\mathrm{(B1)},Z}$.

\subsubsection{Conclusion of the spatial derivative}
Proceeding as in Section~\ref{concluvelocity}, it follows that
\begin{align*}
    \half\frac{d}{dt}\norm{\jb{\nabla_z}{W[P]}(t)}_{H^{\sigma_4}_M}^2&\lesssim\norm{\jb{\nabla_z}{W[P](t)}}_{H^{\sigma_4}_M}\norm{\jb{\hbar k}\ak^{1/2}\jb{k,kt}^{\sigma_4}\widehat{\rho_Q}(t,k)}_{L^2_{k}}\\
    &\hspace{0.25cm}+\sqrt{K_3}\ep\jb{t}\norm{\jb{\nabla_z}W[P](t)}_{H^{\sigma_4}_M}\norm{\jb{\hbar k}\ak^{1/2}\jb{k,kt}^{\sigma_4}\widehat{\rho_Q}(t,k)}_{L^2_k}\\
    &\hspace{0.25cm}+\frac{\sqrt{K_5}\ep}{\jb{t}^{d+1}}\norm{\jb{\nabla_z}W[P](t)}_{H^{\sigma_4}_M}\norm{\jb{t\nabla_z,\nabla_v}W[P](t)}_{H^{\sigma_4}_M}.
    \end{align*}
This implies \begin{align*}
    \frac{d}{dt}&\norm{\jb{\nabla_z}W[P](t)}_{H^{\sigma_4}_M}\\
    &\lesssim\br{1+\sqrt{K_3}\ep\jb{t}}\norm{\jb{\hbar k}\ak^{1/2}\jb{k,kt}^{\sigma_4}\widehat{\rho_Q}(t,k)}_{L^2_{k}}+\frac{\sqrt{K_5}\ep}{\jb{t}^{d+1}}\norm{\jb{t\nabla_z,\nabla_v}W[P](t)}_{H^{\sigma_4}_M}.
\end{align*}
Integrating in time and applying the bootstrap assumptions \eqref{B1} and \eqref{B2} yields
\begin{align*}
    &\norm{\jb{\nabla_z}W[P](t)}_{H^{\sigma_4}_M}-\norm{\jb{\nabla_z}W[\Qin]}_{H^{\sigma_4}_M}\\
    &\lesssim\int_0^t\br{1+\sqrt{K_3}\ep\jb{s}}\norm{\jb{\hbar k}\ak^{1/2}\jb{k,ks}^{\sigma_4}\widehat{\rho_Q}(s,k)}_{L^2_{k}}\,ds+\int_0^t\frac{\sqrt{K_5}\ep}{\jb{s}^{d+1}}\norm{\jb{s\nabla_z,\nabla_v}W[P](s)}_{H^{\sigma_4}_M}\,ds\\
    &\lesssim\br{\int_0^t\br{1+\sqrt{K_3}\ep\jb{s}}^2\,ds}^{1/2}\norm{\jb{\hbar k}\ak^{1/2}\jb{k,ks}^{\sigma_4}\widehat{\rho_Q}(s,k)}_{L^2_{t}([0,T^*],L^2_k)}+\sqrt{K_1K_5}\ep^2\int_0^t\frac{\jb{s}^{d-1/2}}{\jb{s}^{d+1}}\,ds\\
    &\lesssim\sqrt{K_2}\ep\br{1+\sqrt{K_3}\ep\jb{t}}\jb{t}^{1/2}+\sqrt{K_1K_5}\ep^2.
\end{align*}
Hence, there exists a constant $\widetilde K_Z>0$, depending only on the fixed parameters in Proposition~\ref{bootstrapprop} and independent of $\hbar$, $T^*$, $\ep$, $\Qin$ and the bootstrap constants $K_i$, such that
\begin{align}
      \hbar^{-d}\sum_{\aal\leq M}\norm{\bm{\xi}^\alpha\jb{\bm{\nabla}_x,\bm{\nabla}_\xi}^{\sigma_4}\jb{\bm{\nabla}_x}P(t)}_{L^2_{x,y}}^2\leq\widetilde K_Z\ep^2\sbr{\sqrt{K_2}+\ep\br{\sqrt{K_2K_3}+\sqrt{K_1K_5}}}^2\jb{t}^{3}.\label{B1zconc}
\end{align}

\subsubsection{Conclusion of bootstrap improvement \eqref{I1}}
Choose $K_1:=2\max\set{\widetilde K_V,\widetilde K_Z}\br{\sqrt{K_2}+1}^2$ and impose $\ep\leq (\sqrt{K_2K_3}+\sqrt{K_1K_5})^{-1}$. Then, combining the bounds \eqref{B1vconc} and \eqref{B1zconc}, it follows that
\begin{align*}
     \hbar^{-d}\sum_{\aal\leq M}\norm{\bm{\xi}^\alpha\jb{\bm{\nabla}_x,\bm{\nabla}_\xi}^{\sigma_4}\jb{\bm{\nabla}_\xi}P(t)}_{L^2_{x,y}}^2&\leq \frac{1}{2}K_1\ep^2\jb{t}^5
\intertext{and}
     \hbar^{-d}\sum_{\aal\leq M}\norm{\bm{\xi}^\alpha\jb{\bm{\nabla}_x,\bm{\nabla}_\xi}^{\sigma_4}\jb{\bm{\nabla}_x}P(t)}_{L^2_{x,y}}^2&\leq \frac{1}{2} K_1\ep^2\jb{t}^3.
\end{align*}Therefore, for $d\geq 3$, the bound
\begin{align*}
    \hbar^{-d}\sum_{\aal\leq M}\norm{\bm{\xi}^\alpha\jb{\bm{\nabla}_x,\bm{\nabla}_\xi}^{\sigma_4}\jb{t\bm{\nabla}_x,\bm{\nabla}_\xi}P(t)}_{L^2_{x,y}}^2\leq 2K_1\jb{t}^{5}\ep^2\leq 2K_1\jb{t}^{2d-1}\ep^2
\end{align*}holds, completing the proof of \eqref{I1}.

\subsection{Bootstrap improvement \eqref{I3}}\label{sectionb3}
This section establishes the bootstrap improvement \eqref{I3} by employing the equivalence
\begin{align*}
    \sum_{\aal\leq M}\norm{\jb{\bm{\nabla}_x,\bm{\nabla}_\xi}^{\sigma_3}\abs{\bm{\nabla}_x}^\delta \br{\bm{\xi}^\alpha P(t)}}_{\mathcal L^2}^2=\sum_{\aal\leq M}\norm{\abs{\nabla_z}^\delta\jb{\nabla_z,\nabla_v}^{\sigma_3}\br{v^\alpha W[P]}}_{L^2}^2.
\end{align*}
Fix a multi-index \(\aal \leq M\). Differentiating one of the terms on the right-hand side yields
\begin{align*}
    (2\pi)^{2d}\frac{1}{2}\frac{d}{dt}&\norm{\abs{\nabla_z}^\delta\jb{\nabla_z,\nabla_v}^{\sigma_3}(v^\alpha W[P])}_{L^2}^2\\
    &=\Re\int_{k,\eta}\abs{k}^\delta\jb{k,\eta}^{\sigma_3} \overline{D_\eta^\alpha \widehat{W[P]}(t,k,\eta)}\abs{k}^\delta\jb{k,\eta}^{\sigma_3} D_\eta^\alpha\widehat{W[{\mathrm{L}_P}]}(t,k,\eta)\,d\eta dk\\
    &\hspace{0.3cm}+\Re \int_{k,\eta}\abs{k}^\delta\jb{k,\eta}^{\sigma_3} \overline{D_\eta^\alpha \widehat{W[P]}(t,k,\eta)}\abs{k}^\delta\jb{k,\eta}^{\sigma_3} D_\eta^\alpha\widehat{W[{\mathrm{N}_P}]}(t,k,\eta)\,d\eta dk,
\end{align*}where $\widehat{W[{\mathrm{L}_P}]}(t,k,\eta)$ and $\widehat{W[{\mathrm{N}_P}]}(t,k,\eta)$ are defined in \eqref{WL}. Denote the linear and nonlinear contributions by \(\mathcal{L}^{\mathrm{(B3)}}\) and \(\mathcal{N}^{\mathrm{(B3)}}\), respectively. 

To estimate $\mathcal{L}^{\mathrm{(B3)}}$, Lemma~\ref{etalemma} is applied to extract the leading-order contribution from the remainder terms. This yields $\abs{\mathcal{L}^{\mathrm{(B3)}}}\lesssim\mathcal{L}_0^{\mathrm{(B3)}} + \mathcal{L}_M^{\mathrm{(B3)}}$. For $\mathcal{L}_0^{\mathrm{(B3)}}$, the inequality $\jb{k,\eta}^{\sigma_3}\lesssim\jb{k,kt}^{\sigma_3}\jb{\eta-kt}^{\sigma_3}$ is used, along with the estimate\begin{align*}
    \jb{k,kt}^{\sigma_3}\ak^{1+\delta}=\ak^{1/2}\jb{k,kt}^{\sigma_4}\frac{\ak^{1/2+\delta}}{\jb{k,kt}^{\sigma_4-\sigma_3}}\leq \frac{\ak^{1/2}\jb{k,kt}^{\sigma_4}}{\jb{t}^{1/2+\delta}} ,
\end{align*}valid under the assumption $\sigma_4-\sigma_3\geq1/2+\delta$. Consequently, 
\begin{align*}
    \mathcal{L}_0^{\mathrm{(B3)}}&\lesssim\frac{1}{\jb{t}^{1/2+\delta}}\int_{k,\eta}\abs{k}^\delta\jb{k,\eta}^{\sigma_3} \abs{D_\eta^\alpha \widehat{W[P]}(t,k,\eta)}\abs{k}^{1/2}\jb{k,kt}^{\sigma_4} \abs{\widehat{\rho_Q}(t,k)}\\
    &\hspace{5.5cm}\times\jb{\eta-kt}^{\sigma_3+1}\abs{D^\alpha_\eta\widehat g(\eta-kt)}\,d\eta dk\\
    &\lesssim\frac{1}{\jb{t}^{1/2+\delta}}\norm{\abs{\nabla_z}^\delta\jb{\nabla_z,\nabla_v}^{\sigma_3}(v^\alpha W[P])}_{L^2}\norm{\abs{k}^{1/2}\jb{k,kt}^{\sigma_4} \widehat{\rho_Q}(t,k)}_{L^2_k}\norm{g}_{H^{\sigma_3+1}_M},
\end{align*}
where Cauchy–Schwarz is applied in $\eta$ and then $k$. Similarly, under the assumption \eqref{wassumption} on $w$,
\begin{align*}
    \mathcal{L}_M^{\mathrm{(B3)}}\lesssim\frac{1}{\jb{t}^{1/2+\delta}}\norm{\abs{\nabla_z}^\delta\jb{\nabla_z,\nabla_v}^{\sigma_3}(v^\alpha W[P])}_{L^2}\norm{\jb{\hbar k}\abs{k}^{1/2}\jb{k,kt}^{\sigma_4} \widehat{\rho_Q}(t,k)}_{L^2_k}\norm{g}_{H^{\sigma_3}_M}.
\end{align*}

To estimate the nonlinear term $\mathcal{N}^{\mathrm{(B3)}}$, apply Lemma~\ref{etalemma} to decompose it as $\mathcal{N}^{\mathrm{(B3)}} = \mathcal{N}_0^{\mathrm{(B3)}} + \mathcal{N}_M^{\mathrm{(B3)}}$, where $\mathcal{N}_0^{\mathrm{(B3)}}$ captures the contribution with the highest-order derivatives. As in Section~\ref{tricksection}, a cancellation is used to isolate a difference of weights. Specifically, subtracting an integral that evaluates to zero allows the replacement\[\ak^\delta\jb{k,\eta}^{\sigma_3}\mapsto \ak^\delta\jb{k,\eta}^{\sigma_3}-\abs{k-\ell}^\delta\jb{k-\ell,\eta-\ell t}^{\sigma_3},\]which introduces a commutator structure and leads to additional decay. The resulting term satisfies
\begin{align*}
    \abs{\mathcal{N}_0^{\mathrm{(B3)}}}&\lesssim\int_{k,\eta,\ell }\abs{k}^\delta\jb{k,\eta}^{\sigma_3}\abs{D_\eta^\alpha \widehat{W[P]}(t,k,\eta)}\abs{\abs{k}^\delta\jb{k,\eta}^{\sigma_3}-\abs{k-\ell }^\delta\jb{k-\ell ,\eta-\ell t}^{\sigma_3}}\\
&\hspace{1cm}\times\abs{\widehat{w}(\ell )}\abs{\ell }\abs{\widehat{\rho_Q}(t,\ell )}\abs{\eta-kt}\abs{D_\eta^\alpha\widehat{W[P]}(t,k-\ell ,\eta-\ell t)}\,d\ell dkd\eta,
\intertext{while the lower-order term satisfies}\abs{\mathcal{N}_M^{\mathrm{(B3)}}}&\lesssim\sum_{\abs{\beta}\leq \aal-1}\int_{k,\eta,\ell }\abs{k}^\delta\jb{k,\eta}^{\sigma_3}\abs{D_\eta^\alpha \widehat{W[P]}(t,k,\eta)}\\
&\hspace{1cm}\times\abs{k}^\delta\jb{k,\eta}^{\sigma_3}\abs{\widehat{w}(\ell )}\abs{\ell }\jb{\hbar \ell }^{\aal-1}\abs{\widehat{\rho_Q}(t,\ell )}\abs{D_\eta^\beta\widehat{W[P]}(t,k-\ell ,\eta-\ell t)}\,d\ell dkd\eta.
\end{align*}

To estimate $\abs{\mathcal{N}_0^{\mathrm{(B3)}}}$, apply the Mean Value Theorem to the difference of weights:
\begin{align*}
    &\hspace{-0.4cm}\abs{\abs{k}^\delta\jb{k,\eta}^{\sigma_3}-\abs{k-\ell }^\delta\jb{k-\ell ,\eta-\ell t}^{\sigma_3}}\\
    &\hspace{0.4cm}\leq \abs{\ak^\delta-\abs{k-\ell }^\delta}\jb{k,\eta}^{\sigma_3}+\abs{k-\ell }^\delta\abs{\jb{k,\eta}^{\sigma_3}-\jb{k-\ell ,\eta-\ell t}^{\sigma_3}}\vphantom{\int}\\
    &\hspace{0.4cm}\lesssim\abs{\ell }^\delta\jb{k,\eta}^{\sigma_3}+\abs{k-\ell }^\delta\abs{(\ell ,\ell t)}\sup_{\abs{\ell '}\leq \abs{\ell }}\jb{k-\ell ',\eta-\ell 't}^{\sigma_3-1}\\
    &\hspace{0.4cm}\lesssim(\abs{\ell }^\delta+\abs{k-\ell }^\delta)\jb{\ell ,\ell t}^{\sigma_3}+\abs{\ell }^\delta\jb{k-\ell ,\eta-\ell t}^{\sigma_3}+\abs{k-\ell }^\delta\abs{(\ell ,\ell t)}\jb{k-\ell ,\eta-\ell t}^{\sigma_3-1}.
\end{align*}
This gives three contributions $|\mathcal{N}_0^{\mathrm{(B3)}}|\,\lesssim \mathcal{R}_0^{\mathrm{(B3)}}+\mathcal{T}_{0,1}^{\mathrm{(B3)}}+\mathcal{T}_{0,2}^{\mathrm{(B3)}}$. For the reaction term, note that $\abs{\eta-kt}\leq \abs{\eta-\ell t}+t\abs{k-\ell }$, which motivates decomposing $\mathcal{R}_0^{\mathrm{(B3)}}$ into four parts according to whether the prefactor contains $\abs{\ell}^\delta$ or $\abs{k-\ell}^{\delta}$, and whether $\abs{\eta-kt}$ is controlled by $\abs{\eta-\ell t}$ or $t\abs{k-\ell}:$
\begin{align*}
    \mathcal{R}_0^{\mathrm{(B3)}}&\lesssim\int_{k,\eta,\ell }\abs{k}^\delta\jb{k,\eta}^{\sigma_3}\abs{D_\eta^\alpha \widehat{W[P]}(t,k,\eta)}\abs{\ell }^\delta\jb{\ell ,\ell t}^{\sigma_3}\\
&\hspace{2cm}\times\abs{\widehat{w}(\ell )}\abs{\ell }\abs{\widehat{\rho_Q}(t,\ell )}\br{\abs{\eta-\ell t}+t\abs{k-\ell }}\abs{D_\eta^\alpha\widehat{W[P]}(t,k-\ell ,\eta-\ell t)}\,d\ell dkd\eta\\
&\hspace{0.25cm}+\int_{k,\eta,\ell }\abs{k}^\delta\jb{k,\eta}^{\sigma_3}\abs{D_\eta^\alpha \widehat{W[P]}(t,k,\eta)}\abs{k-\ell }^\delta\jb{\ell ,\ell t}^{\sigma_3}\\
&\hspace{2cm}\times\abs{\widehat{w}(\ell )}\abs{\ell }\abs{\widehat{\rho_Q}(t,\ell )}\br{\abs{\eta-\ell t}+t\abs{k-\ell }}\abs{D_\eta^\alpha\widehat{W[P]}(t,k-\ell ,\eta-\ell t)}\,d\ell dkd\eta\\
&=\mathcal{R}_{1;V}^{\mathrm{(B3)}}+\mathcal{R}_{1;Z}^{\mathrm{(B3)}}+\mathcal{R}_{2;V}^{\mathrm{(B3)}}+\mathcal{R}_{2;Z}^{\mathrm{(B3)}}, 
\end{align*}where the subscripts $V$ and $Z$ indicate association with the velocity variable $v$ (terms involving $\abs{\eta-\ell t}$) and the spatial variable $z$ (terms involving $t\abs{k-\ell}$), respectively. These four terms are treated separately, beginning with $\mathcal{R}_{1;V}^{\mathrm{(B3)}}$, which is the most delicate: it alone lacks a factor of $\abs{k-\ell}$ in the integrand, so it cannot be controlled directly through the \eqref{B3} norm, and \eqref{B2} is used in its place. The starting point is the algebraic bound
\begin{align*}
    \abs{\ell }^{1+\delta}\jb{\ell ,\ell t}^{\sigma_3}=\abs{\ell }^{\half}\jb{\ell ,\ell t}^{\sigma_4}\frac{\abs{\ell }^{1/2+\delta}}{\jb{\ell ,\ell t}^{\sigma_4-\sigma_3}}\leq\abs{\ell }^{\half}\jb{\ell ,\ell t}^{\sigma_4}\frac{\abs{\ell }^{1/2+\delta}}{\abs{(\ell ,\ell t)}^{1/2+\delta}}=\frac{\abs{\ell }^{1/2}\jb{\ell ,\ell t}^{\sigma_4}}{\jb{t}^{1/2+\delta}},
\end{align*}
valid provided $\sigma_4-\sigma_3\geq 1/2+\delta$, which replaces the density weight $\abs{\ell}^{1+\delta}\jb{\ell,\ell t}^{\sigma_3}$ by the weight $\abs{\ell}^{1/2}\jb{\ell,\ell t}^{\sigma_4}$ governed by \eqref{B2}, and extracts the decay factor $\jb{t}^{-(1/2+\delta)}$. Applying Lemma~\ref{Lemma2.9}\ref{Lemma2.9b} with the $L^1$ norm taken in $(k,\eta)\mapsto\eta D^\alpha_\eta \widehat{W}[P](t,k,\eta)$, and noting that the weight $\abs{k}^{-2\delta}\jb{k}^{-2(\sigma_3-1)}$ is integrable when $d\geq 3$, then yields
\begin{align*}
    \mathcal{R}_{1;V}^{\mathrm{(B3)}}&\lesssim\frac{1}{\jb{t}^{1/2+\delta}}\norm{\abs{\nabla_z}^\delta\jb{\nabla_z,\nabla_v}^{\sigma_3}(v^\alpha W[P])}_{L^2}\norm{\ak^{1/2}\jb{k,kt}^{\sigma_4}\widehat{\rho_Q}(t,k)}_{L^2_k}\\
    &\hspace{7cm}\times\int_k\norm{\jb{\eta}D^\alpha_\eta \widehat{W}[P](t,k,\eta)}_{L^2_\eta}\,dk\\
    &\lesssim\frac{1}{\jb{t}^{1/2+\delta}}\norm{\abs{\nabla_z}^\delta\jb{\nabla_z,\nabla_v}^{\sigma_3}(v^\alpha W[P])}_{L^2}^2\norm{\ak^{1/2}\jb{k,kt}^{\sigma_4}\widehat{\rho_Q}(t,k)}_{L^2_k}.
\end{align*}

Each of the remaining terms $\mathcal{R}_{1;Z}^{\mathrm{(B3)}}$, $\mathcal{R}_{2;V}^{\mathrm{(B3)}}$, $\mathcal{R}_{2;Z}^{\mathrm{(B3)}}$ contains a factor of $\abs{k-\ell }$ in its integrand, allowing the use of Lemma~\ref{Lemma2.9}~\ref{Lemma2.9a} with the $L^1$ term applied to the density, incorporating the appropriate weights. For $\mathcal{R}_{1;Z}^{\mathrm{(B3)}}$, observe that $\abs{k-\ell }\leq\abs{k-\ell }^{\delta}\jb{k-\ell ,\eta-\ell t}^{\sigma_3}$, so
\begin{align*}
    \mathcal{R}_{1;Z}^{\mathrm{(B3)}}&\lesssim t\norm{\abs{\nabla_z}^\delta\jb{\nabla_z,\nabla_v}^{\sigma_3}(v^\alpha W[P])}_{L^2}^2\int_\ell \abs{\ell }^{1+\delta}\abs{\widehat w(\ell )}\jb{\ell ,\ell t}^{\sigma_3}\abs{\widehat{\rho_Q}(t,\ell )}\,d\ell \\
    &\lesssim t\norm{\abs{\nabla_z}^\delta\jb{\nabla_z,\nabla_v}^{\sigma_3}(v^\alpha W[P])}_{L^2}^2\norm{\ak^{1/2}\jb{k,kt}^{\sigma_4}\widehat{\rho_Q}(t,k)}_{L^2_k}\br{\int_\ell \frac{\abs{\ell }^{1+2\delta}}{\jb{\ell ,\ell t}^{2\sigma_4-2\sigma_3}}\,d\ell }^{1/2}\\
    &\lesssim\frac{1}{\jb{t}^{{d/2}-1/2+\delta}}\norm{\abs{\nabla_z}^\delta\jb{\nabla_z,\nabla_v}^{\sigma_3}(v^\alpha W[P])}_{L^2}^2\norm{\ak^{1/2}\jb{k,kt}^{\sigma_4}\widehat{\rho_Q}(t,k)}_{L^2_k}
\end{align*}provided that $\sigma_4-\sigma_3>d/2+1/2+\delta$.
The remaining terms $\mathcal{R}_{2;V}^{\mathrm{(B3)}}$, $\mathcal{R}_{2;Z}^{\mathrm{(B3)}}$ can be estimated in the same fashion, leading to\begin{align*}
    \mathcal{R}_{2;V}^{\mathrm{(B3)}}+\mathcal{R}_{2;Z}^{\mathrm{(B3)}}\lesssim\frac{1}{\jb{t}}\norm{\abs{\nabla_z}^\delta\jb{\nabla_z,\nabla_v}^{\sigma_3}(v^\alpha W[P])}_{L^2}^2\norm{\ak^{1/2}\jb{k,kt}^{\sigma_4}\widehat{\rho_Q}(t,k)}_{L^2_k}.
\end{align*}

Next, the transport terms $\mathcal{T}_{0,1}^{\mathrm{(B3)}}$ and $\mathcal{T}_{0,2}^{\mathrm{(B3)}}$ are addressed. For $\mathcal{T}_{0,1}^{\mathrm{(B3)}}$, the additional factor of $\abs{\eta-kt}$ necessitates the use of the bootstrap estimate \eqref{B1}, which grows like $\jb{t}^{d-1/2}$. Since no time growth is allowed in \eqref{I3}, decay must instead be extracted from $\abs{\widehat{\rho_Q}(t,\ell )}\lesssim\sqrt{K_5}\ep\jb{\ell ,\ell t}^{-\sigma_1}$ to control this term. Using the bound $\abs{\eta-kt}\lesssim\jb{t(k-\ell ),\eta-\ell t}$, the term can be estimated as follows:
\begin{align*}
\mathcal{T}_{0,1}^{\mathrm{(B3)}}&\lesssim\int_{k,\eta,\ell }\abs{k}^\delta\jb{k,\eta}^{\sigma_3}\abs{D_\eta^\alpha \widehat{W[P]}(t,k,\eta)}\abs{\ell }^{1+\delta}\abs{\widehat{\rho_Q}(t,\ell )}\\
&\hspace{2.8cm}\times\jb{t(k-\ell ),\eta-\ell t}\jb{k-\ell ,\eta-\ell t}^{\sigma_3}\abs{D_\eta^\alpha\widehat{W[P]}(t,k-\ell ,\eta-\ell t)}\,d\ell dkd\eta\\
&\lesssim\norm{\abs{\nabla_z}^\delta\jb{\nabla_z,\nabla_v}^{\sigma_3}(v^\alpha W[P])}_{L^2}\norm{\jb{t\nabla_z,\nabla_v}W[P]}_{H^{\sigma_4}_M}\int_\ell \abs{\ell }^{1+\delta}\abs{\widehat{\rho_Q}(t,\ell )}\,d\ell \\
&\lesssim\frac{\sqrt{K_5}\ep}{\jb{t}^{d+1+\delta}}\norm{\abs{\nabla_z}^\delta\jb{\nabla_z,\nabla_v}^{\sigma_3}(v^\alpha W[P])}_{L^2}\norm{\jb{t\nabla_z,\nabla_v}W[P]}_{H^{\sigma_4}_M}.
\end{align*}
For $\mathcal{T}_{0,2}^{\mathrm{(B3)}}$, the additional flexibility allows for the expression $\abs{\eta-kt}\lesssim\jb{t}\jb{k-\ell ,\eta-\ell t}$ to be used, leading to control via the \eqref{B3} norm:
\begin{align*}
    \mathcal{T}_{0,2}^{\mathrm{(B3)}}&\lesssim\jb{t}\int_{k,\eta,\ell }\abs{k}^\delta\jb{k,\eta}^{\sigma_3}\abs{D_\eta^\alpha \widehat{W[P]}(t,k,\eta)}\abs{\ell }\abs{(\ell ,\ell t)}\abs{\widehat{\rho_Q}(t,\ell )}\\
    &\hspace{3.3cm}\times\jb{k-\ell ,\eta-\ell t}^{\sigma_3}\abs{k-\ell }^{\delta}\abs{D_\eta^\alpha\widehat{W[P]}(t,k-\ell ,\eta-\ell t)}\,d\ell dkd\eta\\
    &\lesssim\jb{t}\norm{\abs{\nabla_z}^\delta\jb{\nabla_z,\nabla_v}^{\sigma_3}(v^\alpha W[P])}_{L^2}^2\int_\ell \abs{\ell }\abs{(\ell ,\ell t)}\abs{\widehat{\rho_Q}(t,\ell )}\,d\ell \\
    &\lesssim\frac{\sqrt{K_5}\ep}{\jb{t}^d}\norm{\abs{\nabla_z}^\delta\jb{\nabla_z,\nabla_v}^{\sigma_3}(v^\alpha W[P])}_{L^2}^2.
\end{align*}This estimate is valid provided that $\sigma_1>d+2$.

For the remainder term $\mathcal{N}_M^{\mathrm{(B3)}}$, the triangle inequality yields $\abs{k}^\delta\leq \abs{\ell }^\delta+\abs{k-\ell }^{\delta}$ and $\jb{k,\eta}^{\sigma_3}\lesssim\jb{\ell ,\ell t}^{\sigma_3}+\jb{k-\ell ,\eta-\ell t}^{\sigma_3}$. This gives the bound
\begin{align*}
    \ak^\delta\jb{k,\eta}^{\sigma_3}\lesssim(\abs{\ell }^{\delta}+\abs{k-\ell }^{\delta})\jb{\ell,\ell t}^{\sigma_3}+\abs{\ell }^\delta\jb{k-\ell ,\eta-\ell t}^{\sigma_3}+\abs{k-\ell }^{\delta}\jb{k-\ell ,\eta-\ell t}^{\sigma_3}.
\end{align*}Consequently, three terms are obtained such that $\abs{\mathcal{N}_M^{\mathrm{(B3)}}}\lesssim \mathcal{R}_M^{\mathrm{(B3)}}+\mathcal{T}_{M,1}^{\mathrm{(B3)}}+\mathcal{T}_{M,2}^{\mathrm{(B3)}}$.
The reaction term is controlled in the same way as the leading-order contribution $\mathcal{R}_0^{\mathrm{(B3)}}$ above:
\begin{align*}
    \mathcal{R}_M^{\mathrm{(B3)}}\lesssim\frac{1}{\jb{t}^{1/2+\delta}}\norm{\abs{\nabla_z}^\delta\jb{\nabla_z,\nabla_v}^{\sigma_3}(v^\alpha W[P])}_{L^2}\norm{\jb{\hbar k}\ak^{1/2}\jb{k,kt}^{\sigma_4}\widehat{\rho_Q}(t,k)}_{L^2_k}\norm{\abs{\nabla_z}^\delta W[P]}_{H^{\sigma_3}_M}.
\end{align*}
The transport term $\mathcal{T}_{M,1}^{\mathrm{(B3)}}$ is treated similarly to $\mathcal{T}_{0,1}^{\mathrm{(B3)}}$, leading to
\begin{align*}
    \mathcal{T}_{M,1}^{\mathrm{(B3)}}\lesssim\frac{\sqrt{K_5}\ep}{\jb{t}^{d+1+\delta}}\norm{\abs{\nabla_z}^\delta\jb{\nabla_z,\nabla_v}^{\sigma_3}(v^\alpha W[P])}_{L^2}\norm{\jb{t\nabla_z,\nabla_v}W[P]}_{H^{\sigma_4}_M}.
\end{align*}Finally, $\mathcal{T}_{M,2}^{\mathrm{(B3)}}$ is handled as in $\mathcal{T}_{0,2}^{\mathrm{(B3)}}$:
\begin{align*}
    \mathcal{T}_{M,2}^{\mathrm{(B3)}}\lesssim\frac{\sqrt{K_5}\ep}{\jb{t}^{d+1}}\norm{\abs{\nabla_z}^\delta\jb{\nabla_z,\nabla_v}^{\sigma_3}(v^\alpha W[P])}_{L^2}\norm{\abs{\nabla_z}^\delta W[P]}_{H^{\sigma_3}_M}.
\end{align*}

\subsubsection*{Conclusion of bootstrap improvement \eqref{I3}}
Combining the estimates derived above, it follows that for each multi-index $\aal\leq M$ and $d\geq 3$,\begin{align*}
    \frac{d}{dt}&\norm{\abs{\nabla_z}^\delta\jb{\nabla_z,\nabla_v}^{\sigma_3}(v^\alpha W[P](t))}_{L^2}\\
    &\hspace{1cm}\lesssim\frac{1}{\jb{t}^{1/2+\delta}}\norm{\jb{\hbar k}\abs{k}^{1/2}\jb{k,kt}^{\sigma_4} \widehat{\rho_Q}(t,k)}_{L^2_k}\br{1+\norm{\abs{\nabla_z}^\delta W[P](t)}_{H^{\sigma_3}_M}}\\
    &\hspace{1.25cm}+\frac{\sqrt{K_5}\ep}{\jb{t}^{d+1+\delta}}\norm{\jb{t\nabla_z,\nabla_v}W[P](t)}_{H^{\sigma_4}_M}+\frac{\sqrt{K_5}\ep}{\jb{t}^d}\norm{\abs{\nabla_z}^\delta W[P](t)}_{H^{\sigma_3}_M}.
\end{align*}
Integrating in time and applying the bootstrap assumptions \eqref{B1}, \eqref{B2}, and \eqref{B3} gives
\begin{align*}
&\norm{\abs{\nabla_z}^\delta\jb{\nabla_z,\nabla_v}^{\sigma_3}(v^\alpha W[P](t))}_{L^2}-\norm{\abs{\nabla_z}^\delta\jb{\nabla_z,\nabla_v}^{\sigma_3}(v^\alpha W[\Qin])}_{L^2}\\
&\lesssim(1+\sqrt{K_3}\ep)\int_0^t\frac{1}{\jb{s}^{{1/2}+\delta}}\norm{\jb{\hbar k}\abs{k}^{1/2}\jb{k,ks}^{\sigma_4} \widehat{\rho_Q}(s,k)}_{L^2_k}\,ds\\
&\hspace{0.25cm}+\sqrt{K_5}\ep\int_0^t\frac{1}{\jb{s}^{3/2+\delta}}\sup_{0\leq s\leq T^*}\br{\frac{\norm{\jb{s\nabla_z,\nabla_v}W[P](s)}_{H^{\sigma_4}_M}}{\jb{s}^{d-1/2}}}\,ds+\sqrt{K_3K_5}\ep^2\int_0^t\frac{1}{\jb{s}^d}\,ds\\
&\lesssim(1+\sqrt{K_3}\ep)\br{\int_0^t\frac{1}{\jb{s}^{1+2\delta}}\,ds}^{1/2}\norm{\jb{\hbar k}\abs{k}^{1/2}\jb{k,kt}^{\sigma_4} \widehat{\rho_Q}(t,k)}_{L^2_{t}([0,T^*],L^2_k)}+\sqrt{K_1K_5}\ep^2+\sqrt{K_3K_5}\ep^2\\
&\lesssim\ep\sbr{\sqrt{K_2}+\br{\sqrt{K_2K_3}+\sqrt{K_1K_5}+\sqrt{K_3K_5}}\ep}.
\end{align*}
Hence, there exists a positive constant $\widetilde K>0$, depending only on the fixed parameters in Proposition~\ref{bootstrapprop} and independent of $\hbar$, $T^*$, $\ep$, $\Qin$ and the bootstrap constants $K_i$, such that \begin{align*}
    \frac{1}{\hbar^d}\sum_{\aal\leq M}\norm{\jb{\bm{\nabla}_x,\bm{\nabla}_\xi}^{\sigma_3}\abs{\bm{\nabla}_x}^\delta \bm{\xi}^\alpha P(t)}_{L^2_{x,y}}^2\leq \widetilde K\ep^2\sbr{\sqrt{K_2}+\br{\sqrt{K_2K_3}+\sqrt{K_1K_5}+\sqrt{K_3K_5}}\ep}^2.
\end{align*}
Choosing $K_3 := \widetilde K (\sqrt{K_2} + 1)^2 / 2$ and $\ep \leq (\sqrt{K_2 K_3} + \sqrt{K_1 K_5} + \sqrt{K_3 K_5})^{-1}$ yields the desired bootstrap improvement \eqref{I3}.

\subsection{Bootstrap improvement \eqref{I5}}\label{sectionb5}
This section establishes the improvement of the bootstrap assumption \eqref{B5}: it suffices to show
\begin{align*}
    \jb{k,\eta}^{\sigma_1}\abs{\widehat{W[P]}(t,k,\eta)}\leq\sqrt{2K_5}\,\ep
\end{align*}
for all $t\in[0,T^*]$, provided $\ep$ is sufficiently small.

Using the decomposition \eqref{WL}, the triangle inequality yields
\begin{align*}
    &\jb{k,\eta}^{\sigma_1}\abs{\widehat{W[P]}(t,k,\eta)}\\
    &\leq\jb{k,\eta}^{\sigma_1}\abs{\widehat{W[\Qin]}(k,\eta)}+\jb{k,\eta}^{\sigma_1}\int_0^t\abs{\widehat{W[{\mathrm{L}_P}]}(s,k,\eta)}\,ds+\jb{k,\eta}^{\sigma_1}\int_0^t\abs{\widehat{W[{\mathrm{N}_P}]}(s,k,\eta)}\,ds\\
    &=:\mathcal{I}^{\mathrm{(B5)}}(k,\eta)+\mathcal{L}^{\mathrm{(B5)}}(k,\eta)+\mathcal{N}^{\mathrm{(B5)}}(k,\eta).
\end{align*} 

The initial term satisfies
\begin{align*}
    \mathcal{I}^{\mathrm{(B5)}}(k,\eta) = \jb{k,\eta}^{\sigma_1}\abs{\widehat{W[\Qin]}(k,\eta)} = \abs{\reallywidehat{W[\jb{\bm{\nabla}_x,\bm{\nabla}_\xi}^{\sigma_1}\Qin]}(k,\eta)} \lesssim \ep,
\end{align*}
by Lemma~\ref{FTofWignertransform} and the assumption on $\Qin$ in Theorem~\ref{maintheorem}.

To estimate the linear term, apply  $\jb{k,\eta}^{\sigma_1} \lesssim \jb{k,ks}^{\sigma_1} \jb{\eta-ks}^{\sigma_1}$, the bound $\abs{\widehat{w}(k)} \lesssim 1$, the bootstrap estimate \eqref{B4}, and Lemma~\ref{L2trace}:
\begin{align*}
    \mathcal{L}^{\mathrm{(B5)}}(k,\eta)&=\jb{k,\eta}^{\sigma_1}\int_0^t\abs{2\hbar^{-1}\widehat w(k)\widehat{\rho_Q}(s,k)\sin\br{\mytextfrac{\hbar}{2}k\cdot(\eta-ks)}\widehat g(\eta-ks)}\,ds\\
    &\lesssim\int_0^t\jb{k,ks}^{\sigma_1}\abs{\widehat w(k)}\abs{\widehat{\rho_Q}(s,k)}\ak\jb{\eta-ks}^{\sigma_1}\abs{\eta-ks}\abs{\widehat g(\eta-ks)}\,ds\\
    &\lesssim\norm{\ak^{1/2}\jb{k,kt}^{\sigma_2}\widehat{\rho_Q}(t,k)}_{L^2_t([0,T^*])}\br{\int_0^t\ak\jb{\eta-ks}^{2\sigma_1+2}\abs{\widehat g(\eta-ks)}^2\,ds}^{1/2}\\
    &\lesssim\sqrt{K_4}\ep\br{\int_0^\infty\jb{\eta-\mytextfrac{k}{\ak}s}^{2\sigma_1+2}\abs{\widehat g(\eta-\mytextfrac{k}{\ak}s)}^2\,ds}^{1/2}\\
    &\lesssim\sqrt{K_4}\ep\norm{g}_{H_{M}^{\sigma_1+1}}.
\end{align*}

For the nonlinear term, the inequality $\jb{k,\eta}^{\sigma_1} \lesssim \jb{\ell,\ell s}^{\sigma_1} + \jb{k-\ell,\eta-\ell s}^{\sigma_1}$ leads to the decomposition $\mathcal{N}^{\mathrm{(B5)}}(k,\eta)\lesssim \mathcal{R}^{\mathrm{(B5)}}(k,\eta)+ \mathcal{T}^{\mathrm{(B5)}}(k,\eta)$, with
\begin{align*}
    \mathcal{R}^{\mathrm{(B5)}}(k,\eta)&\lesssim\int_0^t\int_\ell \jb{\ell,\ell s}^{\sigma_1}\abs{\widehat{w}(\ell )}\abs{\widehat{\rho_Q}(s,\ell )}\abs{\ell}\abs{\eta-ks}\abs{\widehat{W[P]}(s,k-\ell ,\eta-\ell s)}\,d\ell\,ds,\\
    \mathcal{T}^{\mathrm{(B5)}}(k,\eta)&\lesssim\int_0^t\int_\ell \abs{\widehat{w}(\ell )}\abs{\widehat{\rho_Q}(s,\ell )}\abs{\ell}\abs{\eta-ks}\jb{k-\ell,\eta-\ell s}^{\sigma_1}\abs{\widehat{W[P]}(s,k-\ell ,\eta-\ell s)}\,d\ell\,ds.
\end{align*}

For the reaction term, use $\abs{\eta-ks} \lesssim \jb{s} \jb{k-\ell,\eta-\ell s}$ and Cauchy–Schwarz in time for fixed $\ell$:
\begin{align*}
    \mathcal{R}^{\mathrm{(B5)}}(k,\eta)&\lesssim\sup_{\ell\in\R^d}\norm{\abs{\ell}^{1/2}\jb{\ell,\ell t}^{\sigma_2}\widehat{\rho_Q}(t,\ell)}_{L^2_t([0,T^*])}\\
    &\hspace{0.3cm}\times\int_\ell\br{\int_0^t\jb{s}^2 \jb{\ell,\ell s}^{-2(\sigma_2-\sigma_1)}\abs{\ell}\jb{k-\ell,\eta-\ell s}^2\abs{\widehat{W[P]}(s,k-\ell ,\eta-\ell s)}^2\,ds}^{1/2} d\ell,\\
    &\lesssim\sqrt{K_4K_5}\ep^2\int_\ell\br{\int_0^t\jb{s}^2\frac{\abs{\ell}}{\jb{\ell,\ell s}^{2(\sigma_2-\sigma_1)}\jb{k-\ell,\eta-\ell s}^{2\sigma_1-2}}\,ds}^{1/2} d\ell.
\end{align*}
Setting $j_1:=\sigma_2-\sigma_1$ and $j_2:=\sigma_1-1$, and noting $\jb{\ell,\ell s}^{2(\sigma_2-\sigma_1)}\geq \jb{\ell s}^{2j_1}$ and $\jb{k-\ell,\eta-\ell s}^{2\sigma_1-2}\geq \jb{k-\ell}^{2j_2}$, the bound becomes
\begin{align*}\mathcal{R}^{\mathrm{(B5)}}(k,\eta)&\lesssim\sqrt{K_4K_5}\ep^2\int_\ell\frac{1}{\jb{k-\ell}^{j_2}}\br{\int_0^t\abs{\ell}\frac{\jb{s}^2}{\jb{\ell s}^{2j_1}}\,ds}^{1/2} d\ell\\
    &\lesssim\sqrt{K_4K_5}\ep^2\int_\ell\frac{1}{\jb{k-\ell}^{j_2}}\br{\int_0^\infty\frac{\jb{\frac{s}{\abs{\ell}}}^2}{\jb{s}^{2j_1}}\,ds}^{1/2} d\ell\\
&\lesssim\sqrt{K_4K_5}\ep^2\int_\ell\frac{\max\set{1,\abs{\ell}^{-1}}}{\jb{\ell}^{j_2}} d\ell,
\end{align*}which converges for $j_1>3/2$ (i.e.\ $\sigma_2-\sigma_1>3/2$) and $j_2>d$ (i.e.\ $\sigma_1>d+1$).

For the transport term, again use $\abs{\eta-ks} \lesssim \jb{s} \jb{k-\ell,\eta-\ell s}$ and apply Cauchy--Schwarz in $\ell$:
\begin{align*}
    \mathcal{T}^{\mathrm{(B5)}}(k,\eta)&\lesssim\int_0^t\jb{s}\int_\ell \abs{\widehat{\rho_Q}(s,\ell )}\abs{\ell}\jb{k-\ell,\eta-\ell s}^{\sigma_1+1}\abs{\widehat{W[P]}(s,k-\ell ,\eta-\ell s)}\,d\ell\,ds\\
    &\lesssim\sqrt{K_5}\ep\int_0^t\jb{s}\br{\int_\ell\frac{\abs{\ell}^2}{\jb{\ell,\ell s}^{2\sigma_1}\abs{k-\ell}^{2\delta}}\,d\ell}^{1/2}\\
    &\hspace{3cm}\times\br{\int_\ell \abs{k-\ell}^{2\delta}\jb{k-\ell,\eta-\ell s}^{2\sigma_1+2}\abs{\widehat{W[P]}(s,k-\ell ,\eta-\ell s)}^2\,d\ell}^{1/2} ds.
\end{align*}
Treating the $\ell$ integral as in~\eqref{beforebound} gives
\begin{align*}
   \sup_{k}\int_\ell\frac{\abs{\ell}^2}{\jb{\ell,\ell s}^{2\sigma_1}\abs{k-\ell}^{2\delta}}\,d\ell\lesssim\frac{1}{\jb{s}^{d+2-2\delta}},
\end{align*}
provided that $\sigma_1>d/2+1$. For the second $\ell$ integral, the change of variables $\ell\mapsto k-\ell$, the Sobolev embedding and bootstrap assumption \eqref{B3} yield
\begin{align*}
  &\norm{\abs{k-\ell}^{\delta}\jb{k-\ell,\eta-\ell s}^{\sigma_1+1}\widehat{W[P]}(s,k-\ell,\eta-\ell s)}_{L^\infty_\eta L^2_\ell}\\
  &\hspace{4cm}\lesssim\norm{\abs{k}^{\delta}\jb{k,\eta}^{\sigma_3}\widehat{W[P]}(s,k,\eta)}_{L^2_k L^\infty_\eta}\\
    &\hspace{4cm}\lesssim\sum_{\aal\leq M}\norm{D^\alpha_\eta\br{\abs{k}^{\delta}\jb{k,\eta}^{\sigma_3}\widehat{W[P]}(s,k,\eta)}}_{L^2_{k,\eta}}\\
    &\hspace{4cm}\lesssim\sqrt{K_3}\ep.
\end{align*}
Therefore,
\begin{align*}
    \mathcal{T}^{\mathrm{(B5)}}(k,\eta)\lesssim\sqrt{K_3K_5}\ep^2\int_0^t\frac{1}{\jb{s}^{{d/2}-\delta}}\,ds\lesssim\sqrt{K_3K_5}\ep^2.
\end{align*}

\subsubsection*{Conclusion of bootstrap improvement \eqref{I5}}
Under the bootstrap assumptions \eqref{B3}, \eqref{B4}, and \eqref{B5}, there exists a constant $\widetilde K>0$, depending only on the fixed parameters in Proposition~\ref{bootstrapprop} and independent of $\hbar$, $T^*$, $\ep$, $\Qin$ and the bootstrap constants $K_i$, such that 
\begin{align*}
      \jb{k,\eta}^{\sigma_1}\abs{\widehat{W[P]}(t,k,\eta)}\leq \widetilde K\br{\ep+\sqrt{K_4}\ep+\br{\sqrt{K_4K_5}+\sqrt{K_3K_5}}\ep^2}.
\end{align*}
Choose $\ep$ sufficiently small such that $\br{\sqrt{K_4K_5}+\sqrt{K_3K_5}}\ep\leq 1$. Then\begin{align*}
    \jb{k,\eta}^{\sigma_1}\abs{\widehat{W[P]}(t,k,\eta)}\leq \widetilde K\br{2+\sqrt{K_4}}\ep.
\end{align*}
Setting
\begin{align*}
    \sqrt{K_5} := \frac{\widetilde{K} (2 + \sqrt{K_4})}{\sqrt{2}},
\end{align*}
establishes the bootstrap improvement \eqref{I5}:
\begin{align*}
    \jb{k,\eta}^{\sigma_1} \abs{\widehat{W[P]}(t,k,\eta)}
    \leq \sqrt{2} \sqrt{K_5} \ep.
\end{align*}
This completes the proof of Proposition \ref{bootstrapprop}.
\subsection*{Acknowledgement}
Thanks are due to Cl\'ement Mouhot for suggesting the topic, as well as for numerous helpful discussions and ongoing encouragement throughout this work. Gratitude is also extended to Chiara Saffirio for careful reading and constructive feedback.
\printbibliography
\end{document}